\documentclass[14pt]{amsart}
\usepackage{amsmath}
\usepackage{amsthm}
\usepackage{amssymb}
\usepackage{amscd}
\usepackage{amsfonts}
\usepackage{amsbsy}
\usepackage{epsfig,afterpage}
\usepackage[dvips]{psfrag}
\usepackage{subfigure}
\usepackage{epsfig}
\usepackage{amsmath}
\usepackage{srcltx}
\usepackage{amsfonts}
\usepackage{amssymb}
\usepackage{psfrag}
\usepackage{verbatim}
\usepackage{graphicx}
\usepackage{amsmath}
\usepackage{amsthm}
\usepackage{amssymb}
\usepackage{amscd}
\usepackage{amsfonts}
\usepackage{amsbsy}
\usepackage{epsfig}
\usepackage[dvips]{psfrag}
\usepackage{multirow}
\usepackage{graphics}
\usepackage{epstopdf}
\usepackage{overpic}
\usepackage{float}

\usepackage{graphicx}
\usepackage{mathabx}

\def\e{\varepsilon}

\def\la{\lambda}

\newtheorem {theorem} {Theorem}%[section]
\newtheorem {proposition} [theorem]{Proposition}
\newtheorem {corollary} [theorem]{Corollary}
\newtheorem {lemma}  [theorem]{Lemma}
\newtheorem {example} [theorem]{Example}
\newtheorem {remark} [theorem]{Remark}

\newtheorem {definition} [theorem]{Definition}
\newcommand{\R}{\mathbb{R}}
\newcommand{\C}{\mathbb{C}}

\newcommand{\N}{\mathbb{N}}
\newcommand{\Z}{\mathbb{Z}}

\newcommand{\CO}{\ensuremath{\mathcal{O}}}

\usepackage{float}
\usepackage{tikz, wrapfig}
\tikzset{node distance=3cm, auto}
\allowdisplaybreaks

\begin{document}

\title[Piecewise Smooth Holomorphic Systems]
{Piecewise Smooth Holomorphic Systems}

\author[L. F. S. Gouveia, Gabriel Rondón and P. R. da Silva]
{Luiz F. S. Gouveia, Gabriel Rondón and Paulo R. da Silva}

\address{  S\~{a}o Paulo State University (Unesp), Institute of Biosciences, Humanities and
	Exact Sciences. Rua C. Colombo, 2265, CEP 15054--000. S. J. Rio Preto, S\~ao Paulo,
	Brazil.}

\email{paulo.r.silva@unesp.br}
\email{fernando.gouveia@unesp.br}
\email{garv202020@gmail.com}

\thanks{ }


\date{}
\dedicatory{}
\maketitle

	\begin{abstract}
The normal forms associated with holomorphic systems are well known in the literature. In this paper we are concerned about studying the piecewise smooth holomorphic systems (PWHS). Specifically, we classify the possible phase portraits of these systems from the known normal forms and the typical singularities of PWHS. Also, we are interested in understanding how the trajectories of the regularized system associated with the PWHS transits through the region of regularization. In addition, we know that holomorphic systems have no limit cycles, but piecewise smooth holomorphic systems do, so we provide conditions to ensure the existence of limit cycles of these systems. Additional conditions are provided to guarantee the stability and uniqueness of such limit cycles. Finally, we give some families of PWHS that have homoclinic orbits.
	\end{abstract}

\section{Introduction}
The holomorphic systems $\dot{z}=f(z)$ have interesting dynamical properties, for example, the fact that
these systems have no limit cycles and that they have a finite number of equilibrium points, which are isolated provided that $f$ is not identically null. Moreover, holomorphic polynomial systems reduce the number of parameters in the system. Although a polynomial system of degree
$n$ depends on $n^2+3n+2$ parameters, a polynomial holomorphic system depends only
on $2n+2$ parameters. Furthermore, the holomorphic functions has its interest in several areas of applied science, for example, in the study of fluid dynamics. In this context, it is possible to verify that the complex potential of the conjugate holomorphic system $\dot{z}=\overline{f(z)}$ is a primitive of $f(z)$. For more information see, for instance, \cite{BatGK,Mars,Conw}. 

In this paper, we are interested in the study of piecewise smooth holomorphic systems (PWHS),  
\begin{equation}\label{ch4:eq111}
\begin{aligned}
\left\{\begin{array}{l}
\dot{z}^{+}=f^{+}(z)=u_1+iv_1, \text{ when } \Re(z)> 0,\\[5pt]
\dot{z}^{-}=f^{-}(z)=u_2+iv_2,\text{ when }\Re(z)<0,
\end{array} \right.
\end{aligned}
\end{equation}
where  $z=x+iy$ and $f^{\pm}(z)$ are holomorphic functions defined in a domain $\mathcal{V}\subseteq\C$ and satisfying that
\begin{itemize}
	\item[(i)] $u_{1,2}=\operatorname{Re}(f^{\pm})$ and $v_{1,2}=\operatorname{Im}(f^{\pm})$  are continuous;
	\item[(ii)]  there exist the partial derivatives $(u_{1,2})_x,(u_{1,2})_y,(v_{1,2})_x,(v_{1,2})_y$ in $\mathcal{V},$ and 
	\item[(iii)]  the partial derivatives satisfy the Cauchy--Riemann equations 
	\[
	\begin{array}{rcl}
	(u_{1})_x =(v_{1})_y,& (u_{1})_y=-(v_{1})_x,\quad \forall z=x+iy\in\mathcal{V},\\
	(u_{2})_x =(v_{2})_y,& (u_{2})_y=-(v_{2})_x,\quad \forall z=x+iy\in\mathcal{V}.\\
	\end{array}
	\]  
\end{itemize}
We remark that the straight line $\Sigma=\{\Re(z)=0\}$ divides the plane in two half-planes 
$\Sigma^\pm$ given by $\{z :\Re(z)> 0\}$ and $\{z :\Re(z)< 0\}$, respectively. 
The trajectories on $\Sigma$ are defined following the Filippov convention. 

Throughout this article we use the normal forms associated with the holomorphic functions given in
\cite{BT} and \cite{GGJ2}, namely: $1$, $(a+ib)z$, $z^n$, $\frac{\gamma z^n}{1+z^{n-1}},$  and $\frac{1}{z^n}.$ For more details see Proposition \ref{GGJ}. A priori these normal forms depends on the
notion of conformal conjugation. 

One of the properties of the PWHS that we will prove here is that the sliding, sewing and tangential regions are preserved by conformal conjugation, see Theorem \ref{foldtofold}. In particular, Lemma \ref{foldtofold1} establishes that regular-fold singularities are preserved by conformal conjugation. We will use this last result to characterize the type of tangential contact of the holomorphic functions with $\Sigma$, which are conformally conjugated to some of the normal forms. For more information see Theorem \ref{car_nf}.%We say that the holomorphic functions $F$ and $G$ are \textit{$0$--conformally conjugated} if there exist a conformal map $\Phi:D(0,R)\rightarrow D(0,R)$ such that  $\Phi(0)=0$ and $\Phi(\varphi_F(t,z)) =\varphi_G(t,\Phi(z))$,  for any $z\in D(0,R)\setminus\{0\}$ and all $t $ for which the above expressions are well defined and the corresponding points are in $ D(0,R)$.

An interesting property is that the regularized vector field associated to \eqref{ch4:eq111} loses the property of being holomorphic, see Theorems \ref{teo:reg} and \ref{teoreg}. For that, we will use {\it the principle of identity of the analytic functions}, which states: given functions $f$ and $g$ analytic on a domain $D$ (open and connected subset of $\mathbb{C})$, if $f=g$ on some $S\subseteq D$, where $S$  has an accumulation point of $D$, then $f=g$ on $D$. 

Also, we are interested in regularizations of PWHS around visible regular-fold singularities. More specifically, using Theorem 1 of \cite{NR} and the normal forms associated with the homomorphic functions, we propose to understand how the trajectories of the regularized system transits through the region of regularization, see Theorem \ref{ta}.
\begin{figure}[h]
	\begin{center}
		\begin{overpic}[scale=0.35]{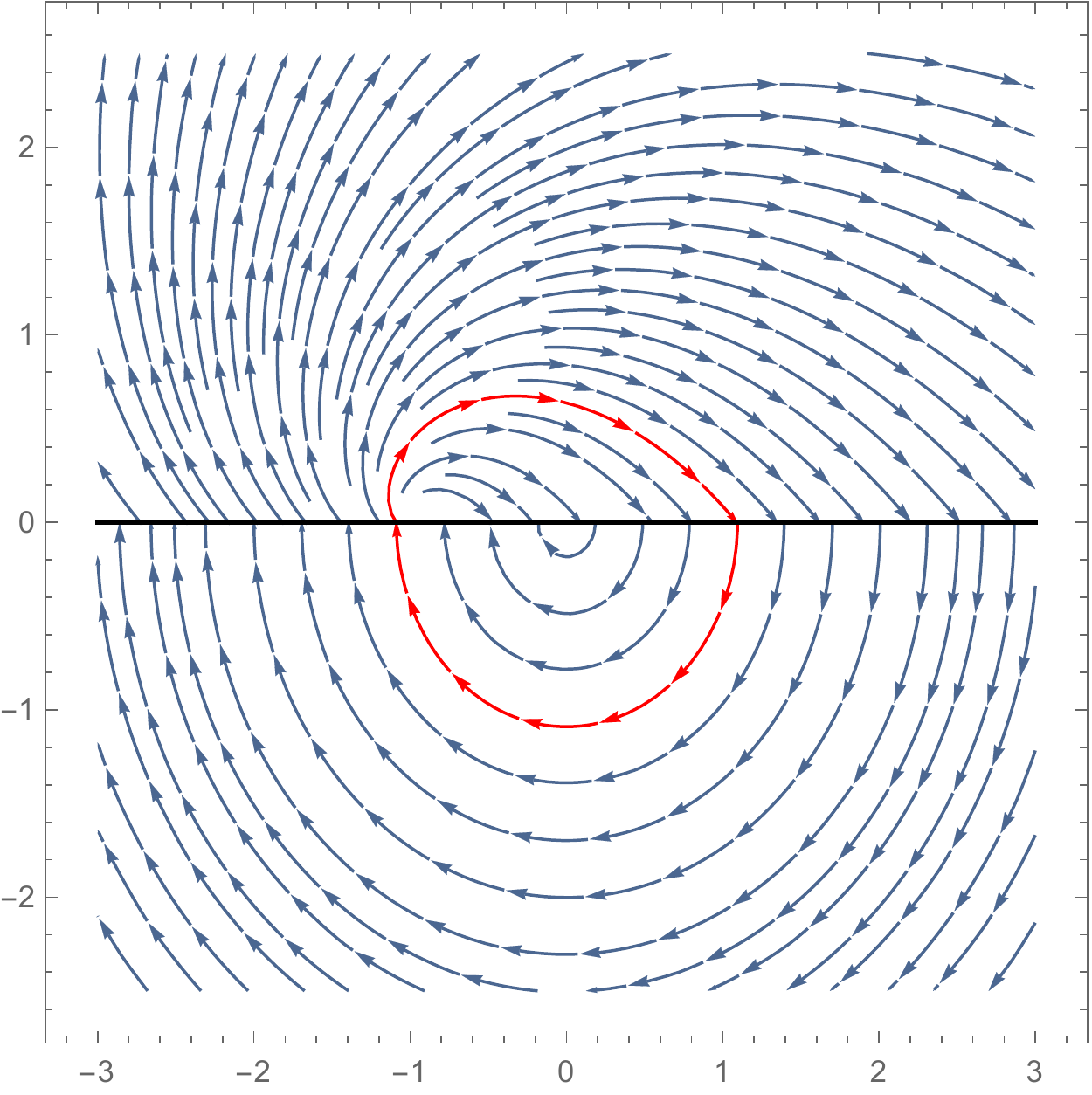}
		%\begin{overpic}[grid,tics=5,width=5cm]{limit_cycle_1.pdf}		
          \put(-7,27){$\Sigma^-$}
        \put(-7,68){$\Sigma^+$}
		\put(102,51){$\Sigma$}
		\end{overpic}
		\caption{Phase portrait of PWHS \eqref{ex_cyclefc}. The red  trajectory is the limit cycle of \eqref{ex_cyclefc}.}
	\label{limit_cycle_fc}
	\end{center}
	\end{figure}
	
In addition, we are concerned about studying the existence of limit cycles for the PWHS. One of the reasons for this study is the fact that the holomorphic systems have no limit cycles, for more information see, for instance, \cite{Ben,Bro,GGJ2,Oto1,Oto2,GXG,NeeKing,Sverdlove}. For that we will use the normal forms mentioned above and we will establish conditions for the existence of limit cycles, see Theorems \ref{cl1}, \ref{cl2}, and \ref{cl3}. Furthermore, additional conditions are provided to guarantee the stability and uniqueness of such limit cycles. In particular, Theorem \ref{cl1} establishes that the piecewise linear holomorphic systems whose equilibrium points are on manifold $\Sigma$ have at most one limit cycle. Also, Corollary \ref{cor:lc} establishes necessary and sufficient conditions for the existence of such a limit cycle. For example, if we consider the PWHS
\begin{equation}\label{ex_cyclefc}
\begin{aligned}
\left\{\begin{array}{l}
\dot{z}^{+}=(1-i)(z+1),\text{ when } \Im{(z)}>0, \\[5pt]
\dot{z}^{-}=-iz, \text{ when } \Im{(z)}< 0,
\end{array} \right.
\end{aligned}
\end{equation} then it has a unique unstable limit cycle (see Figure \ref{limit_cycle_fc}). 

In the context of piecewise linear systems in the real plane, depending on the number of zones generated by the discontinuity manifold $\Sigma$, the maximum number of limit cycles varies. For example, in \cite{MR1681463}, Freire et al. considered 2 zones divided by a straight line and proved that piecewise linear systems in the real plane have at most one limit cycle. However, when considering 3 zones (for example, the discontinuity manifold $\Sigma$ could be 2 parallel straight lines) it is possible to prove the existence of more than one limit cycle, for more details see, for instance, \cite{math8050755,MR3360760,MR3328261}.
	
We emphasize that the focus on the existence of limit cycles in PWHS is one of the main novelties of the present study. Some of the main challenges when working in this context is that building the first return map is a bit complicated, however, if we use the normal forms associated with the holomorphic functions in their polar form, it is much easier to work with. For the construction of the limit cycles we will use the symmetry of the polar equation of the orbits of $z^n$ and $\frac{1}{z^n}$ and the invariance of the rays of such normal forms.

Finally, we are going to use the invariant rays of the normal forms $z^n$ and $\frac{1}{z^n}$ to construct homoclinic orbits of the PWHS, for more details see Propositions \ref{propho1} and \ref{propho2}.
\subsection{Structure of the paper} In Section \ref{sec:preliminares}, we present some basic results on holomorphic functions that will be used throughout the paper. In Section \ref{sec:PWHS}, we use the normal forms given in Proposition \ref{GGJ} to classify the sliding, sewing, and tangential regions. For the tangential region, we study the type of tangential singularities existing in PWHS. In Section \ref{sec:reg},  we perform an analysis of the regularization of PWHS. In Section \ref{sec:limitcycles}, we establish conditions for the existence of limit cycles in PWHS. Finally, in Section \ref{sec:homoclinic_orbits} we give some families of PWHS that have homoclinic orbits.
\section{Preliminaries}\label{sec:preliminares}
In this section we establish some basic results that will be used throughout the paper.
\subsection{Holomorphic functions}
Let $F$ be a holomorphic function on a domain $\mathcal{V}\subseteq\C$.  Thus for any $z_0\in\mathcal{V}$ 
\begin{equation}
F(z)=A_0+A_1(z-z_0)+A_2(z-z_0)^2+...,\quad A_k=a_k+ib_k=\dfrac{F^{(k)}(z_0)}{k!}
\label{analF}
\end{equation}
for $z\in D(z_0,R_{z_0})\subseteq\mathcal{V}$ where $D(z_0,R_{z_0})$ is the largest possible $z_0$--centered disk contained in $\mathcal{V}.$ Unless a translation we can always assume that $z_0 = 0$.

If $F$ is holomorphic in a punctured disc $D(z_0,R)\setminus\{z_0\}$ and it is not derivable at $z_0$ we say that
$z_0$ is a singularity of $F$. In this case $F(z)$ is equal to its Laurent's series in $D(z_0,R)\setminus\{z_0\}$

\begin{equation} F(z)=\sum_{k=1}^{\infty}\dfrac{B_k}{(z-z_0)^k}+ \sum_{k=0}^{\infty}A_k(z-z_0)^k, \label{laurent}\end{equation}
where \[B_k=\dfrac{1}{2\pi i}\int_{C_{\e}}F(z)(z-z_0)^{k-1}dz,\quad A_k=\dfrac{1}{2\pi i}\int_{C_{\e}}\dfrac{F(z)}{(z-z_0)^{k+1}}dz\]
with $C_{\e}$ parameterized by $z(t)=\e e^{it}, \e\sim0$. \\

If $B_k\neq0$  for an infinite set of indices $k$ we say that $z_0$ is an \textit{essential singularity} and if
there exists $n \geq1$ such that $B_n \neq0$ and $B_k = 0$ for every $k>n$ then we say that $z_0$ is a \textit{pole of order n}. 
Moreover, $B_1$ is called \textit{residue} of $F$ at $z_0$ and it is denoted by $B_1 = \operatorname{res} (F, z_0)$.\\

Let  $F:D(0,R)\setminus\{0\}\rightarrow\C$ be a holomorphic function as \eqref{laurent} with $z_0=0, B_k=c_k+id_k$ and $A_k=a_k+ib_k$.
Consider the ordinary differential equation 
\begin{equation}\label{hde}
\dot{z}(t)=F(z(t)),\quad t\in\R.
\end{equation}
The solution of \eqref{hde} passing through $z\in D(0,R)\setminus\{0\}$  at 
	$t=0$ is denoted by $\varphi_F(t,z).$\\
We have 
\[F(z)=\sum_{k=1}^{\infty}\dfrac{c_k+id_k}{z^k}+ \sum_{k=0}^{\infty}(a_k+ib_k)z^k.\]
A direct calculation using Newton's binomial formula gives us
$z^k=(x+iy)^k=p_k+iq_k$
with $p_k$ and $q_k$ as in the table 
\begin{equation}
\begin{array}{llll}
\hline
&k	&p_k &q_k \\
\hline\\
&1	& x & y\\
\hline\\
&2	& x^2-y^2&2x y\\
\hline\\
&3	& x^3-3xy^2 & 3x^2y-y^3\\
\hline\\
&4	& x^4-6x^2y^2+y^4 & 4x^3y-4xy^3\\
\hline\\
&5	& x^5-10x^3y^2+5xy^4 & 5x^4y-10x^2y^3+y^5\\
\hline
&...&...&...\\

\end{array}
\label{Tpq}
\end{equation}
Thus $$(a_k+ib_k)z^k=(a_kp_k-b_kq_k)+i(b_kp_k+a_kq_k)$$
and $$\frac{c_k+id_k}{z^k}=\frac{(c_kp_k+d_kq_k)+i(d_kp_k-c_kq_k)}{(x^2+y^2)^k}.$$

Hence $\dot{x}=\operatorname{Re} (F(z))$ and $\dot{y}=\operatorname{Im}(F(z))$ must satisfy the following system
	\begin{equation}\left\{\begin{array}{ll}
	\dot{x}&= \displaystyle\sum_{k=1}^{\infty}\left(c_k\dfrac{p_k}{(x^2+y^2)^k}+d_k\dfrac{q_k}{(x^2+y^2)^k}\right)+a_0+
	\displaystyle\sum_{k=1}^{\infty}\left(a_kp_k-b_kq_k\right)\\
	\dot{y}&= \displaystyle\sum_{k=1}^{\infty}\left(d_k\dfrac{p_k}{(x^2+y^2)^k}-c_k\dfrac{q_k}{(x^2+y^2)^k}\right)+b_0+
	\displaystyle\sum_{k=1}^{\infty}\left(b_kp_k+a_kq_k\right)
	\end{array}
	\right.\label{hvf}
	\end{equation}
	with $p_k,q_k$ given in Table \eqref{Tpq}. We refer to system \eqref{hvf} as a holomorphic system. The coefficients $c_k,d_k$ are zero provided that $F$ is holomorphic at $0$.\\

\noindent\textbf {Remark.} If $F=u+iv$ is holomorphic in $D(0,R)\setminus\{0\}$ and it  is not identically null 
	then system \eqref{hvf} has a finite number of equilibrium points
	and all of them are isolated. In fact,  if there exists a sequence of distinct equilibria $(x_n,y_n)$ of \eqref{hvf}  
then the sequence $z_n=x_n+iy_n$ will be formed by zeros of $F$. Taking  $\overline{D(0,R)}$ if necessary, we can 
assume that $z_n$ admits  a convergent subsequence $z_{n_k}$. In this case $F$ is identically null  in a set that 
has an accumulation point. It follows from the principle of identity of analytic functions that $F \equiv 0$. 

\subsection{Conformally conjugate holomorphic functions}
In this section we introduce the notion of conformally conjugated holomorphic functions that allow us to obtain the normal forms for this class of functions. Before that, we need to define conformal mappings.
\begin{definition}
A map $\Phi:\C\to\C$ is called conformal if it preserves angles.
\end{definition}
%
%\begin{figure}[h]%\vspace{0.5cm}
%	\begin{overpic}[width=6cm]{angle_conformal}
%		%\begin{overpic}[grid,tics=10,width=10cm]{formNOR1}
%		\put(80,58){$\dot{z}$}
%		\put(75,73){$C$}
%			\put(72,46){$arg(\dot{z})$}
%			\put(103,28){$x$}
%				\put(47,76){$y$}	
%	\end{overpic}
%	\caption{ Complex curve $C$ and tangent vector $\dot{z}$.}
%	\label{angle_conformal}
%\end{figure}
In \cite{GAvila} was proved that the angle between 2 curves which intersect at a point $z_0$ is preserved by conformal maps (see Figure \ref{angle_conformal_3}).
\begin{figure}[h]%\vspace{0.5cm}
	\begin{overpic}[width=12cm]{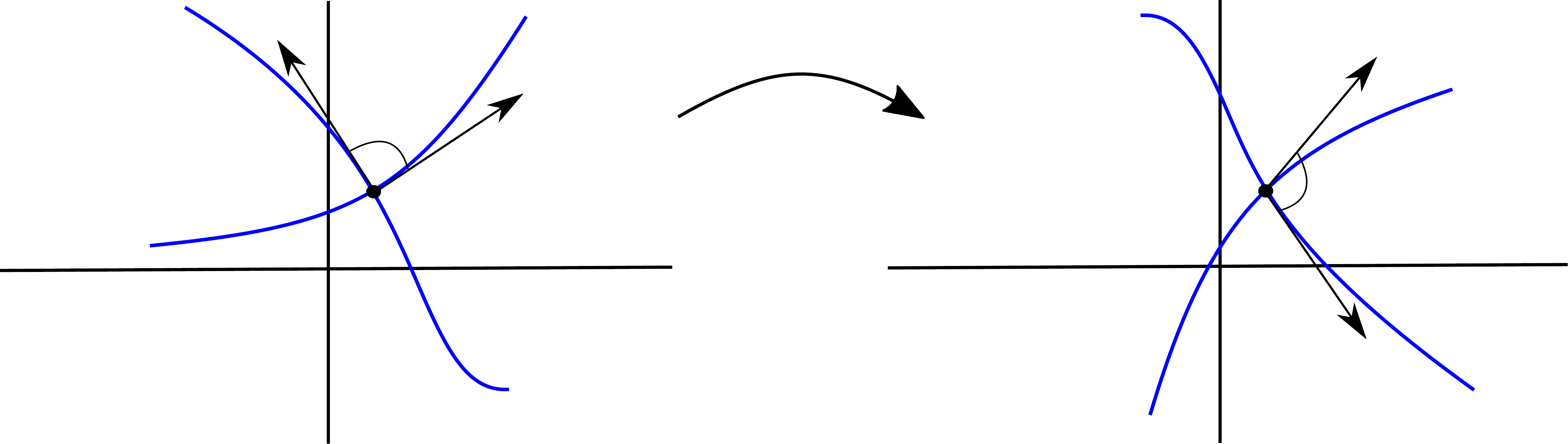}
	%	\begin{overpic}[grid,tics=5,width=12cm]{angle_conformal_3}
%		\put(74,57){$\dot{w}$}
%		\put(82,54){$\Gamma$}
%			\put(72,42){$arg(w)$}
%			\put(103,25){$u$}
%				\put(47,69){$v$}
\put(49,25){$\Phi$}
\put(24,21){\scriptsize $\theta$ \par}
		\put(94,22){$\Phi(C_1)$}
			\put(7.5,27){$C_2$}
			\put(101,11){$u$}
				\put(77,30){$v$}
				\put(85,16){\scriptsize $\theta$ \par}
		\put(34,27){$C_1$}
			\put(68,29){$\Phi(C_2)$}
			\put(44,11){$x$}
				\put(20,30){$y$}	
	\end{overpic}
	\caption{ The angle between any two curves is preserved.}
	\label{angle_conformal_3}
\end{figure}
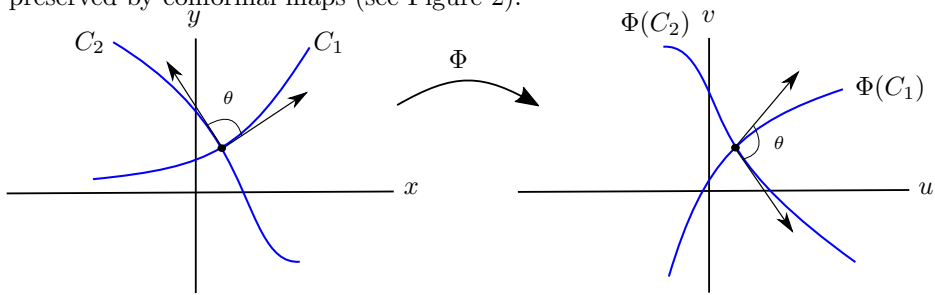
An interesting geometric property that complex analytic functions satisfy is that, at non-critical points (points with nonzero derivative), they preserve angles and consequently define conformal mappings.
\begin{proposition}
 If $w=\Phi(z)$ is an analytic function and $\Phi'(z)\neq 0$, then $\Phi$ defines a conformal map.
\end{proposition}
Notice that the converse is also valid, because every planar conformal map comes from a complex
analytic function with nonvanishing derivative.

\begin{remark}
Let $\Phi(z)$ be a conformal map with $\Phi(0)=0$. Then, the linear approximation of $\Phi$ near 0 (first two terms of the Taylor series) is given by
$$\Phi(z)\approx \Phi(0)+\Phi'(0)z=\Phi'(0)z$$
and if $\gamma(t)$ is a curve with $\gamma(t_0)=0$ for some $t_0\in\R,$ then 
$\Phi(\gamma(t))\approx\Phi'(0)\gamma(t),$
for all $t$ near to $t_0$.
\end{remark}

We will classify the local phase portraits of piecewise smooth holomorphic systems. To do this, we start by introducing the concept of conformal conjugation.

Let $F$ and $G$ be holomorphic functions defined in some punctured neighborhood of $0\in\C$. We say that $F$ and $G$ are \textit{$0$--conformally conjugated} if there exist $R>0$ and a conformal map $\Phi:D(0,R)\rightarrow D(0,R)$ such that  $\Phi(0)=0$ and $\Phi(\varphi_F(t,z)) =\varphi_G(t,\Phi(z))$,  for any $z\in D(0,R)\setminus\{0\}$ and all $t $ for which the above expressions are well defined and the corresponding points are in $ D(0,R)$.
	
Let $F$ and $G$ be holomorphic functions defined in some punctured neighborhoods of $z_1\in\C$ and $z_2\in\C$, respectively. We say that  $F$ and $G$ are \textit{$z_1z_2$--conformally conjugated} if $F(z-z_1)$ and $G(z-z_2 )$ are conformally conjugated at $0$.\\

If $F$ and $G$ are holomorphic in  $D(0,R)$ then we have:
\begin{itemize}
	\item If $F(0)\neq0$, $G(0)\neq0$ then $F$ and $G$ are $0$--conformally conjugated;
	\item If $F(0)\neq0$, $G(0)=0$  then $F$ and $G$ are not $0$--conformally conjugated;
    \item If $F(0)=0$, $G(0)=0$ and $F,G$ are non constant then 
    \[  \Phi(\varphi_F(t,z)) =\varphi_G(t,\Phi(z))\Leftrightarrow \Phi'(z)F(z)=G(\Phi(z)),\]
    for $ |z|$ sufficiently small.
\end{itemize}
The following proposition, whose proof can be found in \cite{BT,GGJ2}, gives us important information about the normal forms of holomorphic functions.
\begin{proposition}\label{GGJ}
Let $F$ be a holomorphic function defined in some punctured neighborhood of $w_0\in\C$.
\begin{itemize}
	\item [(a)] If $F(w_0)\neq0$  then $F$ and $G(z)\equiv 1$   are $w_00$--conformally conjugated.
		\item [(b)] If $F(w_0)=0$ and  $F'(w_0)\neq0$     then $F$ and $G(z)\equiv F'(w_0)z$   are $w_00$--conformally conjugated. 
			\item [(c)] If $F(w_0)=0$, $w_0$  is a zero of $F$ of order $n>1$ and  $\operatorname{Res}(1/F,w_0)=1/\gamma$ then 
			$F$ and $G(z)\equiv \gamma z^n/(1+z^{n-1})$   are $w_00$--conformally conjugated. 
			\item[(d)] If $F(w_0)=0$, $w_0$  is a zero of $F$ of order $n>1$ and  $\operatorname{Res}(1/F,w_0)=0$ then 
			$F$ and $G(z)\equiv z^n$   are $w_00$--conformally conjugated. 
			\item[(e)] If $w_0$  is a pole of $F$ of order $n$  then 
			$F$ and $G(z)\equiv \frac{1}{z^n} $   are $w_00$--conformally conjugated. 
			
\end{itemize}
\end{proposition}
Due to the beauty of the argument used in \cite{GGJ2} to demonstrate the following result, let us reproduce its demonstration here.
\begin{proposition}\label{teo_nocl}
Let $F$ be a holomorphic function defined in a domain $\mathcal{V}\subseteq\C$. The phase portrait of
$\dot{z}=F(z)$ has no limit cycle.
\end{proposition}
\begin{proof} Suppose $ \gamma $ is a periodic orbit of $ \dot {z} = F (z) $ with period $ T $, i.e. 
$ \varphi_F (z, T) = z $ whatever $ z \in\gamma $. Let us fix any point in $ \gamma $ and consider the transition
 function given $ \xi (z) = \varphi_F (z, T)$. The transition function is analytic and is equal to identity at all 
 points that are in $ \gamma $. Thus, this function coincides with the identity in a neighborhood of $ z $. 
 This means that the periodic orbit belongs to a continuum of periodic orbits, all with the same period $T$.\end{proof} 
%\begin{theorem}
%	Let $F$ be a holomorphic function in $\C$.  If the attraction basin of $0$ is the whole plan $\C$
%	then there exists $c\in\C$ with $\operatorname{Re}(c)<0$ such that $F(z)=cz.$
%\end{theorem}
%
%\begin{theorem}
%	Let $F$ be a holomorphic function in $\C$.  If for all $z\in\C$ the solution of $\dot{z}=F(z)$ 
%	is periodic and turn around the origin then  there exists $c\in\C,c\neq 0$ with $\operatorname{Re}(c)=0$ such that $F(z)=cz.$
%\end{theorem}
%
%For the proofs of both theorems we indicate \cite{ BT}.
\section{Piecewise smooth holomorphic systems}\label{sec:PWHS}
This section is devoted to study the piecewise smooth holomorphic systems,  
\begin{equation}\label{ch4:eq1}
\begin{aligned}
\left\{\begin{array}{l}
\dot{z}^{+}=f^{+}(z)=u_1+iv_1, \text{ when } \Re(z)> 0,\\[5pt]
\dot{z}^{-}=f^{-}(z)=u_2+iv_2,\text{ when }\Re(z)<0,
\end{array} \right.
\end{aligned}
\end{equation}
where  $z=x+iy$ and $f^{\pm}(z)$ are holomorphic functions. The straight line $\Sigma=\{\Re(z)=0\}$ divides the plane in two half-planes 
$\Sigma^\pm$ given by $\{z :\Re(z)> 0\}$ and $\{z :\Re(z)< 0\}$, respectively. 
The trajectories on $\Sigma$ are defined following the Filippov convention. 

\begin{itemize}
	\item[(i)] $\Sigma^w =\{z\in\Sigma:u_1u_2 >0 \}$ is the \textit{sewing} region;
	\smallskip
	\item[(ii)] $\Sigma^s=\{z\in\Sigma:u_1u_2 <0 \}$ is the \textit{sliding} region;
	\smallskip
	\item[(iii)] $\Sigma^t=\{z\in\Sigma:u_1u_2=0 \}$ is the \textit{tangent} region.
\end{itemize}

We say that $p\in\Sigma^{s}$ is an \textit{attracting sliding point} and 
denote $p\in\Sigma^{s}_s$ if $u_1<0$ and $u_2>0$. We say that $p\in\Sigma^{s}$ is a \textit{repelling sliding point}
and denote $p\in\Sigma^{s}_u$ if $u_1>0$ and $u_2<0$.

The orbits of the PWHS by $\Sigma^w$ are naturally concatenated. The orbits by $\Sigma^s$ follow the
flow of  the \emph{sliding vector field}  $F^{\Sigma}$,  which is a linear convex combination of $(u_1,v_1)$ and $(u_2,v_2)$
tangent to $\Sigma$:
\begin{equation}\label{GeralSVF}
F^{\Sigma} = \Big(0,\dfrac{u_1v_2-u_2v_1}{u_1-u_2}\Big).
\end{equation}

\subsection{Phase portrait of the PWHS}
To study the phase portraits, we shall make combinations of the items of Proposition \ref{GGJ}.\\

%\noindent\textbf{Case 1.}
%The first case is given by the constants vector fields. The PWHS is
%\begin{equation}
%\begin{aligned}
%\left\{\begin{array}{l}
%\dot{z}^{-}=1,\text{ when } \Re(z)<0, \\[5pt]
%\dot{z}^{+}=1, \text{ when } \Re(z)>0.
%\end{array} \right.
%\end{aligned}
%\end{equation}
%
%In cartesian coordinates, we have 
%
%\begin{equation}
%\begin{aligned}
%\left\{\begin{array}{l}
%(\dot{x}^{-},\dot{y}^{-})=(1,0),\text{ when } x<0, \\[5pt]
%(\dot{x}^{+},\dot{y}^{+})=(1,0), \text{ when } x>0.
%\end{array} \right.
%\end{aligned}
%\end{equation}
%
%As $u_1u_2 =1$, it follows that $\Sigma$ has only sewing region.
%See figure \eqref{caso1}.
%\begin{figure}[h]
%	\begin{center}
%		\begin{overpic}[scale=0.4]{case1.pdf}
%		%\begin{overpic}[grid,tics=5,width=6cm]{case1.pdf}		
%        \put(25,102){$\Sigma^-$}
%        \put(75,102){$\Sigma^+$}
%		\put(50,-6){$\Sigma$}
%		\end{overpic}
%		\caption{Phase portrait of the normal form $\dot{z}^{+}=1$ and $\dot{z}^{-}=1$.}
%	\label{caso1}
%	\end{center}
%	\end{figure}

\noindent\textbf{Case 1.}
%\noindent\textbf{ (1A).} 
We take $\dot{z}^{+}=f'(p)(z-z_0)$, where $f'(p)=a+ib$ and $z_0=x_0+iy_0$. The PWHS is
\begin{equation}\label{case2a}
\begin{aligned}
\left\{\begin{array}{l}
\dot{z}^{-}=1,\text{ when } \Re(z)<0, \\[5pt]
\dot{z}^{+}=f'(p)(z-z_0),\text{ when } \Re(z)>0.
\end{array} \right.
\end{aligned}
\end{equation}
In cartesian coordinates, we have 
\begin{equation}
\begin{aligned}
\left\{\begin{array}{l}
(\dot{x}^{-},\dot{y}^{-})=(1,0),\text{ when } x<0, \\[5pt]
(\dot{x}^{+},\dot{y}^{+})=(a(x-x_0)-b(y-y_0),b(x-x_0)+a(y-y_0)), \text{ when } x>0.
\end{array} \right.
\end{aligned}
\end{equation} 
As $u_1u_2=-ax_0-b(y-y_0)$ in $\Sigma$, then we get the following table:
\begin{equation*}\label{table_case2}
\begin{array}{|| c| c|c| c | c | c | c||}
\hline
a&b&x_0&\Sigma^w &\Sigma^{s}_s & \Sigma^t \\
\hline\hline
+&0&0	& &  & \R\\
\hline
+&0&+	& &\R  & \\
\hline
+&0&-	&\R &  & \\
\hline
-&0&0 & & &    \R \\
\hline
-&0&+ &\R &  &   \\
\hline
-&0&-&- &\R   &\\
\hline
\R&+&\R&(-\infty,y_0-\frac{a}{b}x_0) & (y_0-\frac{a}{b}x_0,+\infty)   &y_0-\frac{a}{b}x_0\\
\hline
\R&-&\R&(y_0-\frac{a}{b}x_0,+\infty) & (-\infty,y_0-\frac{a}{b}x_0)& y_0-\frac{a}{b}x_0 \\
\hline
\end{array}
\end{equation*}
\begin{figure}[h]
	\begin{center}
		\begin{overpic}[scale=0.4]{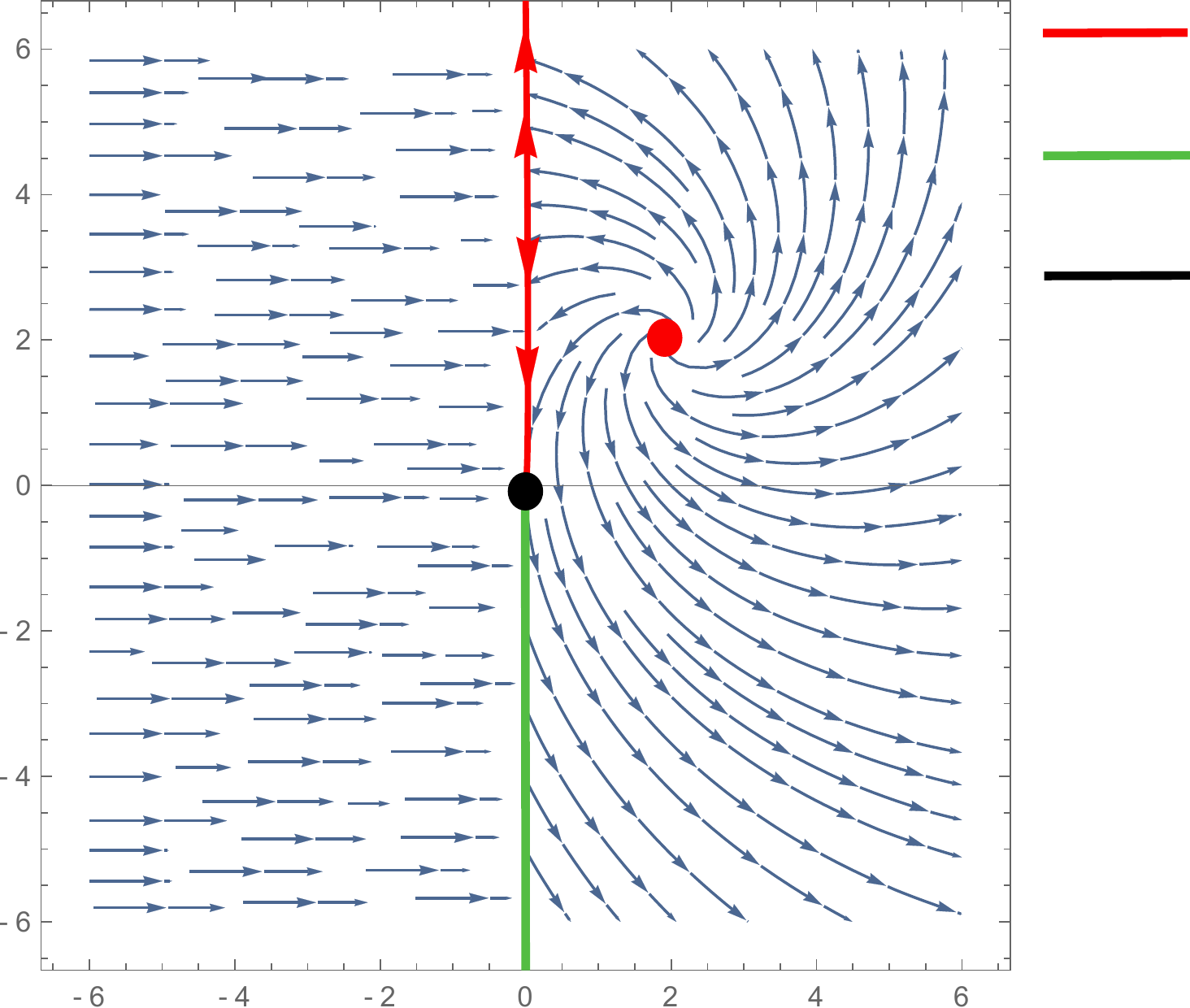}
		%\begin{overpic}[grid,tics=5,width=12cm]{caso2_c.pdf}		
		 \put(23,86){$\Sigma^-$}
        \put(64,86){$\Sigma^+$}
        \put(102,82){$\Sigma^s_s$}
        \put(102,71){$\Sigma^w$}
        \put(102,61){$\Sigma^t$}
		\put(43,-5){$\Sigma$}
		\end{overpic}
		\caption{Phase portrait of PWHS \eqref{case2a}, with $z_0=2+2i,$ $a=1$ and $b=1$.}
	\label{caso2_e}
	\end{center}
	\end{figure}
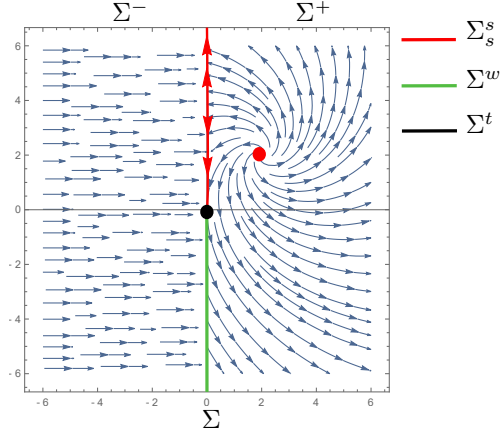

\noindent\textbf{Case 2.}
Now  $\dot{z}^{-}$ and $\dot{z}^{+}$ are given by $(a+ib)(z-z_0)$ and
$(c+id)(z-z_0)$ respectively. The PWHS is
\begin{equation}\label{case3}
\begin{aligned}
\left\{\begin{array}{l}
\dot{z}^{-}=(a+ib)(z-z_0),\text{ when } \Re(z)<0, \\[5pt]
\dot{z}^{+}=(c+id)(z-z_0),\text{ when } \Re(z)>0.
\end{array} \right.
\end{aligned}
\end{equation}
In cartesian coordinates, we have 
\begin{equation}
\begin{aligned}
\left\{\begin{array}{l}
(\dot{x}^{-},\dot{y}^{-})=(a(x-x_0)-b(y-y_0),b(x-x_0)+a(y-y_0)),\text{ when } x<0, \\[5pt]
(\dot{x}^{+},\dot{y}^{+})=(c(x-x_0)-d(y-y_0),d(x-x_0)+c(y-y_0)), \text{ when } x>0.
\end{array} \right.
\end{aligned}
\end{equation} 
As  $u_1u_2=d(y-y_0)(ax_0+b(y-y_0))$ in $\Sigma,$ then we get the following table:
\begin{equation*}\label{table_case2}
\begin{array}{|| c |c |c | c|c| c | c | c | c | c||}
\hline
a&b&c&d&x_0&\Sigma^w &\Sigma^{s}_s & \Sigma^{s}_u & \Sigma^t \\
\hline\hline
0&+&0&+&\R	&\R\setminus\{y_0\} &&  & y_0\\
\hline
0&+&0&-&\R & & (-\infty,y_0)& (y_0,+\infty) &  y_0 \\
\hline
+&+&0&-&0& &(-\infty,y_0) & (y_0,+\infty)  &y_0\\
\hline
+&+&0&-&+& (y_0-\frac{a}{b}x_0,y_0) & (-\infty,y_0-\frac{a}{b}x_0) & (y_0,+\infty) &y_0;y_0-\frac{a}{b}x_0\\
\hline
+&+&0&-&-&(y_0,y_0-\frac{a}{b}x_0) & (-\infty,y_0) & (y_0-\frac{a}{b}x_0,+\infty)&y_0;y_0-\frac{a}{b}x_0 \\
\hline
\end{array}
\end{equation*}
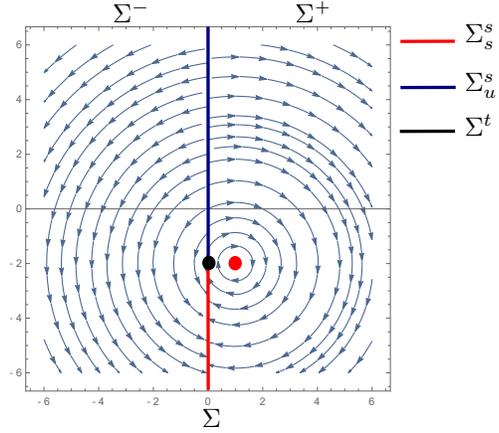
\begin{figure}[h]
	\begin{center}
		\begin{overpic}[scale=0.4]{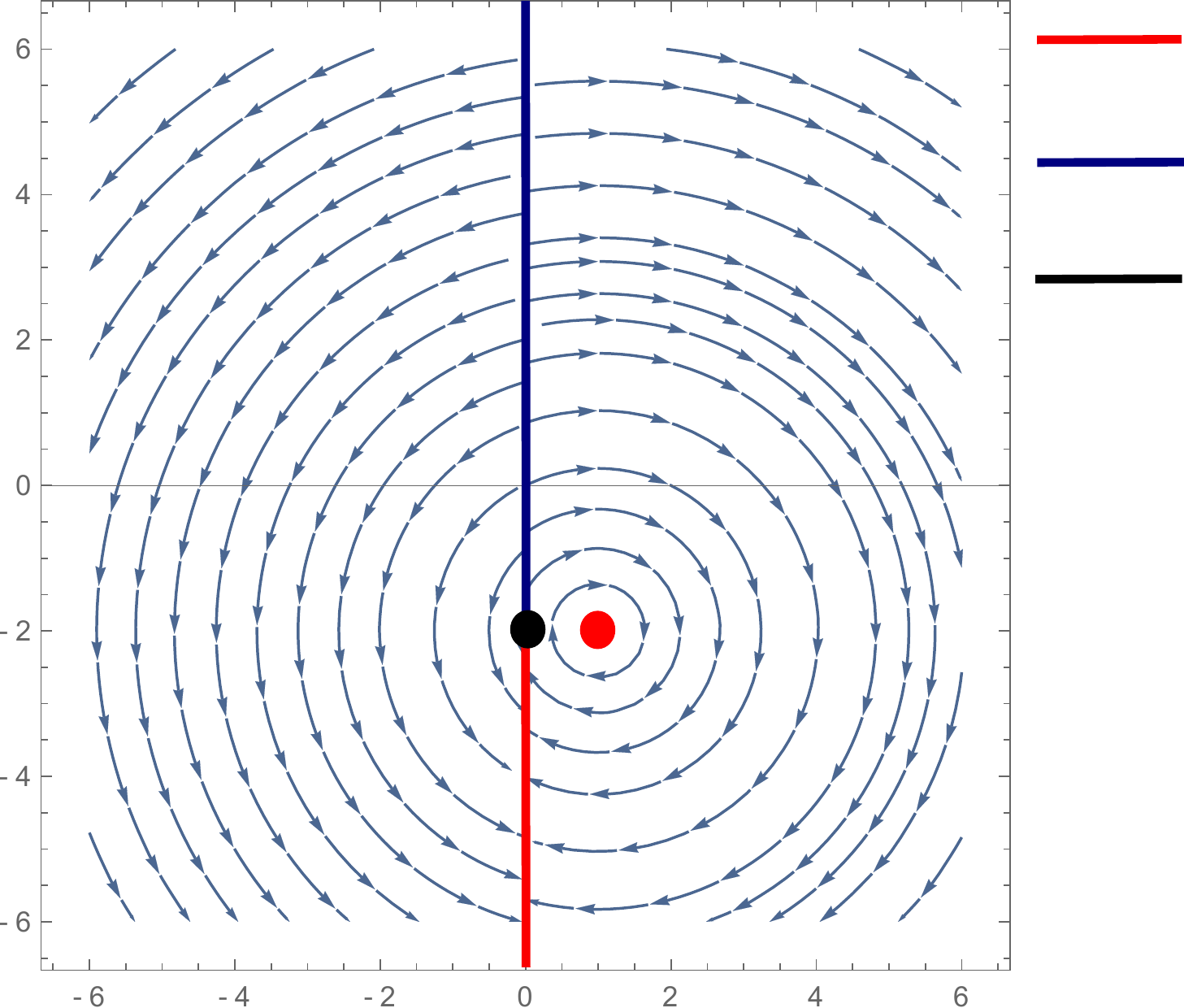}
		%\begin{overpic}[grid,tics=5,width=12cm]{caso2_c.pdf}		
\put(23,86){$\Sigma^-$}
        \put(64,86){$\Sigma^+$}
        \put(102,81.5){$\Sigma^{s}_s$}
        \put(102,70.5){$\Sigma^{s}_u$}
        \put(102,60.5){$\Sigma^t$}
		\put(43,-5){$\Sigma$}
		\end{overpic}
		\caption{Phase portrait of PWHS \eqref{case3}, with $z_0=1-2i$, $b>0,$ and $d<0$.}
	\label{caso3_b}
	\end{center}
	\end{figure}

\noindent\textbf{Case 3.}
Here, we consider $\dot{z}^{-}=f'(p)(z-z_0)$ and $\dot{z}^{+}=(z-z_0)^{n}$, where $n=2,$ $f'(p)=a+ib$, and $z_0=x_0+iy_0$. The PWHS is
\begin{equation}\label{case4}
\begin{aligned}
\left\{\begin{array}{l}
\dot{z}^{-}=(a+ib)(z-z_0),\text{ when } \Re(z)<0, \\[5pt]
\dot{z}^{+}=(z-z_0)^2,\text{ when } \Re(z)>0.
\end{array} \right.
\end{aligned}
\end{equation}
In cartesian coordinates, we have
\begin{equation}
\begin{aligned}
\left\{\begin{array}{l}
(\dot{x}^{-},\dot{y}^{-})=(a(x-x_0)-b(y-y_0),b(x-x_0)+a(y-y_0)),\text{ when } x<0, \\[5pt]
(\dot{x}^{+},\dot{y}^{+})=((x-x_0)^{2}-(y-y_0)^{2},2(x-x_0)(y-y_0)), \text{ when } x> 0.
\end{array} \right.
\end{aligned}
\end{equation}
As  $u_1u_2=(x_0^2-(y-y_0)^2)(-ax_0-b(y-y_0))$ in $\Sigma$, then we get the following table:
\begin{equation*}\label{table_case3}
\begin{array}{|| c |c| c | c | c | c | c||}
\hline
b&	x_0&\Sigma^w &\Sigma^{s}_s & \Sigma^{s}_u\\
\hline\hline
+&0	& (y_0,+\infty) &(-\infty,y_0) & \\
\hline
+&+ &(y_{min},y_{max})\cup & (-\infty,y_{min})& (y_{max},y_0+x_0)\\
& &(y_0+x_0,\infty) &  & \\
\hline
+&-& (y_0+x_0,y_{min})\cup & (-\infty,y_0+x_0)\cup & (y_0-\frac{a}{b}x_0,y_0-x_0)\\
&& (y_{max},+\infty) &(y_0-x_0,y_0-\frac{a}{b}x_0)  & \\
\hline
-&0& (-\infty,y_0) & (y_0,+\infty) & \\
\hline
-&+&(-\infty,y_{min})\cup & (y_0-\frac{a}{b}x_0,y_0-x_0)\cup & (y_0-x_0,y_0-\frac{a}{b}x_0) \\
&&(y_{max},y_0+x_0) & (y_0+x_0,\infty) &  \\
\hline
-&-& (-\infty,y_0+x_0)\cup  & (y_{max},+\infty)& (y_0+x_0,y_{min}) \\
&& (y_{min},y_{max})  & &  \\
\hline
0&\operatorname{sgn}(x_0)=\operatorname{sgn}(a)&\R\setminus[y^0_{min},y^0_{max}]  & & (y^0_{min},y^0_{max}) \\
\hline
0&\operatorname{sgn}(x_0)\neq \operatorname{sgn}(a)&  (y^0_{min},y^0_{max})  & \R\setminus[y^0_{min},y^0_{max}] & \\
\hline
\end{array}
\end{equation*}
where $y_{min}:=\min\{y_0-\frac{a}{b}x_0,y_0-x_0\},$ $y_{max}:=\max\{y_0-\frac{a}{b}x_0,y_0-x_0\},$ $y^0_{min}:=\min\{y_0+x_0,y_0-x_0\},$ $y^0_{max}:=\max\{y_0+x_0,y_0-x_0\}.$\\
\begin{figure}[h]
	\begin{center}
		\begin{overpic}[scale=0.3]{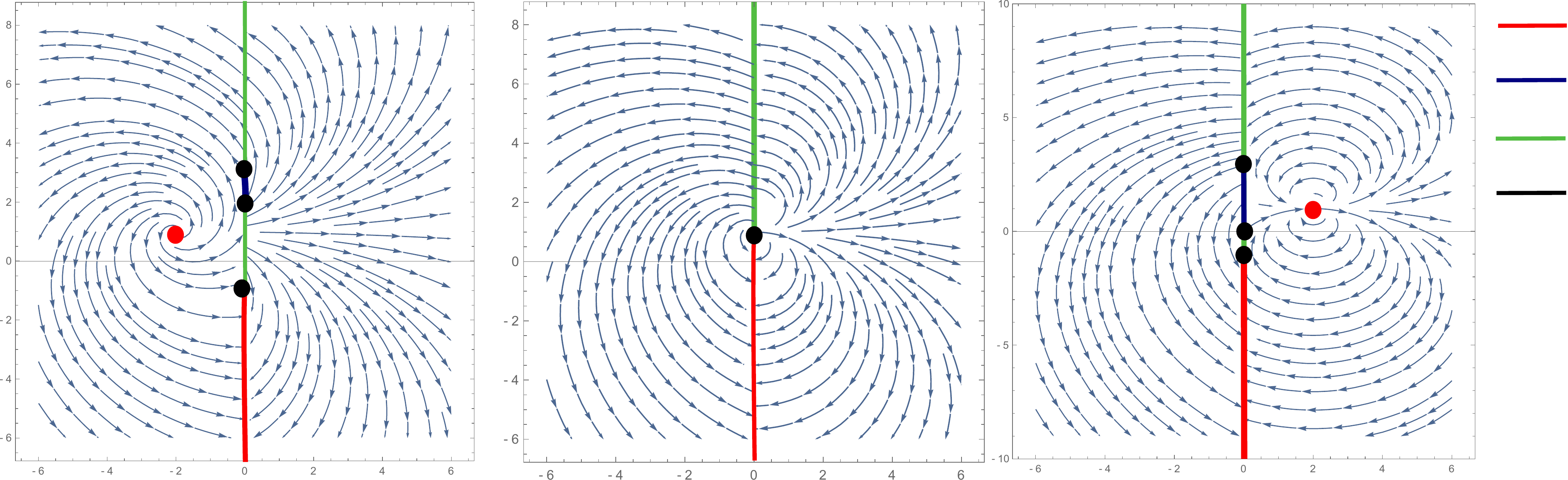}
		%\begin{overpic}[grid,tics=5,width=12cm]{caso2_c.pdf}		
       \put(7,31){$\Sigma^-$}
        \put(22,31){$\Sigma^+$}
        \put(39,31){$\Sigma^-$}
        \put(53,31){$\Sigma^+$}
        \put(71,31){$\Sigma^-$}
        \put(85,31){$\Sigma^+$}
		\put(101,29){$\Sigma^{s}_s$}
		\put(101,25){$\Sigma^{s}_u$}
		\put(101,21.5){$\Sigma^w$}
		\put(101,18){$\Sigma^t$}
		\put(11,-2){$x_0<0$}
		\put(43,-2){$x_0=0$}
		\put(76,-2){$x_0>0$}
		\end{overpic}
		\caption{Phase portrait of PWHS \eqref{case4}, with $z_0=x_0+i,$ $a=1,$ and $b=2$.}
	\label{4_a}
	\end{center}
	\end{figure}

\noindent\textbf{Case 4.}
Now, we consider $\dot{z}^{-}=f'(p)(z-z_0)$ and $\dot{z}^{+}=\frac{1}{(z-z_0)^{n}}$, with 
$n=1,$ $f'(p)=a+ib$, $a,b>0,$ and $z_0=x_0+iy_0$. The PWHS is
\begin{equation}\label{caso5}
\begin{aligned}
\left\{\begin{array}{l}
z^{-}=(a+ib)(z-z_0),\text{ when } x<0, \\[5pt]
z^{+}=\frac{1}{z-z_0}, \text{ when } x> 0.
\end{array} \right.
\end{aligned}
\end{equation}
Writing in cartesian coordinates, we have
\begin{equation}
\begin{aligned}
\left\{\begin{array}{l}
(\dot{x}^{-},\dot{y}^{-})=(a(x-x_0)-b(y-y_0),b(x-x_0)+a(y-y_0)),\text{ when } x<0, \\[5pt]
(\dot{x}^{+},\dot{y}^{+})=\left(\frac{x-x_0}{(x-x_0)^{2}+(y-y_0)^{2}},-\frac{y-y_0}{(x-x_0)^{2}+(y-y_0)^{2}}\right), \text{ when } x> 0.
\end{array} \right.
\end{aligned}
\end{equation}
As  $u_1u_2=\frac{x_0(ax_0+b(y-y_0))}{x_0^{2}+(y-y_0)^{2}}$ in $\Sigma,$ then we get the following table:
\begin{equation*}\label{table_case5}
\begin{array}{|| c |c| c | c | c | c | c||}
\hline
b&	x_0&\Sigma^w &\Sigma^{s}_s & \Sigma^{s}_u & \Sigma^t \\
\hline\hline
\R&0	& & & & \R\setminus\{y_0\}\\
\hline
+&+ &(y_0-\frac{a}{b}x_0,+\infty) & (-\infty,y_0-\frac{a}{b}x_0)&  &  y_0-\frac{a}{b}x_0  \\
\hline
+&-& (-\infty,y_0-\frac{a}{b}x_0) &  & (y_0-\frac{a}{b}x_0,+\infty)&y_0-\frac{a}{b}x_0\\
\hline
-&+& (-\infty,y_0-\frac{a}{b}x_0) & (y_0-\frac{a}{b}x_0,+\infty) & &y_0-\frac{a}{b}x_0\\
\hline
-&-& (y_0-\frac{a}{b}x_0,+\infty) &  &(-\infty,y_0-\frac{a}{b}x_0) &y_0-\frac{a}{b}x_0\\
\hline
\end{array}
\end{equation*}
\begin{figure}[h]
	\begin{center}
		\begin{overpic}[scale=0.3]{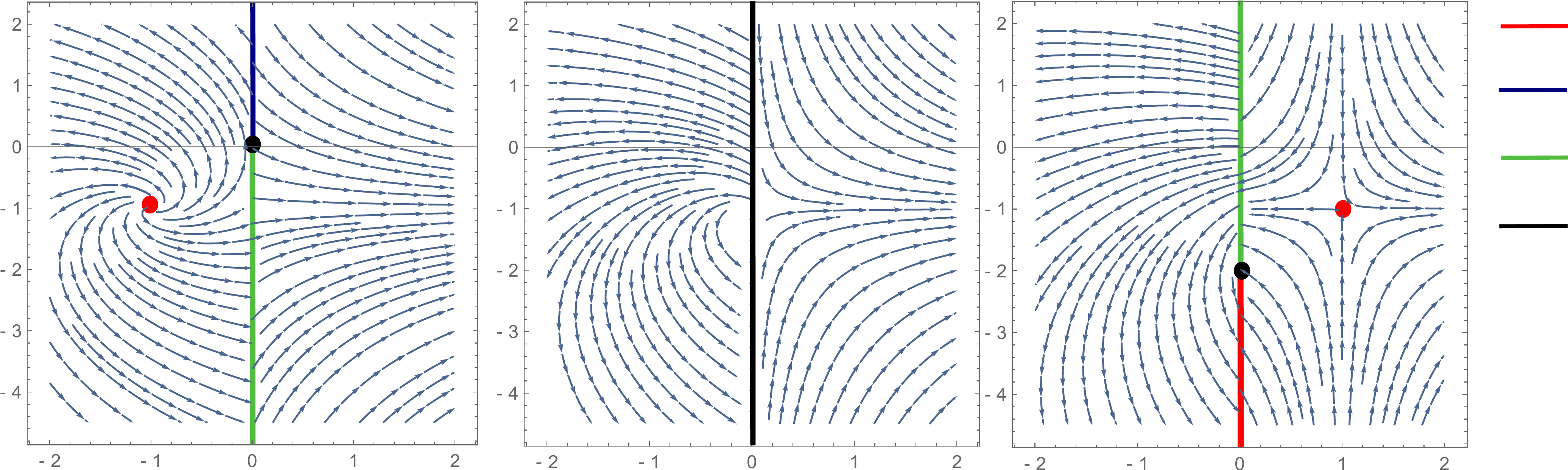}
		%\begin{overpic}[grid,tics=5,width=12cm]{caso2_c.pdf}		
        \put(7,31){$\Sigma^-$}
        \put(22,31){$\Sigma^+$}
        \put(39,31){$\Sigma^-$}
        \put(53,31){$\Sigma^+$}
        \put(71,31){$\Sigma^-$}
        \put(85,31){$\Sigma^+$}
		\put(101,28){$\Sigma^{s}_s$}
		\put(101,24){$\Sigma^{s}_u$}
		\put(101,20){$\Sigma^w$}
		\put(101,15){$\Sigma^t$}
		\put(11,-2){$x_0<0$}
		\put(43,-2){$x_0=0$}
		\put(76,-2){$x_0>0$}
		\end{overpic}
		\caption{Phase portrait of PWHS \eqref{caso5}, with $z_0=x_0-i,$ $a=1,$ and $b=1$.}
	\label{case5}
	\end{center}
	\end{figure}

	\noindent\textbf{Case 5.}
Now, we consider $\dot{z}^{-}=f'(p)(z-z_0)$ and $\dot{z}^{+}=\frac{\gamma(z-z_0)^n}{(z-z_0)^{n-1}}$, with $n=2,$ $\gamma=1$, $f'(p)=a+ib$, and $z_0=iy_0$. The PWHS is
\begin{equation}\label{case6}
\begin{aligned}
\left\{\begin{array}{l}
z^{-}=(a+ib)(z-iy_0),\text{ when } x<0, \\[5pt]
z^{+}=\frac{(z-iy_0)^2}{1+(z-iy_0)}, \text{ when } x> 0.
\end{array} \right.
\end{aligned}
\end{equation}
Writing in cartesian coordinates, we have
\begin{equation}
\begin{aligned}
\left\{\begin{array}{l}
(\dot{x}^{-},\dot{y}^{-})=(ax-b(y-y_0),bx+a(y-y_0)),\text{ when } x<0, \\[5pt]
(\dot{x}^{+},\dot{y}^{+})=\left(\frac{x^2+x^3-(y-y_0)^2+x(y-y_0)^2}{(x+1)^2+(y-y_0)^{2}},\frac{(2x+x^2+(y-y_0)^2)(y-y_0)}{(x+1)^2+(y-y_0)^{2}}\right), \text{ when } x> 0.
\end{array} \right.
\end{aligned}
\end{equation}
As  $u_1u_2=\frac{b(y-y_0)^3}{1+(y-y_0)^{2}}$ in $\Sigma,$ then we get the following table:
%As  $u_1u_2=\frac{b(y-y_0)^3}{1+(y-y_0)^{2}}$ in $\Sigma.$  Thus, if $b>0,$ then we have sewing region for $(0,y)\in \Sigma$ with $y>y_0$ and attracting sliding region for $(0,y)\in \Sigma$ with $y<y_0.$ If $b<0,$ then we have sewing region for $(0,y)\in \Sigma$ with $y<y_0$  and attracting sliding region for $(0,y)\in \Sigma$ with $y>y_0$.
\begin{equation*}\label{table_case5}
\begin{array}{|| c |c| c | c |c| c||}
\hline
b&\Sigma^w &\Sigma^{s}_s&\Sigma^t \\
\hline\hline
+& (y_0,+\infty) &(-\infty,y_0)& y_0  \\
\hline
-&(-\infty,y_0) & (y_0,+\infty)& y_0 \\
\hline
\end{array}
\end{equation*}
\begin{figure}[h]
	\begin{center}
		\begin{overpic}[scale=0.36]{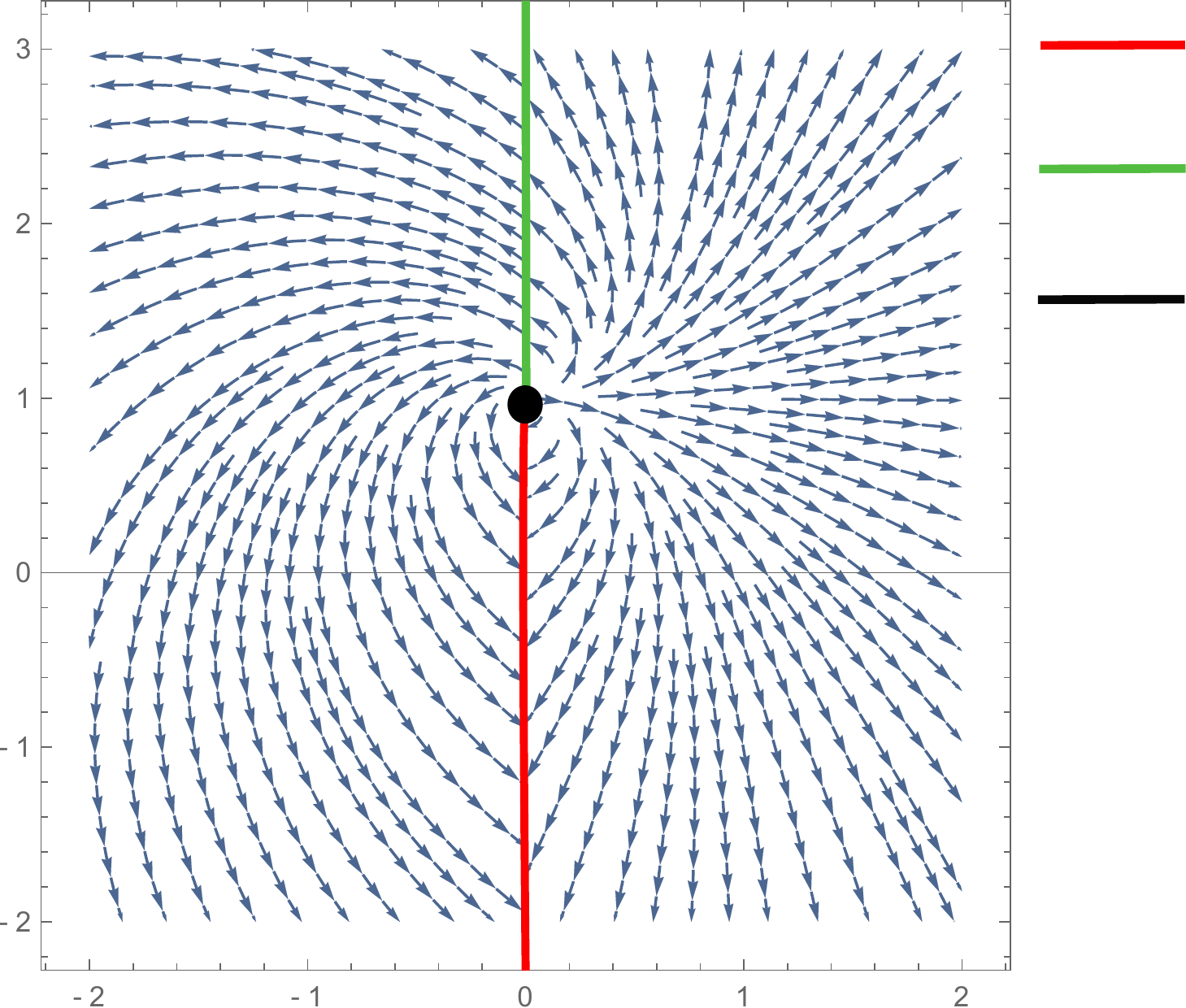}
		%\begin{overpic}[grid,tics=5,width=6cm]{case2a.pdf}		
        \put(23,88){$\Sigma^-$}
        \put(64,88){$\Sigma^+$}
        \put(102,81){$\Sigma^s_s$}
        \put(102,71){$\Sigma^w$}
        \put(102,60){$\Sigma^t$}
		\put(43,-6){$\Sigma$}
		\end{overpic}
		\caption{Phase portrait of PWHS \eqref{case6}, with $z_0=i,$ $a=1,$ and $b=1$.}
	\label{caso6}
	\end{center}
	\end{figure}
%\noindent\textbf{Case 5.}
%Now, we will consider $\dot{z}^{-}=(z-z_0)^{n}$ and $\dot{z}^{+}=(z-z_0)^{n}$, with 
%$n=2$. Writing in cartesian coordinates, we have
%
%\begin{equation}
%\begin{aligned}
%\left\{\begin{array}{l}
%(\dot{x}^{-},\dot{y}^{-})=((x-x_0)^{2}-(y-y_0)^{2},2(x-x_0)(y-y_0)),\text{ when } x<0, \\[5pt]
%(\dot{x}^{+},\dot{y}^{+})=((x-x_0)^{2}-(y-y_0)^{2},2(x-x_0)(y-y_0)), \text{ when } x> 0
%\end{array} \right.
%\end{aligned}
%\end{equation}
%
%As  $u_1u_2=((x_0)^{2}-(y-y_0)^{2})^2>0$ in $\Sigma$ we have sewing  region for $(0,y)\in \Sigma$.   with $y\neq y_0\pm x_0$. See figure \eqref{caso5}.\\
%\begin{figure}[h]
%	\begin{center}
%		\begin{overpic}[scale=0.4]{case5_a.pdf}
%		%\begin{overpic}[grid,tics=5,width=12cm]{caso2_c.pdf}		
%          \put(25,102){$\Sigma^-$}
%        \put(75,102){$\Sigma^+$}
%		\put(50,-6){$\Sigma$}
%		\end{overpic}
%		\caption{Phase portrait of the normal form $\dot{z}^{-}=(z-z_0)^{2}$ and $\dot{z}^{+}=(z-z_0)^{2}$ with $z_0=(-2,1).$}
%	\label{caso5}
%	\end{center}
%	\end{figure}
%%\begin{figure}
%%	\includegraphics[width=0.3\textwidth]{caso5}
%%	\caption{ Phase portrait of the normal form $\dot{z}^{-}=z^{2}$ and $\dot{z}^{+}=z^{2}$.}
%%{\label{caso5}}
%%\end{figure}

Now, consider $f^+$ and $f^-$ as fields in the plane (see Remark \ref{pcomp}), i.e. $f^+=(u_1,v_1)$ and $f^-=(u_2,v_2)$. Then, we can write the sewing region, the attracting sliding region, and the repelling sliding region as follow:
\[\begin{array}{rcl}
\Sigma^w&=&\{p\in \Sigma:\, f^+h(p)\cdot f^-h(p)>0\},\\
\Sigma^{s}_s&=&\{p\in \Sigma:\, f^+h(p)<0,f^-h(p)>0\}, \,\text{and}\\
\Sigma^{s}_u&=&\{p\in \Sigma:\, f^+h(p)>0, f^-h(p)<0\},\\
\end{array}
\] where $h(x,y)=x,$ $\Sigma=\{(x,y)|x=0\}=h^{-1}(0)$, and $f^{\pm}h(p)=\langle\nabla h(p),f^{\pm}(p)\rangle$ denotes the Lie derivative  of $h$ in the direction of the vector fields $f^{\pm}.$ 

 Recall that if $f_\pm^h(t):= h\circ \varphi_{f^\pm}(t,p),$ where $t\mapsto \varphi_{f^\pm}(t,p)$ is the trajectory of $f^\pm$ starting at $p,$ then $(f_\pm^h)'(0)=\langle\nabla h(p),f^{\pm}(p)\rangle=f^\pm h(p).$ 

\begin{remark}\label{pcomp}
We emphasize that there is a change of coordinates $\rho:\C\to\R^2$ between vector fields in the real plane and vector fields in the complex plane (see Figure \ref{fig:conm}).
\end{remark}
\begin{figure}[h]
	\begin{center}
		\begin{overpic}[scale=0.2]{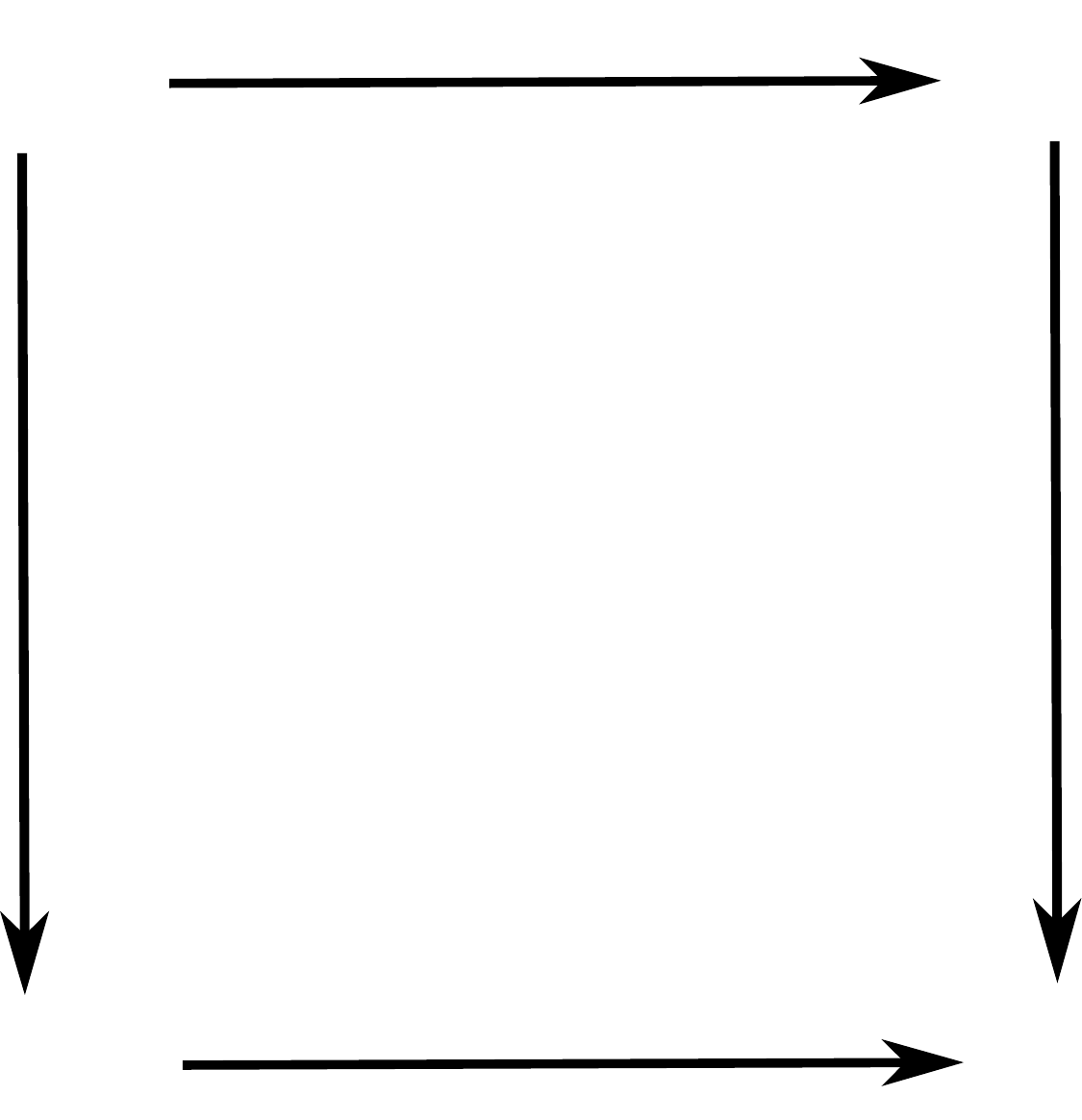}
		%\begin{overpic}[grid,tics=25,width=2.5cm]{conm.pdf}		
        \put(-2,90){$\C$}
        \put(-2,-1){$\R^2$}
        \put(90,90){$\C$}
        \put(90,-1){$\R^2$}
		\put(48,95){$F$}
		\put(48,-11){$\widetilde{F}$}
		\put(-9,50){$\rho$}
		\put(101,50){$\rho$}
		%\put(4,-3){$\operatorname{sgn}(x_0)\neq \operatorname{sgn}(a)$}
		%\put(54,-3){$\operatorname{sgn}(x_0)=\operatorname{sgn}(a)$}
		\end{overpic}
		\caption{Change of coordinates $\rho:\C\to\R^2$ associated with the vector fields $F$ and $\widetilde{F}$.}
	\label{fig:conm}
	\end{center}
	\end{figure}
\begin{definition}
We say that $\widetilde{F}$ and $\widetilde{G}$ are $0-$conformally conjugated as functions of $\R^2$ in $\R^2$ if $F$ and $G$ are $0-$conformally conjugated as functions of $\C$ in $\C$.
\end{definition}
\begin{theorem}\label{foldtofold}
Suppose that $f^\pm$ and $g^\pm$ are holomorphic and $0-$conformally conjugate with conformal map $\Phi$. Then $\Phi$ preserves sewing and sliding regions, i.e.: 
\begin{itemize}
\item[(a)] If $p\in\Sigma^w,$ then $\Phi(p)\in\widetilde{\Sigma}^w:=\Phi(\Sigma^w).$
\item[(b)] If $p\in\Sigma^{s}_s,$ then $\Phi(p)\in\widetilde{\Sigma}^s_a:=\Phi(\Sigma^{s}_s).$
\item[(c)] If $p\in\Sigma^{s}_u,$ then $\Phi(p)\in\widetilde{\Sigma}^s_r:=\Phi(\Sigma^{s}_u).$
\end{itemize}
\end{theorem}
\begin{proof}
We will prove item $(a)$, items $(b)$ and $(c)$ are verified analogously. Consider $p\in\Sigma^w,$ then $f^+h(p)\cdot f^-h(p)=f^+_1(p)\cdot f^-_1(p)>0.$
Now, let $\widetilde{h}$ be a function, such that $\widetilde{\Sigma}=\widetilde{h}^{-1}(0).$ Notice that $\widetilde{\Sigma}=\Phi(\Sigma)$ implies that $h(x,y)=\widetilde{h}\circ\Phi(x,y),$ for all $(x,y)\in\Sigma.$ Moreover, 
\begin{equation}\label{hwideh_}
\begin{array}{rcl}
\nabla h(p)&=&\nabla\widetilde{h}(\Phi(p))D\Phi(p).\\
\end{array}
\end{equation}
To prove that $\Phi(p)\in\widetilde{\Sigma}^w$ it is enough to verify that $g^+\widetilde{h}(p)\cdot g^-\widetilde{h}(p)>0.$ Indeed, consider $f_{\pm}^{\widetilde{h}}(t)=\widetilde{h}\circ\varphi_{g^\pm}(t,\Phi(p)).$ Since $f^\pm$ and $g^\pm$ are 0-conformally conjugate, then $\varphi_{g^\pm}(t,\Phi(x,y))=\Phi\circ\varphi_{f^\pm}(t,x,y),$ for all $(x,y)\in D(0,R)\setminus\{0\}.$ Hence, $f_{\pm}^{\widetilde{h}}(t)=\widetilde{h}\circ\Phi\circ\varphi_{f^\pm}(t,p)$ and using equation \eqref{hwideh_}, we get  
$(f_{\pm}^{\widetilde{h}})'(0)=f_1^\pm(p)$. Therefore, $g^+\widetilde{h}(p)\cdot g^-\widetilde{h}(p)=f_1^+(p)\cdot f_1^-(p)>0$. Consequently, we get item $(a).$ 
\end{proof}
Using the same ideas of the proof of Theorem \ref{foldtofold}, we obtain the following result.
\begin{corollary}
Suppose that $f^\pm$ and $g^\pm$ are holomorphic and $0-$conformally conjugate with conformal map $\Phi^\pm$. If $\Phi^\pm(\Sigma)=\widetilde{\Sigma}$%and $\Phi^\pm(\Sigma^\pm)\subset \widetilde{\Sigma}^\pm$
, then 
$$\Phi=\Big(\frac{1+\operatorname{sgn} (x)}{2}  \Big) \Phi^+ +\Big(\frac{1-\operatorname{sgn} (x)}{2}\Big)\Phi^-$$ preserves  sewing
 and sliding regions.
\end{corollary}

%%%\begin{remark}
%%%%Let $F$, $G$ be holomorphic functions, which are $0-$conformally conjugated, and consider $\Phi$ the associated conformal map. Now, let $p\in\Sigma$ and consider the curves $\gamma_1(t)=t+0.i$ and $\gamma_2=\varphi_G(t,p)$. We emphasize that both curves pass through point $p$, so the angle generated by these curves will be preserved by $\Phi.$
%%%Since conformal maps preserve angles, so conformal maps preserve sewing regions, attracting and repelling sliding regions (see Figure \ref{fig:sc}).
%%%\end{remark}
%%%
%%%\begin{figure}[h]
%%%	\begin{center}
%%%		\begin{overpic}[scale=0.3]{angle_conformal_sc.pdf}
%%%		%\begin{overpic}[grid,tics=5,width=12cm]{angle_conformal_sc.pdf}		
%%%        \put(44,33){$\Sigma^w$}
%%%        \put(44,8){$\Sigma^s$}
%%%        \put(102,33){$\Sigma^w$}
%%%        \put(102,8){$\Sigma^s$}
%%%		%\put(50,-2){$\Sigma$}
%%%		%\put(4,-3){$\operatorname{sgn}(x_0)\neq \operatorname{sgn}(a)$}
%%%		%\put(54,-3){$\operatorname{sgn}(x_0)=\operatorname{sgn}(a)$}
%%%		\end{overpic}
%%%		\caption{Sewing region (above), attracting sliding region  (below, right), and repelling sliding region  (below, left)
%%%.}
%%%	\label{fig:sc}
%%%	\end{center}
%%%	\end{figure}
%\begin{remark}
%The sewing regions, the attracting sliding regions, and the repelling sliding regions of the PWHS in this section are preserved by conformal conjugations.
%\end{remark}
\subsection{Tangencial points}

In the Filippov context, the notion of {\it$\Sigma$-singular points} comprehends the tangential points $\Sigma^t$ constituted by the contact points between $f^+=u_1+iv_1$ and $f^-=u_2+iv_2$ with $\Sigma,$ i.e. $\Sigma^t=\{p\in \Sigma:\, u_1\cdot u_2 = 0\},$ where $\Sigma=\{\Re(z)=0\}.$

Here, we are interested in contact points of finite degeneracy. For that reason, consider the following definition.
\begin{definition}\label{rts} Consider the PWHS given by \eqref{ch4:eq1}, $k\in\N,$ and $q\in\Sigma^t.$
\begin{itemize}
\item $q$ is called a contact of multiplicity $k$ between $f^+$ and $\Sigma$ when $u_1(q)=0,$ $v_1(q)\neq 0,$ $\frac{\partial^{i-2}}{\partial y^{i-2}}\left(\frac{\partial v_1}{\partial x}\right)(q)= 0,$ for each $i=2,\cdots, k-1,$ and $\frac{\partial^{k-2}}{\partial y^{k-2}}\left(\frac{\partial v_1}{\partial x}\right)(q)\neq 0$. Even more, if $k$ is even, then $q$ is visible when $v_1^{k-1}(q)\frac{\partial^{k-2}}{\partial y^{k-2}}\left(\frac{\partial v_1}{\partial x}\right)(q)<0$ and invisible otherwise. 
\item $q$ is called a contact of multiplicity $k$ between $f^-$ and $\Sigma$ when $u_2(q)=0,$ $v_2(q)\neq 0,$ $\frac{\partial^{i-2}}{\partial y^{i-2}}\left(\frac{\partial v_2}{\partial x}\right)(q)= 0,$ for each $i=2,\cdots, k-1,$ and $\frac{\partial^{k-2}}{\partial y^{k-2}}\left(\frac{\partial v_2}{\partial x}\right)(q)\neq 0$. Even more, if $k$ is even, then $q$ is visible when $v_2^{k-1}(q)\frac{\partial^{k-2}}{\partial y^{k-2}}\left(\frac{\partial v_2}{\partial x}\right)(q)>0$ and invisible otherwise.
\end{itemize} 
\end{definition}

\begin{definition}
Consider the PWHS given by \eqref{ch4:eq1} and $q\in\Sigma^t$:
\begin{itemize}
\item If $q$ is a visible (resp. invisible) contact of multiplicity $k$ between $f^+$ and $\Sigma$ and $q$ is a visible (resp. invisible) contact of multiplicity $k$ between $f^-$ and $\Sigma,$ then it is called a visible-visible (resp. invisible-invisible) tangential singularity of multiplicity $k.$
\item If $q$ is a visible (resp. invisible) contact of multiplicity $k$ between $f^+$ and $\Sigma$ and $q$ is an invisible (resp. visible) contact of multiplicity $k$ between $f^-$ and $\Sigma,$ then it is called an invisible-visible (resp. visible-invisible) tangential singularity of multiplicity $k.$
\end{itemize}

\end{definition}

In what follows we characterize the contacts multiplicity $k=2$ of the holomorphic functions $f^\pm$ with $\Sigma$. These contacts are known as fold singularities.
\begin{proposition}\label{rfs} 
Let $F=u+iv$ be a holomorphic function defined in some punctured neighborhood of $z_0\in\C$ and $z=x+iy$. Then $q$ is a fold singularity of $F$ with respect to $\Sigma=\{\Re(z)=0\}$ if, and only if, $u(q)=0$, $v(q)\neq 0,$ and $\Im(F'(q))\neq 0$.
\end{proposition}
\begin{proof}
Since $F$ is a holomorphic function at $q$, then
$F'(q)=\frac{\partial u}{\partial x}(q)+i\frac{\partial v}{\partial x}(q).$
Thus, $\Im(F'(q))=\frac{\partial v}{\partial x}(q).$ The result follows from the Definition \ref{rts} for $k=2.$
\end{proof}

Since conformal maps preserve angles, then tangential contacts are also preserved by these maps. Furthermore, locally the conformal maps preserve figures in the vicinity of the tangential contact, so if a tangential contact is even (resp. odd), then it is preserved by said maps.

Now, consider $f^+$ and $f^-$ as fields in the plane (see Remark \ref{pcomp}), i.e. $f^+=(u_1,v_1)$ and $f^-=(u_2,v_2)$. Then, we can write the set of tangential points as $\Sigma^t=\{p\in \Sigma:\, f^+h(p)\cdot f^-h(p) = 0\},$ where $h(x,y)=x,$ $\Sigma=\{(x,y)|x=0\}=h^{-1}(0)$, and $f^{\pm}h(p)=\langle\nabla h(p),f^{\pm}(p)\rangle$ denotes the Lie derivative  of $h$ in the direction of the vector fields $f^{\pm}.$ In addition, $(f^{\pm})^ih(p)=f^{\pm}((f^{\pm})^{i-1}h)(p)$ for $i>1.$  

Recall that $p$ is a {\it contact of order $k-1$} (or multiplicity $k$) between $f^\pm$ and $\Sigma$ if $0$ is a root of multiplicity $k$ of $f(t):= h\circ \varphi_{f^\pm}(t,p),$ where $t\mapsto \varphi_{f^\pm}(t,p)$ is the trajectory of $f^\pm$ starting at $p.$ Equivalently, $f^\pm h(p) = (f^\pm)^2h(p) = \ldots = (f^\pm)^{k-1}h(p) =0,\text{ and } (f^\pm)^{k} h(p)\neq 0.$

In addition, an even multiplicity contact, say $2k,$ is called {\it visible} for $f^+$ (resp. $f^-$) when $(f^{+})^{2k}h(p)>0$ (resp. $(f^{-})^{2k}h(p)<0$). Otherwise, it is called {\it invisible}. 

Recall that a {\it regular-tangential singularity of multiplicity $2k$} is formed by a contact of multiplicity $2k$ of $f^+$ and a regular point of $f^-,$ or vice versa. In the literature, when $k=1$, then a regular-tangential singularity of multiplicity 2 is called a regular-fold singularity.

%there exists a conformal map $\Phi:D(0_{\R^2},R)\to D(0_{\R^2},R)$ al least $C^2$, such that $\Phi(0)=0$ and $\varphi_F(t,\Phi(x,y))=\Phi\circ\varphi_G(t,x,y),$ for all $(x,y)\in D(0_{\R^2},R)\setminus\{0_{\R^2}\}.$

\begin{lemma}\label{foldtofold1}
Suppose that $G:\R^2\to\R^2$ and $F:\R^2\to\R^2$ are  $0$-conformally conjugate $C^2$ maps with conformal map  $\Phi$. If $p\in\Sigma$ is a fold singularity associated to $G,$ then $\Phi(p)\in\widetilde{\Sigma}:=\Phi(\Sigma)$ is a fold singularity associated to $F.$
\end{lemma}
\begin{proof}
Let $p\in\Sigma$ be a fold singularity associated to $G,$ then 0 is a root of multiplicity 2 of $g(t)=h\circ\varphi_G(t,p).$ Thus, $g(0)=0,$ $g'(0)=\nabla h(p)G(p)=0$ and $g''(0)=D^2h_p\left(G(p),G(p)\right)+\nabla h(p)\frac{\partial^2\varphi_G}{\partial t^2}(0,p)\neq 0,$ where we have used chain rule and that $\frac{\partial\varphi_G(0,p)}{\partial t}=G(p).$ 
%\frac{\partial\varphi_G(0,p)}{\partial t},\frac{\partial\varphi_G(0,p)}{\partial t}

Now, consider the function $\widetilde{h},$ such that $\widetilde{\Sigma}=\widetilde{h}^{-1}(0).$ Notice that $\widetilde{\Sigma}=\Phi(\Sigma)$ implies that $h(x,y)=\widetilde{h}\circ\Phi(x,y),$ for all $(x,y)\in\Sigma.$ Moreover, 
\begin{equation}\label{hwideh}
\begin{array}{rcl}
\nabla h(p)&=&\nabla\widetilde{h}(\Phi(p))D\Phi(p),\\
D^2h(p)&=&D^2\widetilde{h}_{\Phi(p)}(D\Phi(p),D\Phi(p))+\nabla\widetilde{h}(\Phi(p))D^2\Phi_p.\\
\end{array}
\end{equation}
To prove that $\Phi(p)$ is a fold singularity of $F$ it is enough to verify that 0 is a root of multiplicity 2 of $f(t)=\widetilde{h}\circ\varphi_F(t,\Phi(p)).$ Indeed, since $G$ and $F$ are 0-conformally conjugate, then $\varphi_F(t,\Phi(x,y))=\Phi\circ\varphi_G(t,x,y),$ for all $(x,y)\in D(0,R)\setminus\{0\}.$ Hence, using the equations of \eqref{hwideh}, we get
$f(0)=\widetilde{h}\circ\varphi_F(0,\Phi(p))= h(p)=0$, 
$f'(0)=\nabla\widetilde{h}(\Phi(p))D\Phi(p)G(p)=\nabla h(p)G(p)=0$ and
$$f''(0)=D^2h_p(G(p),G(p))+\nabla h(p)\frac{\partial^2\varphi_G}{\partial t^2}(0,p)\neq 0.$$
Therefore, 0 is a root of multiplicity 2 and we can conclude the result. 
\end{proof}

The following result determines the type of contacts of the normal forms given in Proposition \ref{GGJ}. 
\begin{proposition}\label{Th:folds}
Let $G=u+iv$ be a holomorphic function defined in some punctured neighborhood of $z_0$. Consider $z\in \C\setminus\{z_0\}$ and the constants $\gamma\in\R$, $n\in\N,$ and $k\in\Z$.
\begin{enumerate}
\item[(a)] If $G(z)=(a+ib)(z-z_0)$, then $G$ only has fold singularities with respect to $\Sigma$ when $b\neq 0$.
\item[(b)] If $G(z)=(z-z_0)^n$ and $n$ is even, then $G$ only has fold singularities with respect to $\Sigma$. 
\item[(c)] If $G(z)=(z-z_0)^n$ and $n>1$ is odd, then $G$ only has fold singularities with respect to $\Sigma$ when $\frac{(n-1)(2k+1)}{2n} \notin \Z$.  
\item[(d)] If $G(z)=\frac{1}{(z-z_0)^n}$ and $n$ is even, then $G$ only has fold singularities with respect to $\Sigma$. 
\item[(e)] If $G(z)=\frac{1}{(z-z_0)^n}$ and $n$ is odd, then $G$ only has fold singularities with respect to $\Sigma$ when $\frac{(n+1)(2k+1)}{2n} \notin \Z$.
\item[(f)] If $G(z)=\frac{\gamma(z-z_0)^n}{1+(z-z_0)^{n-1}}$, then $G$ only has singularities of multiplicity even with respect to $\Sigma$.
\end{enumerate}
\end{proposition}
\begin{proof}
First, consider $G(z)=(a+ib)(z-z_0)$ whose polar form is given by
$$G(z)=|(a+ib)(z-z_0)|\cos(\theta)+i|(a+ib)(z-z_0)|\sin(\theta).$$
Hence, 
$u(z)=|(a+ib)(z-z_0)|\cos(\theta)$ and 
$v(z)=|(a+ib)(z-z_0)|\sin(\theta).$
Notice that $u=0$ if, and only if, $\theta=\frac{(2k+1)\pi}{2},$ for all $k\in\Z$. Therefore,  
$v=|(a+ib)(z-z_0)|(-1)^k\neq 0$ when $\theta=\frac{(2k+1)\pi}{2},$ for all $k\in\Z$. In addition, the derivative of the complex function $(a+ib)(z-z_0)$ is $a+ib$, consequently 
$\Im(G'(z))=b,$
when $\theta=\frac{(2k+1)\pi}{2},$ for all $k\in\Z$. By Proposition \ref{rfs} we can conclude item $(a)$. 

Now, consider $G(z)=(z-z_0)^n$. Writing $G$ in its polar form, we have that
$G(z)=|z-z_0|^n\cos(n\theta)+i|z-z_0|^n\sin(n\theta).$
Thus, 
$u(z)=|z-z_0|^n\cos(n\theta)$ and 
$v(z)=|z-z_0|^n\sin(n\theta).$
Notice that $u=0$ if, and only if, $\theta=\frac{(2k+1)\pi}{2n},$ for all $k\in\Z$. Therefore,  
$v=|z-z_0|^n(-1)^k\neq 0$ when $\theta=\frac{(2k+1)\pi}{2n},$ for all $k\in\Z$. Moreover, since the derivative of the complex function $(z-z_0)^n$ is $n(z-z_0)^{n-1}$, we have that 
$$\Im(G'(z))=n|z-z_0|^{n-1}\sin\left(\frac{(n-1)(2k+1)\pi}{2n}\right),$$
when $\theta=\frac{(2k+1)\pi}{2n},$ for all $k\in\Z$. Moreover, if $n$ is even, then $\frac{(n-1)(2k+1)}{2n}\notin\Z,$ i.e. $\Im(G'(q))\neq 0.$ In addition, if $n$ is odd, then $\Im(G'(q))\neq 0$ if, and only if, $\frac{(n-1)(2k+1)}{2n}\notin\Z.$ By Proposition \ref{rfs} we get items $(b)$ and $(c)$.

On the other hand, consider $G(z)=\frac{1}{(z-z_0)^n}$. Writing $G$ in its polar form, we get
$$G(z)=|z-z_0|^{-n}\cos(n\theta)-i|z-z_0|^{-n}\sin(n\theta).$$
Thus, 
$u(z)=|z-z_0|^{-n}\cos(n\theta)$ and 
$v(z)=-|z-z_0|^{-n}\sin(n\theta)$. Notice that $u=0$ if, and only if, $\theta=\frac{(2k+1)\pi}{2n},$ for all $k\in\Z$. Therefore,  
$v=-|z-z_0|^n(-1)^k\neq 0$ when $\theta=\frac{(2k+1)\pi}{2n},$ for all $k\in\Z$. Moreover, since the derivative of the complex function $(z-z_0)^{-n}$ is $-n(z-z_0)^{-n-1}$, we have that 
$$\Im(G'(z))=n|z-z_0|^{-n-1}\sin\left(\frac{(n+1)(2k+1)\pi}{2n}\right),$$
when $\theta=\frac{(2k+1)\pi}{2n},$ for all $k\in\Z$. Thus, if $n$ is even, then $\frac{(n+1)(2k+1)}{2n}\notin\Z,$ i.e. $\Im(G'(q))\neq 0.$ In addition, if $n$ is odd, then $\Im(G'(q))\neq 0$ if, and only if, $\frac{(n+1)(2k+1)}{2n}\notin\Z.$ By Proposition \ref{rfs} we get items $(d)$ and $(e)$.

Finally, consider $G(z)=\frac{\gamma(z-z_0)^n}{1+(z-z_0)^{n-1}}$. Let $w=\Phi(z)=(z-z_0)^{1-n}$ be a conformal map. Thus we obtain the vector field
$F(w)=\frac{\gamma(1-n)}{1+\frac{1}{w}}.$
Writing $F$ in its polar form, we get
$$F(w)=\frac{\gamma(1-n)(1+|w|^{-1}\cos(\theta))}{1+2|w|^{-1}\cos(\theta)+|w|^{-2}}+i\frac{\gamma(1-n)|w|^{-1}\sin(\theta)}{1+2|w|^{-1}\cos(\theta)+|w|^{-2}}.$$
Hence, 
$$\begin{array}{rcl}
u(w)&=&\dfrac{\gamma(1-n)(1+|w|^{-1}\cos(\theta))}{1+2|w|^{-1}\cos(\theta)+|w|^{-2}},\\
v(w)&=&\dfrac{\gamma(1-n)|w|^{-1}\sin(\theta)}{1+2|w|^{-1}\cos(\theta)+|w|^{-2}}.\\
\end{array}$$
Notice that $u=0$ if, and only if, $1+|w|^{-1}\cos(\theta)=0$, which implies that $cos(\theta)\neq 0.$ Hence, $v(w)=\frac{\gamma(n-1)\cos(\theta)}{\sin(\theta)}\neq 0$, when $1+|w|^{-1}\cos(\theta)=0$. Moreover, since the derivative of the complex function $\frac{\gamma(1-n)}{1+\frac{1}{w}}$ is $\frac{\gamma(1-n)}{(1+w)^2}$, we have that 
$$\Im(F'(w))=\frac{2\gamma(1-n)\cos(\theta)}{\sin(\theta)}\neq 0,$$
when $1+|w|^{-1}\cos(\theta)=0$. By Proposition \ref{rfs}  we get item $(f)$.
 \end{proof}
Now we establish the main theorem of this section, which is a direct consequence of Propositions \ref{GGJ} and \ref{Th:folds} and Lemma \ref{foldtofold1}.
\begin{theorem}\label{car_nf}
%Suppose that $G$ and $F$ are $w_0z_0-$conformally conjugate and let $\Phi$ be the associated conformal map. Consider the constants $n\in\N$ and $k\in\Z$.
%$z\in \C\setminus\{w_0\}$ and
Let $F$ be a holomorphic function defined in some punctured neighborhood of $z_0\in\C.$ Consider $z\in \C\setminus\{z_0\}$ and the constants $n\in\N,$ $k\in\Z,$ and $\gamma\in\R$.
\begin{itemize}
	%\item [(a)] If $F(z_0)\neq0$  then $F$ and $G(z)\equiv 1$   are $z_00$--conformally conjugated.
		\item [(a)] If $F(z_0)=0$ and  $Im(F'(z_0))\neq0,$ then there exists a conformal map $\Phi$ such that $F$ only has fold singularities with respect to $\Phi(\Sigma).$
			\item[(b)] If $F(z_0)=0$, $z_0$  is a zero of $F$ of order $n>1$ with $n$ even, and $\operatorname{Res}(1/F,z_0)=0,$ then there exists a conformal map $\Phi$ such that $F$ only has fold singularities with respect to $\Phi(\Sigma)$.
			\item[(c)] If $F(z_0)=0$, $z_0$  is a zero of $F$ of order $n>1$ with $n$ odd, $\frac{(n-1)(2k+1)}{2n} \notin \Z,$ and $\operatorname{Res}(1/F,z_0)=0,$ then there exists a conformal map $\Phi$ such that $F$ only has fold singularities with respect to $\Phi(\Sigma)$.
			\item[(d)] If $z_0$  is a pole of $F$ of order $n$ with $n$ even, then there exists a conformal map $\Phi$ such that $F$ only has fold singularities with respect to $\Phi(\Sigma)$.
			\item[(e)] If $z_0$  is a pole of $F$ of order $n$ with $n$ odd and $\frac{(n-1)(2k+1)}{2n} \notin \Z,$ then there exists a conformal map $\Phi$ such that $F$ only has fold singularities with respect to $\Phi(\Sigma)$.
			\item [(f)] If $F(z_0)=0$, $z_0$  is a zero of $F$ of order $n>1$, and $\operatorname{Res}(1/F,z_0)=1/\gamma,$ then there exists a conformal map $\Phi$ such that $F$ only has tangential singularities of multiplicity even with respect to $\Phi(\Sigma)$.
\end{itemize}
\end{theorem}
 
 Now, we present an example of a holomorphic function defined in some punctured neighborhood of $z_0=0\in\C$ with an essential singularity at $z_0=0,$ which has infinite contacts of multiplicity 3.
 
 \begin{example}
 Consider the ODE 
 \begin{equation}\label{eqesse}
 \dot{z}=z^m\exp\left(\frac{1}{z^n}\right),
 \end{equation} 
where $n\geq 1$ and $m\geq n+1$. Doing a scaling of the time and writing $z=x+iy$, system \eqref{eqesse} becomes a real smooth planar system in a punctured neighborhood of the origin of the form
\begin{equation}
\begin{aligned}
\left\{\begin{array}{l}
\dot{x}=P_m(x,y)\cos\left(\frac{R_n(x,y)}{(x^2+y^2)^n}\right)-Q_m(x,y)\sin\left(\frac{R_n(x,y)}{(x^2+y^2)^n}\right),\\
\dot{y}=P_m(x,y)\sin\left(\frac{R_n(x,y)}{(x^2+y^2)^n}\right)+Q_m(x,y)\cos\left(\frac{R_n(x,y)}{(x^2+y^2)^n}\right),
\end{array} \right.
\end{aligned}
\end{equation}
where $P_m$, $Q_m$, and $R_n$ are given by
$$\begin{array}{rcl}
P_m(x,y)&=&\displaystyle\sum_{j=0}^{l}\frac{m!}{(m-2j)!(2j)!}x^{m-2j}(-1)^jy^{2j},\\
Q_m(x,y)&=&\displaystyle\sum_{j=1}^{s}\frac{m!}{(m-(2j-1))!(2j-1)!}x^{m-(2j-1)}(-1)^{j-1}y^{2j-1},\\
R_n(x,y)&=&\displaystyle\sum_{j=1}^{s}\frac{n!}{(n-(2j-1))!(2j-1)!}x^{n-(2j-1)}(-1)^{j}y^{2j-1},\end{array}$$
where $l=s=\frac{m}{2}$ when $m$ is even, $l=\frac{m-1}{2}$ and $s=\frac{m+1}{2}$ when $m$ is odd, $s=\frac{n}{2}$ when $n$ is even, and $s=\frac{n+1}{2}$ when $n$ is odd. Now, consider the functions 
$$\begin{array}{rcl}
u_1(x,y)&=&\displaystyle P_m(x,y)\cos\left(\frac{R_n(x,y)}{(x^2+y^2)^n}\right)-Q_m(x,y)\sin\left(\frac{R_n(x,y)}{(x^2+y^2)^n}\right),\\
v_1(x,y)&=&\displaystyle P_m(x,y)\sin\left(\frac{R_n(x,y)}{(x^2+y^2)^n}\right)+Q_m(x,y)\cos\left(\frac{R_n(x,y)}{(x^2+y^2)^n}\right).
\end{array}$$
Notice that if $n$ is odd, $m$ is even, and $k\geq 1$ then $u_1(p_k)=0$ if, and only if, $p_k=\left(0,\sqrt[n]{\frac{2(-1)^{\frac{n+1}{2}}}{(2k+1)\pi}}\right)$. And if $n,m$ are odd and $k\geq 1$ then $u_1(q_k)=0$ if, and only if, $q_k=\left(0,\sqrt[n]{\frac{(-1)^{\frac{n+1}{2}}}{k\pi}}\right)$. Moreover,  
$$\begin{array}{rcl}
v_1(p_k)&=&\left(\sqrt[n]{\frac{2(-1)^{\frac{n+1}{2}}}{(2k+1)\pi}}\right)^m(- 1)^{k+\frac{m}{2}}\neq 0,\\
\frac{\partial v_1}{\partial x}(p_k)&=&0,\\
\frac{\partial}{\partial y}\left(\frac{\partial v_1}{\partial x}\right)(p_k)&=&mn(-1)^{\frac{m+n-1}{2}+k}\left(\sqrt[n]{\frac{2(-1)^{\frac{n+1}{2}}}{(2k+1)\pi}}\right)^{m-n-2}\neq 0,\\
\end{array}$$
and
$$\begin{array}{rcl}
v_1(q_k)&=&\left(\sqrt[n]{\frac{(-1)^{\frac{n+1}{2}}}{k\pi}}\right)^m(- 1)^{k+\frac{m-1}{2}}\neq 0,\\
\frac{\partial v_1}{\partial x}(q_k)&=&0,\\
\frac{\partial}{\partial y}\left(\frac{\partial v_1}{\partial x}\right)(q_k)&=&mn(-1)^{\frac{m+n}{2}+1+k}\left(\sqrt[n]{\frac{(-1)^{\frac{n+1}{2}}}{k\pi}}\right)^{m-n-2}\neq 0.

\end{array}
$$ 
for all $k\in\Z.$ By Definition \ref{rts}, we conclude that $p_k$ and $q_k$ are contacts of multiplicity 3, for all $k\in\Z$. 
 \end{example}

We conclude this session by presenting some PWHS from Section \ref{sec:PWHS} that has at least one singularity of the above.

\begin{example}\label{ex_reg}
We take $\dot{z}^{+}=f'(p)(z-z_0)$, where $f'(p)=a+ib,$ $b\neq 0,$ and $z_0=x_0+iy_0$. The PWHS is given by
\begin{equation}
\begin{aligned}
\left\{\begin{array}{l}
\dot{z}^{-}=1,\text{ when } \Re(z)<0, \\[5pt]
\dot{z}^{+}=f'(p)(z-z_0),\text{ when } \Re(z)>0.
\end{array} \right.
\end{aligned}
\end{equation}
In cartesian coordinates, we have 
\begin{equation}
\begin{aligned}
\left\{\begin{array}{l}
(\dot{x}^{-},\dot{y}^{-})=(1,0),\text{ when } x<0, \\[5pt]
(\dot{x}^{+},\dot{y}^{+})=(a(x-x_0)-b(y-y_0),b(x-x_0)+a(y-y_0)), \text{ when } x>0.
\end{array} \right.
\end{aligned}
\end{equation} 
Taking $q=(0,y_0-\frac{a}{b}x_0)$, we have that $u_1(q)=0,$ $v_1(q)=-\frac{(b^2+a^2)}{b}x_0\neq 0$ for $x_0\neq 0,$ $\frac{\partial v_1}{\partial x}(q)=b\neq 0,$ and $f^-(q)=(1,0).$ Thus, by Definition \ref{rts} we conclude that $q$ is a visible regular-fold singularity when $x_0>0$ and $q$ is an invisible regular-fold singularity when $x_0<0$ (see Figure \ref{ex1_contact}). 
\end{example}

\begin{figure}[h]
	\begin{center}
		\begin{overpic}[scale=0.4]{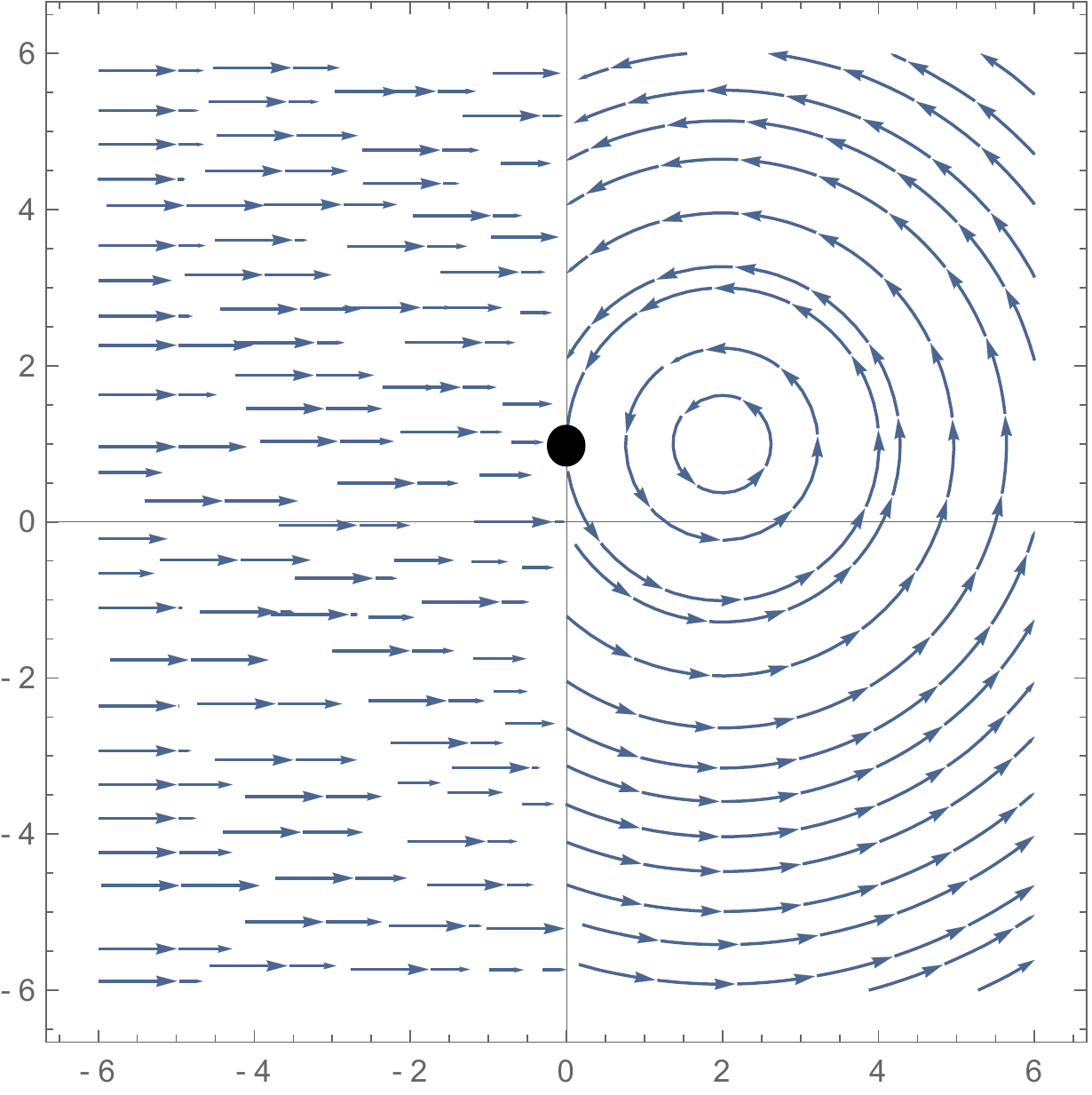}
		%\begin{overpic}[grid,tics=5,width=12cm]{caso2_c.pdf}		
          \put(25,102){$\Sigma^-$}
        \put(75,102){$\Sigma^+$}
		\put(50,-6){$\Sigma$}
		\end{overpic}
		\caption{The visible regular-fold singularity $q=(0,1)$, for $a=0$ and $b=1.$}
	\label{ex1_contact}
	\end{center}
	\end{figure}

\begin{example}
Consider $\dot{z}^{-}$ and $\dot{z}^{+}$ given by $(a+ib)(z-z_0)$ and
$id(z-z_0)$ respectively, with $b,d\neq 0$ and $z_0=x_0+iy_0$. In cartesian coordinates, we have 
\begin{equation}\label{ex:vis_inv}
\begin{aligned}
\left\{\begin{array}{l}
(\dot{x}^{-},\dot{y}^{-})=(a(x-x_0)-b(y-y_0),b(x-x_0)+a(y-y_0)),\text{ when } x<0, \\[5pt]
(\dot{x}^{+},\dot{y}^{+})=(-d(y-y_0),d(x-x_0)), \text{ when } x>0.
%(\dot{x}^{-},\dot{y}^{-})=(a(x-x_0)-b(y-y_0),b(x-x_0)+a(y-y_0)),\text{ when } x<0, \\[5pt]
%(\dot{x}^{+},\dot{y}^{+})=(c(x-x_0)-d(y-y_0),d(x-x_0)+c(y-y_0)), \text{ when } x>0.
\end{array} \right.
\end{aligned}
\end{equation} 

Suppose that $a=0$ and $x_0\neq 0.$ Taking $q=(0,y_0)$, we have that $u_1(q)=0,$ $v_1(q)=-dx_0\neq 0,$ and $\frac{\partial v_1}{\partial x}(q)=d\neq 0.$ In addition, notice that $u_2(q)=0,$ $v_2(q)=-bx_0\neq 0,$ and $\frac{\partial v_2}{\partial x}(q)=b\neq 0.$ Thus, by Definition \ref{rts} we conclude that $q$ is an invisible-visible fold singularity of \eqref{ex:vis_inv} when $x_0>0$ and $q$ is a visible-invisible fold singularity of \eqref{ex:vis_inv} when $x_0<0$.

Suppose that $a\neq 0$ and $x_0\neq 0.$
\begin{itemize}
\item Taking $q=(0,y_0)$, we have that $u_1(q)=0,$ $v_1(q)=-dx_0\neq 0,$ and $\frac{\partial v_1}{\partial x}(q)=d\neq 0.$ In addition, notice that $f^-(q)=(-ax_0,-bx_0)\neq (0,0).$ Thus, by Definition \ref{rts} we conclude that $q$ is a visible regular-fold singularity of \eqref{ex:vis_inv} when $x_0>0$ and $q$ is an invisible regular-fold singularity of \eqref{ex:vis_inv} when $x_0<0$.
\item Taking $q=(0,y_0-\frac{a}{b}x_0)$, we have that $u_2(q)=0,$ $v_2(q)=-\frac{b^2+a^2}{b}x_0\neq 0,$ and $\frac{\partial v_1}{\partial x}(q)=b\neq 0.$ In addition, notice that $f^+(q)=(\frac{ad}{b}x_0,-dx_0)\neq (0,0).$ Thus, by Definition \ref{rts} we conclude that $q$ is an invisible regular-fold singularity of \eqref{ex:vis_inv} when $x_0>0$ and $q$ is a visible regular-fold singularity of \eqref{ex:vis_inv} when $x_0<0$ (see Figure \ref{ex2_contact}).
\end{itemize}
\end{example}

\begin{figure}[h]
	\begin{center}
		\begin{overpic}[scale=0.4]{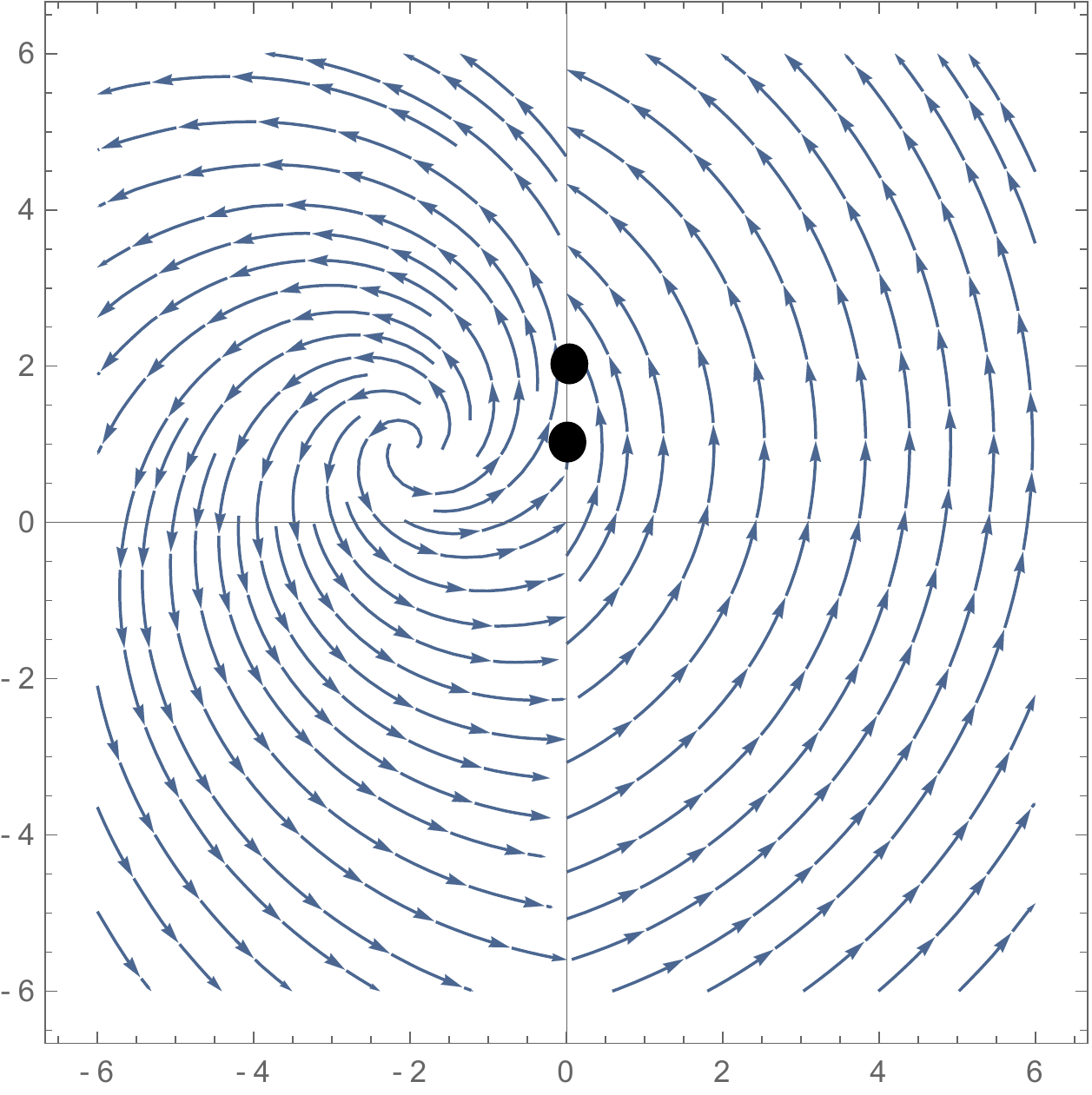}
		%\begin{overpic}[grid,tics=5,width=12cm]{caso2_c.pdf}		
          \put(25,102){$\Sigma^-$}
        \put(75,102){$\Sigma^+$}
		\put(50,-6){$\Sigma$}
		\end{overpic}
		\caption{The invisible regular-fold singularity $q_1=(0,1)$ and the visible regular-fold singularity $q_2=(0,2)$ for $a=1,$ $b=2,$ and $d=1.$}
	\label{ex2_contact}
	\end{center}
	\end{figure}

\begin{example}
Consider $\dot{z}^{-}=f'(p)(z-z_0)$ and $\dot{z}^{+}=(z-z_0)^{n}$, with $n=2$ and $z_0=x_0+iy_0$. In cartesian coordinates, we have
\begin{equation}\label{ex:vis_inv_case4}
\begin{aligned}
\left\{\begin{array}{l}
(\dot{x}^{-},\dot{y}^{-})=(a(x-x_0)-b(y-y_0),b(x-x_0)+a(y-y_0)),\text{ when } x<0, \\[5pt]
(\dot{x}^{+},\dot{y}^{+})=((x-x_0)^{2}-(y-y_0)^{2},2(x-x_0)(y-y_0)), \text{ when } x> 0.
\end{array} \right.
\end{aligned}
\end{equation}
Suppose that $b\neq 0$ and $x_0\neq 0.$
\begin{itemize}
\item Taking $q=(0,y_0+x_0),$ we have that $u_1(q)=0,$ $v_1(q)=-2x_0^2\neq 0,$ and $\frac{\partial v_1}{\partial x}(q)=2x_0\neq 0.$ In addition, notice that 
$$u_2(q)=-x_0(a+b)=\left\{\begin{array}{rcl}
0&if&a= -b,\\
\neq 0&if& a\neq-b,
\end{array}\right.$$ 
$$v_2(q)=x_0(a-b)=\left\{\begin{array}{rcl}
0&if&a= b,\\
\neq 0&if& a\neq b,
\end{array}\right.$$  and $\frac{\partial v_2}{\partial x}(q)=b\neq 0.$ Thus, by Definition \ref{rts} we conclude that if $x_0>0$ and $a=-b$ then $q$ is an invisible-visible fold singularity of \eqref{ex:vis_inv_case4} and if $x_0<0$ and $a= -b$ then $q$ is a visible-invisible fold singularity of \eqref{ex:vis_inv_case4}. Moreover, if $x_0>0$ and $a\neq -b,$ then $q$ is a visible fold singularity of $f^+$ and a regular point of $f^-,$ and if $x_0<0$ and $a\neq -b,$ then $q$ is an invisible fold singularity of $f^+$ and a regular point of $f^-$ (see Figure \ref{ex3_contact}).
\item Taking $q=(0,y_0-x_0),$ we have that $u_1(q)=0,$ $v_1(q)=2x_0^2\neq 0,$ and $\frac{\partial v_1}{\partial x}(q)=-2x_0\neq 0.$ In addition, notice that 
$$u_2(q)=x_0(b-a)=\left\{\begin{array}{rcl}
0&if&a=b,\\
\neq 0&if& a\neq b,
\end{array}\right.$$ 
$$v_2(q)=-x_0(a+b)=\left\{\begin{array}{rcl}
0&if&a= -b,\\
\neq 0&if& a\neq -b,
\end{array}\right.$$  and $\frac{\partial v_2}{\partial x}(q)=b\neq 0.$ Thus, by Definition \ref{rts} we conclude that if $x_0>0$ and $a=b$ then $q$ is an invisible-visible fold singularity of \eqref{ex:vis_inv_case4} and if $x_0<0$ and $a= b$ then $q$ is a visible-invisible fold singularity of \eqref{ex:vis_inv_case4} (see Figure \ref{ex3_contact}).
\item Taking $q=(0,y_0-\frac{a}{b}x_0),$ we have that $u_2(q)=0,$ $v_2(q)=-x_0\frac{a^2+b^2}{b}\neq 0,$ and $\frac{\partial v_1}{\partial x}(q)=b\neq 0.$ In addition, notice that 
$$u_1(q)=\frac{(b-a)(b+a)}{b^2}x_0^2=\left\{\begin{array}{rcl}
0&if&b=\pm a,\\
\neq 0&if& b=\pm a.
\end{array}\right.$$ 
$$v_1(q)=\frac{2a}{b}x_0^2=\left\{\begin{array}{rcl}
0&if&a= 0,\\
\neq 0&if& a\neq 0,
\end{array}\right.$$  and $\frac{\partial v_1}{\partial x}(q)=\frac{-2a}{b}x_0.$ Thus, by Definition \ref{rts} we conclude that if $x_0>0$ and $b=\pm a,$ then $q$ is an invisible-visible fold singularity of \eqref{ex:vis_inv_case4} and if $x_0<0$ and $b=\pm a,$ then $q$ is a visible-invisible fold singularity of \eqref{ex:vis_inv_case4}. Moreover, if $x_0<0$ and $b\neq \pm a,$ then $q$ is a visible fold singularity of $f^-$ and a regular point of $f^+,$ and if $x_0>0$ and $b\neq \pm a,$ then $q$ is an invisible fold singularity of $f^-$ and a regular point of $f^+.$
\end{itemize}
Now, suppose that $b=0,$ $a\neq 0,$ and $x_0\neq 0.$ Taking $q=(0,y_0\pm x_0),$ we have that $u_1(q)=0,$ $v_1(q)=\mp 2x_0^2\neq 0,$ and $\frac{\partial v_1}{\partial x}(q)=\pm 2x_0\neq 0.$ In addition, notice that 
$u_2(q)=-ax_0\neq 0$ and $\frac{\partial v_2}{\partial x}(q)=\pm ax_0\neq 0.$ Thus, by Definition \ref{rts} we conclude that $q$ is a visible regular-fold singularity when $x_0>0$ and $q$ is an invisible regular-fold singularity when  $x_0<0.$
\end{example}

\begin{figure}[h]
	\begin{center}
		\begin{overpic}[scale=0.4]{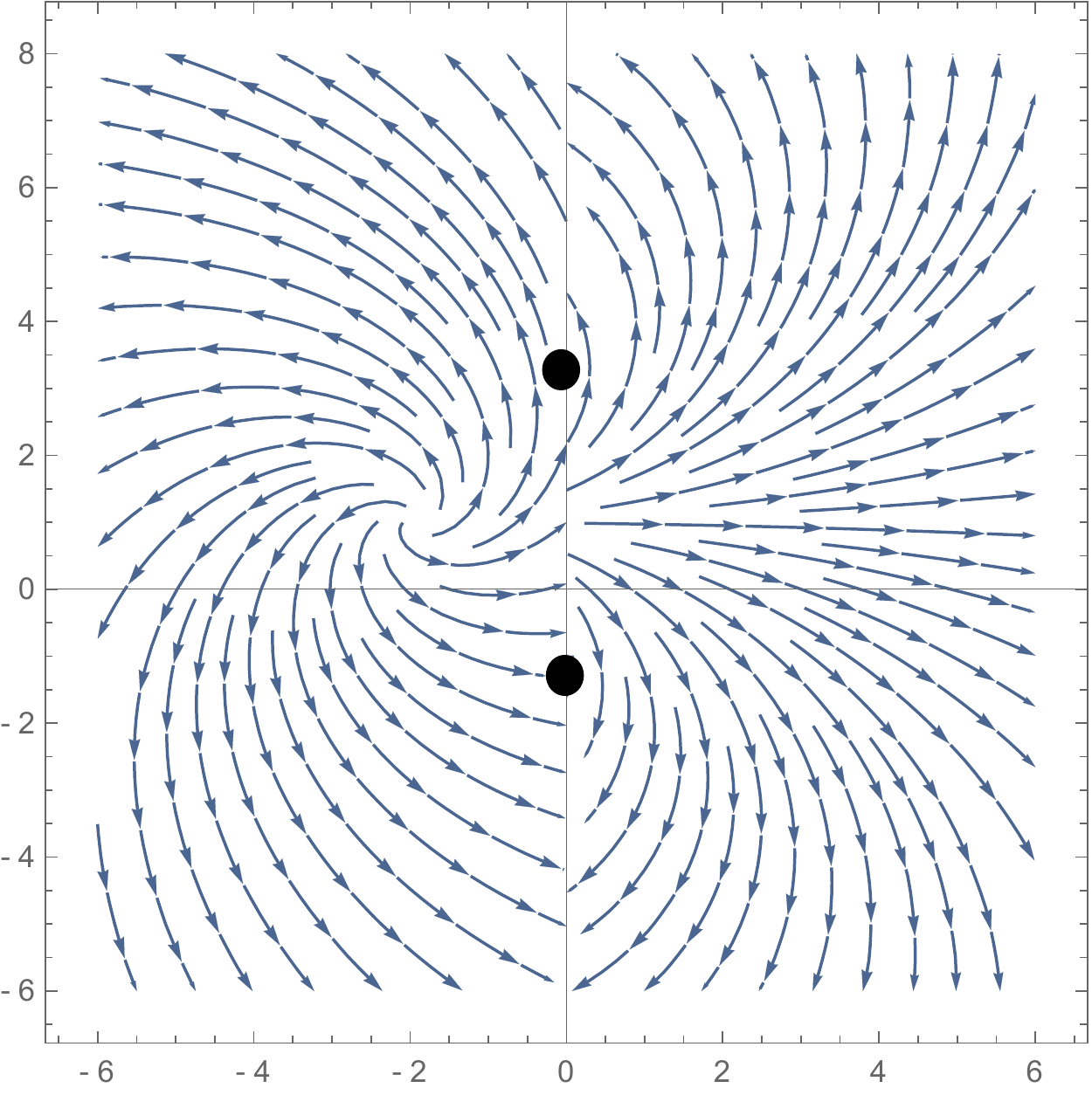}
		%\begin{overpic}[grid,tics=5,width=12cm]{caso2_c.pdf}		
          \put(25,102){$\Sigma^-$}
        \put(75,102){$\Sigma^+$}
		\put(50,-6){$\Sigma$}
		\end{overpic}
		\caption{The invisible regular-fold singularity $q_1=(0,-1)$ and the visible-invisible fold singularity $q_2=(0,3)$ for $a=1$ and $b=1.$}
	\label{ex3_contact}
	\end{center}
	\end{figure}
%\newpage
\section{Regularization of PWHS}\label{sec:reg}
In this section we are interested in determining if the regularized system associated with a PWHS preserves the property of being holomorphic. For that, consider two holomorphic ordinary differential equations 
\[\dot{z}^+=f^+(z),\quad \dot{z}^-=f^-(z)\]
defined in $\C.$
A piecewise-smooth holomorphic system  is $\dot{z}=F(z)$ with
\begin{equation}
\label{sis1} F=\Big(\frac{1+\operatorname{sgn} (\Re)}{2}  \Big) f^+ +\Big(\frac{1-\operatorname{sgn} (\Re)}{2}\Big)f^-.
\end{equation}
The set $\Sigma=\{z\in\C:\Re(z)=0\}$ is called 
\emph{switching manifold}.

The regularization process of a piecewise smooth vector field $F$ consists in obtaining a one-parameter family of continuous vector fields $F_{\e}$ converging to $F$ when $\e\to 0.$ More specifically, the Sotomayor-Teixeira regularization (\emph{ST-regularization})   is
the one parameter family $F^{\e}$ given by 
\begin{equation}  \label{regst}
F^{\e}= \Big(\frac{1+\varphi(\Re/\e)}{2}\Big) f^+ + \Big(\frac{1-\varphi(\Re/\e)}{2}\Big)f^-,
\end{equation} 
where $\varphi:\R\rightarrow[-1,1]$ is a \emph{Sotomayor-Teixeira transition function}, i.e. a smooth function satisfying that 
$\varphi(t)=1$ for $t\geq 1$, $\varphi(t)=-1$ for $t\leq -1$ and $\varphi'(t)>0$ 
for $t\in(-1,1)$ and $\varphi^{(i)}(\pm1)=0$ for $i=1,2,\ldots,n$.
The regularization  is smooth for $\e > 0$ and satisfies that 
$F^{\e}=  f^+$ on $\{z\in\C:\Re(z)\geq\e\}$ and $F^{\e}=  f^-$ on $\{z\in\C:\Re(z)\leq-\e\}$. 
%\begin{remark}
%Another transition function that can be considered is continuous piecewise linear function $\varphi:\R\to[-1,1]$ defined by
% \begin{equation}\label{linearcase}
% \varphi(t)=\left\{\begin{array}{rcl}
% 1&\text{for}& t\geq 1\\
%t&\text{for}& -1<t<1\\
% -1&\text{for}& t\leq -1.
% \end{array}\right.
% \end{equation}
%\end{remark}
\begin{theorem}\label{teo:reg} Let $\varphi$ be a Sotomayor-Teixeira transition function. If there exists some $\e>0$ such that the regularization \eqref{regst} is holomorphic 
then $f^+(z)=f^-(z)$ for all $z\in\C$.
\end{theorem}
\begin{proof}  The result is an immediate consequence of the principle of identity of the analytic functions. 
Indeed, if two analytical functions coincide in an open subset, then they coincide throughout their domain.
\end{proof}

Recently, some authors have considered a broader family of transition functions (see Definition \ref{reggeral}) that include analytical functions such as $\tanh(x)$ and $\frac{2}{\pi}\arctan(x)$ and other non-analytic functions such as the Sotomayor-Teixeira transition functions. Readers are referred to \cite{MR3976635,kris,MR3927112} for more information on these transition functions.
\begin{definition}\label{reggeral}
The transition function $\phi:\R\rightarrow[-1,1]$ is a smooth function $C^n$ which is strictly increasing $\phi'(s)>0$ for every $s$ such that $\phi(s)\in(-1,1)$ and $\phi(s)\rightarrow\pm 1,$ for $s\rightarrow\pm\infty.$
\end{definition}
%%%Now, the set of functions of the Definition \ref{reggeral} includes analytic transition functions such as $\tanh(x)$ and $\frac{2}{\pi}\arctan(x).$ However, this class of transition functions also include the non-analytical Sotomayor-Teixeira monotonic transition functions, which were introduced in \cite{ST1995}, which satisfy
%%%\begin{equation}\label{Phi}
%%%\Phi(s)=\left\{\begin{array}{ll}
%%%\varphi(s)&\text{se}\quad|s|\leqslant 1,\\
%%%sgn(s)&\text{se}\quad|s|\geqslant1,
%%%\end{array}\right.
%%%\end{equation}
%%%%where $\phi:\R\rightarrow\R$ is a function $C^{\infty}$ satisfying $\phi(\pm1)=\pm1,$ $\phi^{(i)}(\pm1)= 0$ for $i=1,2,\ldots,n,$ and $\phi'(s)>0$ for $s\in(-1,1).$ We will denote by $C^{n}_{ ST}$ to the set of $C^{n}$-functions $\Phi$ given in \eqref{Phi} that are not $C^{n+1}$ in $\pm 1.$
%%%where $\varphi:\R\rightarrow\R$ is the \emph{transition function of Sotomayor-Teixeira}, that is, a smooth function satisfying that 
%%%%$\varphi(t)=1$ for $t\geq 1$, $\varphi(t)=-1$ for $t\leq -1,$ 
%%%$\varphi'(t)>0$ for $t\in(-1,1),$ $\varphi(\pm 1)=\pm 1$ and $\varphi^{(i)}(\pm1)=0$ for $i=1,2,\ldots,n$.
%%%
%%%Now, we define smoothing for piecewise smooth vector fields. 
%%%\begin{definition}
%%%The regularization of $F=(f^+,f^-)$ given in \eqref{sis1} is the smooth vector field given by
%%%\begin{equation}\label{regst}
%%%F^{\e}= \Big(\frac{1+\phi(\Re/\e)}{2}\Big) f^+ + \Big(\frac{1-\phi(\Re/\e)}{2}\Big)f^-
%%%\end{equation}
%%%for $0<\e<<1,$ where $\varphi$ satisfies the Definition \ref{reggeral}.
%%%\end{definition}
Notice that 
$$F^\e(z)\to\left\{\begin{array}{rcl}
 f^+(z)&\text{for}& \Re(z)>0,\\
 \\
 f^-(z)&\text{for}& \Re(z)<0.
 \end{array}\right.$$
%%% In particular, the Sotomayor-Teixeira regularization is smooth for $\e > 0$ and satisfies that 
%%%$F^{\e}=  f^+$ on $\{z\in\C:\Re(z)\geq\epsilon\}$ and $F^{\e}=  f^-$ on $\{z\in\C:\Re(z)\leq-\epsilon%%%%\}$. 

%In particular, the $C^n-$Sotomayor-Teixeira regularization (\emph{ST-regularization})   is
%the one parameter family $F^{\e}$ given by 
%\begin{equation}  \label{regst}
%F^{\e}= \Big(\frac{1+\varphi(\Re/\e)}{2}\Big) f^+ + \Big(\frac{1-\varphi(\Re/\e)}{2}\Big)f^-
%\end{equation} 
%where $\varphi:\R\rightarrow[-1,1]$ is a \emph{transition function}, that is, a smooth function satisfying that 
%$\varphi(t)=1$ for $t\geq 1$, $\varphi(t)=-1$ for $t\leq -1,$ $\varphi'(t)>0$ 
%for $t\in(-1,1),$ and $\varphi^{(i)}(\pm1)=0$ for $i=1,2,\ldots,n$.
%The regularization  is smooth for $\e > 0$ and satisfies that 
%$F^{\e}=  f^+$ on $\{z\in\C:\Re(z)\geq\epsilon\}$ and $F^{\e}=  f^-$ on $\{z\in\C:\Re(z)\leq-\epsilon\}$. 
Theorem \ref{teo:reg} still holds for the transition functions of Definition \ref{reggeral}. 
\begin{theorem}\label{teoreg} Let $\phi$ be a transition function satisfying the conditions of Definition \ref{reggeral}. If there exists some $\e>0$ such that the regularization \eqref{regst} is holomorphic 
then $f^+(z)=f^-(z)$ for all $z\in\C$.
\end{theorem}
\begin{proof}
%The result is an immediate consequence of the principle of identity of the analytical functions. 
%In fact, if two analytical functions coincide in an open subset, then they coincide throughout their domain.
Since $F^\e=u^\e+iv^\e$ is a holomorphic function, then its partial derivatives satisfy the Cauchy-Riemann equations:
\begin{equation}\label{cre}
u^\e_x=v^\e_y,\quad u^\e_y=-v^\e_x,\quad \forall z=x+iy\in\C.
\end{equation}
If $f^{+}=u_1+iv_1$, $f^-=u_2+iv_2$ then
 \begin{equation}\begin{array}{rcl}\label{cr4}
 u^\e_x &=&\frac{\phi'(\frac{x}{\e})}{4}u_1(x+iy)+\frac{1+\phi(\frac{x}{\e})}{2}\frac{\partial u_1}{\partial x}(x+iy)-\frac{\phi'(\frac{x}{\e})}{4}u_2(x+iy)+\frac{1-\phi(\frac{x}{\e})}{2}\frac{\partial u_2}{\partial x}(x+iy);\vspace{0.3cm}\\ 
 u^\e_y &=&\frac{1+\phi(\frac{x}{\e})}{2}\frac{\partial u_1}{\partial y}(x+iy)-\frac{1-\phi(\frac{x}{\e})}{2}\frac{\partial u_2}{\partial y}(x+iy);\vspace{0.3cm}\\ 
  v^\e_x &=&\frac{\phi'(\frac{x}{\e})}{4}v_1(x+iy)+\frac{1+\phi(\frac{x}{\e})}{2}\frac{\partial v_1}{\partial x}(x+iy)-\frac{\phi'(\frac{x}{\e})}{4}v_2(x+iy)+\frac{1-\phi(\frac{x}{\e})}{2}\frac{\partial v_2}{\partial x}(x+iy);\vspace{0.3cm}\\ 
 v^\e_y &=&\frac{1+\phi(\frac{x}{\e})}{2}\frac{\partial v_1}{\partial y}(x+iy)-\frac{1-\phi(\frac{x}{\e})}{2}\frac{\partial v_2}{\partial y}(x+iy).
 \end{array}\end{equation} 
 Since $f^+$ and $f^-$ are holomorphic functions, then its partial derivatives satisfy the Cauchy-Riemann equations:
 \begin{equation}
 \begin{array}{rcl}\label{cr2}
 (u_1)_x &=&(v_1)_y,\quad (u_1)_y=-(v_1)_x,\quad \forall z=x+iy\in\C;\\
 (u_2)_x &=&(v_2)_y,\quad (u_2)_y=-(v_2)_x,\quad \forall z=x+iy\in\C.
\end{array}\end{equation}
Thus, substituting \eqref{cr2} in \eqref{cr4} e using \eqref{cre} we get that 
$\frac{\phi'(\frac{x}{\e})}{4}(u_1(x+iy)-u_2(x+iy))=0$
and
$\frac{\phi'(\frac{x}{\e})}{4}(v_1(x+iy)-v_2(x+iy))=0$,
respectively. As $\phi'(\frac{x}{\e})>0$ for all $x\in(-\e,\e),$ then $f^+(x+iy)=f^-(x+iy)$ for all $x\in(-\e,\e)$ and $y\in\R.$ Therefore,  using the principle of identity of the analytic functions, we can conclude that $f^+(z)=f^-(z)$ for all $z\in\C.$
\end{proof} 

%Now, the above result can be generalized to a wider family of transition functions. Indeed, consider the transition function $\phi:\R\rightarrow[-1,1]$, which is a smooth function satisfying that $\phi'(t)>0$ for $t\in(-1,1)$ and $\lim_{t\to\pm\infty}\phi(t)=\pm 1$. Notice that 
%$$F^\e(z)\to\left\{\begin{array}{rcl}
% f^+(z)&\text{for}& \Re(z)>0\\
% \\
% f^-(z)&\text{for}& \Re(z)<0
% \end{array}\right.$$
%\textbf{Claim}: 

 %If $f^{+}=(u_1,v_1)$, $f^-=(u_2,v_2)$ then
 Now, the trajectories of the regularized system \eqref{regst} are the solutions of the slow--fast system
 \begin{equation}
 \e\dot{\bar{x}}= \dfrac{(1+\varphi(\bar{x})) u_1 + (1-\varphi(\bar{x}))u_2}{2},\quad
 \dot{y}= \dfrac{(1+\varphi(\bar{x}))v_1 + (1-\varphi(\bar{x}))v_2}{2},
  \label{singbasic}\end{equation}
 where $x=\e\bar{x}$. 
 
 We refer to the set 
$\mathcal{S}=\{(\bar{x},y):(1+\varphi(\bar{x}))u_1 + (1-\varphi(\bar{x}))u_2=0\}$
 as being the \textit{critical manifold}.   System \eqref{singbasic} when the parameter $\e$ is $0$ is called \textit{reduced system}.
 
 \begin{theorem}
 The sliding region $\Sigma^s$  is homeomorphic to the normally hyperbolic part of the critical manifold $\mathcal{S}$
 and the sliding vector field $F^{\Sigma}$ is topologically equivalent to 
 the reduced system. \label{teoBB}	
 	\end{theorem}
 See Figure \ref{figPR}.
 	\begin{figure}[h]
 		\begin{overpic}[width=3.5cm]{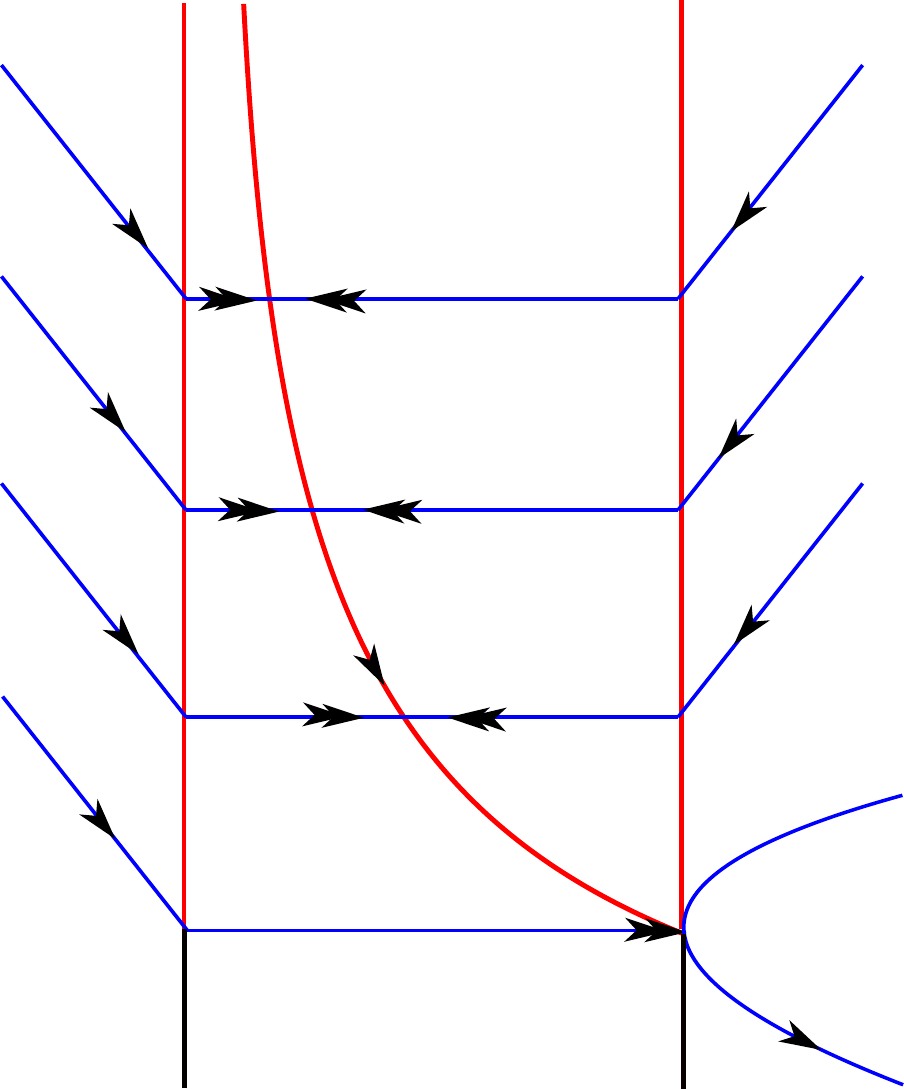}
 				%\begin{overpic}[grid,tics=10,width=4cm]{F1}
 						\put(75,77){$f^+$} 	\put(-3,77){$f^-$}	\put(23,40){$\mathcal{S}$}	\put(51,40){$\Sigma^s$}
 		\end{overpic}
 		\caption{\small{In the vertical range is drawn the phase portrait of the slow-fast system with  $\e=0,\bar{x} \in [-1,1], y \in\R.$ The red curve is the slow manifold $\mathcal{S}$ and the sliding region $\Sigma^s$. The double arrow represents the fast flow.}}
 		\label{figPR}
 	\end{figure}
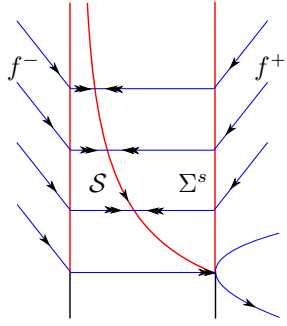

In \cite{BS} and \cite{kris}, asymptotic methods and blow-up methods were used to study $C^n$-regularizations of generic regular-fold singularities respectively.  Following \cite{GST}, these authors used the local normal form of Filippov systems, close to $\Sigma=\{x=0\}$, around a visible fold-regular singularity, which is given by $\widetilde{f^-}=(1,0)$ and $\widetilde{f^+}=(2y,1).$ Notice that $f^+=u_1+iv_1$ is not a holomorphic function because $(u_1)_y=2\neq 0=-(v_1)_x.$ 

%Likewise, the canonical form $f^-=(1,0)$ and $f^+=(\alpha y^{2k-1}+g(y)+xh(x,y),1)$ given by \cite{NR} for a contact of multiplicity k is not homomorphic. Indeed, if $f^+=(u_1,v_1),$ then $(u_1)_y\neq-(v_1)_x.$

In what follows, we are concerned in studying the regularization of PWHS around visible regular-fold singularities. For that we use the normal forms given in Proposition \ref{GGJ} and Theorem 1 of \cite{NR}.

Consider $\dot{z}^{+}=G(z)$, where $G$ is one of the following fields:
\begin{itemize}
\item[(i)] $G(z)=f'(p)(z-z_0)$, with $f'(p)=a+ib,$ $z_0=x_0+iy_0,$ $b<0,$ and $x_0>0$;
\item[(ii)] $G(z)=(z-z_0)^2$, with $x_0>0$;
\item[(iii)] $G(z)=\frac{1}{(z-z_0)^2}$, with $x_0<0$;
\item[(iv)]$G(z)=\frac{(z-z_0)^2}{1+(z-z_0)}$, with $0<x_0<1$.
\end{itemize}
The PWHS is given by
\begin{equation}\label{reg_sys_1}
\begin{aligned}
\left\{\begin{array}{l}
\dot{z}^{-}=1,\text{ when } \Re(z)<0, \\[5pt]
\dot{z}^{+}=G(z),\text{ when } \Re(z)>0.
\end{array} \right.
\end{aligned}
\end{equation}
In cartesian coordinates, we have 
\begin{equation}\label{carcoord_reg_1}
\begin{aligned}
\left\{\begin{array}{l}
(\dot{x}^{-},\dot{y}^{-})=(1,0),\text{ when } x<0, \\[5pt]
(\dot{x}^{+},\dot{y}^{+})=\widetilde{f^+}(x,y), \text{ when } x>0,
\end{array} \right.
\end{aligned}
\end{equation} 
where 
\begin{itemize}
\item[(i)] $\widetilde{f^+}(x,y)=(a(x-x_0)-b(y-y_0),b(x-x_0)+a(y-y_0))$,
\item[(ii)] $\widetilde{f^+}(x,y)=((x-x_0)^2-(y-y_0)^2,2(x-x_0)(y-y_0))$,
\item[(iii)] $\widetilde{f^+}(x,y)=\left(\dfrac{(x-x_0+y-y_0)(x-x_0-y+y_0)}{((x-x_0)^2+(y-y_0)^2)^2},-\dfrac{2(x-x_0)(y-y_0)}{((x-x_0)^2+(y-y_0)^2)^2}\right)$, and
\item[(iv)] $\begin{array}{rcl}
\widetilde{f^+}(x,y)&=&\left(\frac{x^3+x^2(1-3x_0)+x_0^2-x_0^3+x(-2x_0+3x_0^2+(y-y_0)^2)-(1+x_0)(y-y_0)^2}{1+x^2-2x(x_0-1)-2x_0+x_0^2+(y-y_0)^2}\right.,\\
& &\left.+\frac{(x^2-2x(x_0-1)-2x_0+x_0^2+(y-y_0)^2)(y-y_0)}{1+x^2-2x(x_0-1)-2x_0+x_0^2+(y-y_0)^2}\right),\end{array}$
\end{itemize}
 respectively. Notice that 
\begin{itemize}
\item[(i)] $p=(0,y_0-\frac{a}{b}x_0)$,
\item[(ii)] $p=(0,y_0-x_0)$,
\item[(iii)] $p=(0,y_0+x_0)$, and
\item[(iv)] $p=\left(0,\frac{-\sqrt{x_0^2-x_0^4}+y_0+x_0y_0}{1+x_0}\right)$
\end{itemize} 
 are visible regular-fold singularities, respectively.
In what follows we study the dynamics of the regularized system of \eqref{carcoord_reg_1} around $p.$ For that, we consider the translations:
\begin{itemize}
\item[(i)] $\hat{x}=x$ and $\hat{y}=y-y_0+\frac{a}{b}x_0,$
\item[(ii)] $\hat{x}=x$ and $\hat{y}=y-y_0+x_0,$ 
\item[(iii)] $\hat{x}=x$ and $\hat{y}=y-y_0-x_0,$ and
\item[(iv)] $\hat{x}=x$ and $\hat{y}=y-\frac{-\sqrt{x_0^2-x_0^4}+y_0+x_0y_0}{1+x_0},$
\end{itemize}
 respectively. Then, the vector field $\widetilde{f^+}$ at the coordinates $(\hat{x},\hat{y})$ is given by 
 \begin{itemize}
 \item[(i)] $\widehat{f^+}(\hat{x},\hat{y})=\left(a(\hat{x}-x_0)-b\left(\hat{y}-\frac{a}{b}x_0\right),b(\hat{x}-x_0)+a\left(\hat{y}-\frac{a}{b}x_0\right)\right),$
 \item[(ii)] $\widehat{f^+}(\hat{x},\hat{y})=\left((\hat{x}-x_0)^2-(\hat{y}-x_0)^2,2(\hat{x}-x_0)(\hat{y}-x_0)\right),$ 
 \item[(iii)] $\widehat{f^+}(\hat{x},\hat{y})=\left(\dfrac{(\hat{y}+\hat{x})(-\hat{y}+\hat{x}-2x_0)}{((\hat{x}-x_0)^2+(\hat{y}+x_0)^2)^2},\dfrac{-2(\hat{x}-x_0)(\hat{y}+x_0)}{((\hat{x}-x_0)^2+(\hat{y}+x_0)^2)^2}\right),$ and
 \item[(iv)] $\begin{array}{rl}
 \widehat{f^+}(\hat{x},\hat{y})=&\left(-\frac{-2xx_0+\hat{y}^2(-1+x-x_0)(1+x_0)+x(1+x_0)(x+x^2-3xx_0+2x_0^2)+2\hat{y}(1-x+x_0)\sqrt{x_0^2-x_0^4}}{-1+x_0-\hat{y}^2(1+x_0)-x(2+x-2x_0)(1+x_0)+2\hat{y}\sqrt{x_0^2-x_0^4}}\right.,\\
 &\left.\frac{(\hat{y}+\hat{y}x_0-\sqrt{x_0^2-x_0^4})(2x_0-\hat{y}^2(1+x_0)-x(2+x-2x_0)(1+x_0)+2\hat{y}\sqrt{x_0^2-x_0^4})}{(1+x_0)(-1+x_0-\hat{y}^2(1+x_0)-x(2+x-2x_0)(1+x_0)+2\hat{y}\sqrt{x_0^2-x_0^4)}}\right),\end{array}$
 \end{itemize}
respectively. Recall that $\hat{p}=(0,0)$ is a visible regular-fold singularity of $\widehat{f^+}$. Now, since 
\begin{itemize}
\item[(i)] $\widehat{f_2^+}(\hat{p})=-\frac{(a^2+b^2)}{b}x_0>0,$
\item[(ii)] $\widehat{f_2^+}(\hat{p})=2x_0^2>0,$ 
\item[(iii)] $\widehat{f_2^+}(\hat{p})=\frac{1}{2x_0^2}>0,$ and
\item[(iv)]  $\widehat{f_2^+}(\hat{p})=-\frac{2x_0\sqrt{x_0^2-x_0^4}}{(x_0-1)(1+x_0)}>0,$
\end{itemize}
then there exists a neighborhood $\mathcal{U}$ of $\hat{p}$, such that $\widehat{f_2^+}(\hat{x},\hat{y})>0,$
for all $(\hat{x},\hat{y})\in \mathcal{U}$. Performing a time rescaling in $\widehat{f^+},$ we get $\widecheck{f^+}(\hat{x},\hat{y})=(f(\hat{x},\hat{y}),1),$ where 
\begin{itemize}
\item[(i)] $f(\hat{x},\hat{y})=\dfrac{a(\hat{x}-x_0)-b(\hat{y}-\frac{a}{b}x_0)}{b(\hat{x}-x_0)+a(\hat{y}-\frac{a}{b}x_0)},$
\item[(ii)] $f(\hat{x},\hat{y})=\dfrac{(\hat{x}-x_0)^2-(\hat{y}-x_0)^2}{2(\hat{x}-x_0)(\hat{y}-x_0)},$ 
\item[(iii)] $f(\hat{x},\hat{y})=\dfrac{(\hat{y}+\hat{x})(\hat{y}-\hat{x}+2x_0)}{2(\hat{x}-x_0)(\hat{y}+x_0)},$ and
\item[(iv)] $f(\hat{x},\hat{y})=-\frac{(1+x_0)(-2xx_0+\hat{y}^2(-1+x-x_0)(1+x_0)+x(1+x_0)(x+x^2-3xx_0+2x_0^2)+2\hat{y}(1-x+x_0)\sqrt{x_0^2-x_0^4}}{(\hat{y}+\hat{y}x_0-\sqrt{x_0^2-x_0^4})(2x_0-\hat{y}^2(1+x_0)-x(2+x-2x_0)(1+x_0)+2\hat{y}\sqrt{x_0^2-x_0^4})},$  
\end{itemize}
respectively. Notice that $\widehat{f^+}$ and $\widecheck{f^+}$ have the same orbits in $\mathcal{U}$ with the same orientation. 

Now, expanding $f$ around $(\hat{x},\hat{y})=(0,0),$ we get
$$f(\hat{x},\hat{y})=\alpha\hat{y}+g(\hat{y})+\hat{x}\vartheta(\hat{x},\hat{y}),$$
where 
\begin{itemize}
\item[(i)] $\alpha=\frac{b^2}{(a^2+b^2)x_0},$ $g(\hat{y})=\frac{ab^3}{(a^2+b^2)^2x_0^2}\hat{y}^2+\CO(\hat{y}^3)$ and $\vartheta(\hat{x},\hat{y})=\frac{-ab}{(a^2+b^2)x_0}+\CO(\hat{x},\hat{y}),$
\item[(ii)] $\alpha=\frac{1}{x_0},$ $g(\hat{y})=\frac{1}{2x_0^2}\hat{y}^2+\CO(\hat{y}^3)$, and $\vartheta(\hat{x},\hat{y})=-\frac{1}{x_0}+\CO(\hat{x},\hat{y}),$ and
\item[(iii)] $\alpha=-\frac{1}{x_0},$ $g(\hat{y})=\frac{1}{2x_0^2}\hat{y}^2+\CO(\hat{y}^3)$, and $\vartheta(\hat{x},\hat{y})=-\frac{1}{x_0}+\CO(\hat{x},\hat{y}),$ and
\item[(iv)] $\alpha=\frac{(1+x_0)^2}{x_0},$ $g(\hat{y})=\frac{(1+x_0)^3(1+2(x_0-1)x_0)}{2x_0\sqrt{x_0^2-x_0^4}}\hat{y}^2+\CO(\hat{y}^3)$, and $\vartheta(\hat{x},\hat{y})=\frac{(1+x_0)(-1+x_0+x_0^2)}{\sqrt{x_0^2-x_0^4}}+\CO(\hat{x},\hat{y}),$
\end{itemize}
 respectively. Now, using Theorem 1 of \cite{NR}, we get the following result.  
%\begin{itemize}
%\item[{\bf (A)}]  $\widecheck{f^+}$ has a visible  $2$-multiplicity contact with $\Sigma$ at $\hat{q},$ $\widecheck{f_2^+}(\hat{q})>0,$ and there exists a neighborhood $\mathcal{U}\subset\R^2$ of $\hat{q}$ such that $f^-\big|_\mathcal{U}=(1,0)$ and $\Sigma\cap \mathcal{U}=\{(0,\hat{y}):\, \hat{y}\in(-\hat{y}_\mathcal{U},\hat{y}_\mathcal{U}\}.$
%\end{itemize}
\begin{theorem}\label{ta}
Consider system \eqref{carcoord_reg_1}, i.e. $\widetilde{f^+}$ has a visible fold singularity at $p=(0,p^*),$ $\widetilde{f_2^+}(p)>0,$ and $\widetilde{f^-}=(1,0).$ For $n\geqslant 2,$ consider the regularized system $F^{\e}$ \eqref{regst}. Then, there exist $\rho_0,\theta_0>0,$ and constants  $\beta<0$ and $c,r>0,$ such that for every $\rho\in(\e^\la,\rho_0],$ $\theta\in[y_\e,\theta_0],$ $\la\in(0,\la^*),$ with $\la^*= \frac{n}{2n-1},$ $q=1-\dfrac{\lambda}{\lambda^*}\in(0,1),$ and $\e>0$ sufficiently small, the flow of $F^{\e}$ defines a map $U_{\e}$ between the transversal sections $\widehat V_{\rho,\lambda}^{\e}=[\e,x_{\rho,\lambda}^{\e}]\times\{-\rho+p^*\}$ and $\widetilde V_{\theta}^{\e}=[x_\theta^\e,x_\theta^\e+r e^{-\frac{c}{\e^q}}]\times\{\theta+p^*\},$ satisfying
\[
\begin{array}{cccl}
U_{\e}:& \widehat V_{\rho,\lambda}^{\e}& \longrightarrow& \widetilde V_{\theta}^{\e}\\
&x&\longmapsto&x_{\theta}^{\e}+\CO(e^{-\frac{c}{\e^q}}),
\end{array}
\]  
where 
\[\begin{array}{rcl}
x_{\theta}^{\e}& =&\frac{\alpha\theta^2}{2}+\CO(\theta^3)+\e+\mathcal{O}(\e\theta)+\mathcal{O}(\theta^{2} y_\e)+\mathcal{O}(y_{\e}^{2}), \quad\text{and} \\
y_\e&=&\e^{\lambda^*}\eta+\CO(\e^{\lambda^*+\frac{1}{2n-1}}), \quad\text{for}\quad \eta>0, \\
x^\e_{\rho,\la}&=&\frac{\alpha\rho^2}{2}+\CO(\rho^3)+\e+\mathcal{O}(\e \rho)+\beta \e^{2\la}+\mathcal{O}(\e^{3\la})+\mathcal{O}(\e^{1+\la}).
\end{array} \] (see Figure \ref{figMAP1}).
\end{theorem}

\begin{figure}[h]
	\begin{center}
		\begin{overpic}[scale=0.52]{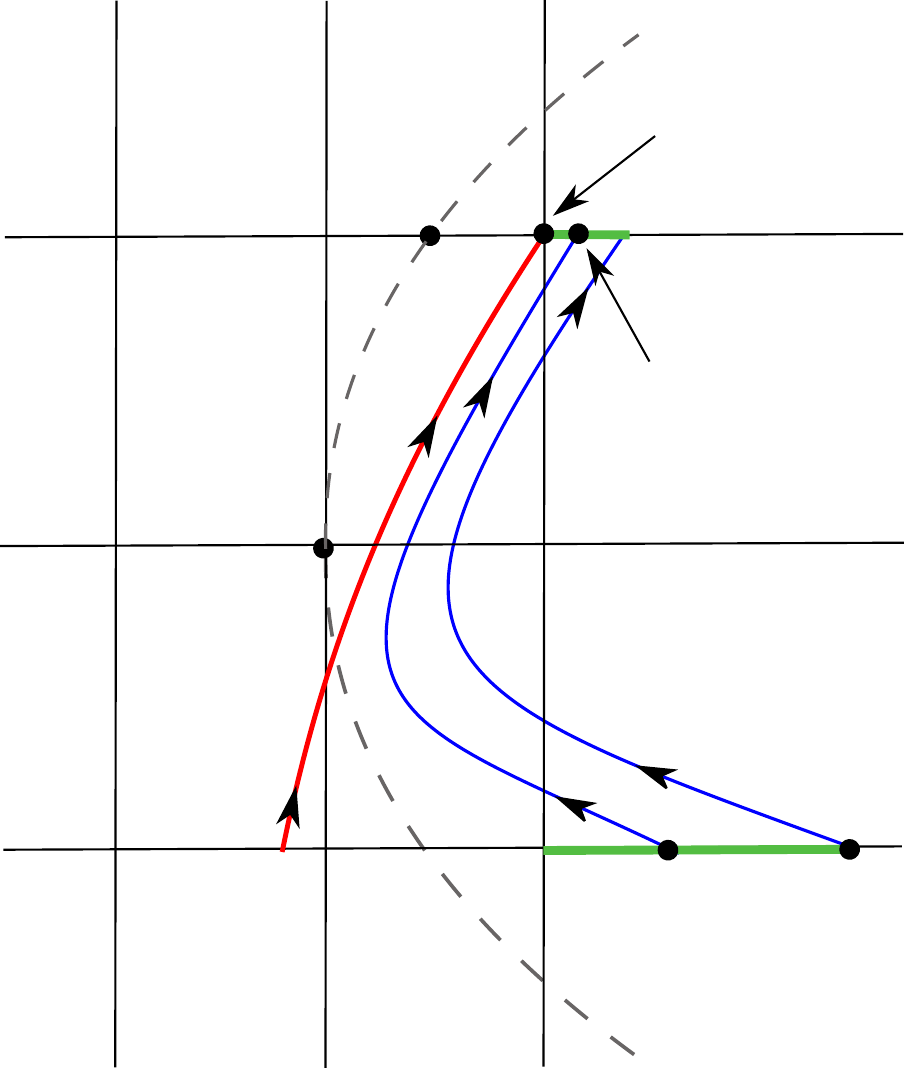}
		%\begin{overpic}[grid,tics=5,width=6cm]{poincaremap1.pdf}
		\put(65,12){$\widehat V_{\rho,\la}^{\e}$}
		\put(62,80.5){$\widetilde V_{\theta}^{\e}$}
		\put(86,78){$y=\theta+p^*$}
		\put(86,20){$y=-\rho+p^*$}
		\put(28,-5){$\Sigma$}
		\put(26,51){$p$}
		\put(44,-5){$x=\e$}
		\put(4,-5){$x=-\e$}
		\put(62,89){$(\theta,x_{\theta}^{\e})$}
		%\put(34,81){$\overline{x}_{\theta}$}
		\put(55,61){$U_{\e}(x)$}
		\put(61,15){$x$}
		\put(77,15){$x^\e_{\rho,\la}$}
		\end{overpic}
	\end{center}	
	\bigskip	
	\caption{The Transition Map $U_{\e}$ of $F^\e.$The dotted curve is the trajectory of $\widetilde{f^+}$ passing through the visible regular-fold singularity. The red curve is the Fenichel manifold.}
	\label{figMAP1}
	\end{figure} 

\section{Limit cycles of PWHS}\label{sec:limitcycles}
In Proposition \ref{teo_nocl} was shown that holomorphic systems have no limit cycles, however it is possible to prove that PWHS have limit cycles. For that reason, in this section we focus on finding the conditions for the existence of limit cycles of the PWHS, which are formed by the normal forms given in Proposition \ref{GGJ}. 

We start by studying the linear case. In this case, we consider equilibrium points on manifold $\Sigma=\{x=0\}.$
\begin{theorem}\label{cl1}
The piecewise linear holomorphic systems whose equilibrium points are on manifold $\Sigma$ have at most one limit cycle.
\end{theorem}
%\begin{theorem}
%	Let $b,c,d$ and $x_0$ be non-zero real numbers and $a\in\R$. The holomorphic piecewise system
%	\begin{equation}\label{ex1_cycle}
%\begin{aligned}
%\left\{\begin{array}{l}
%\dot{z}^{+}=(a+ib)(z-x_0),\text{ when } \Im{(z)}>0, \\[5pt]
%\dot{z}^{-}=(c+id)z, \text{ when } \Im{(z)}< 0,
%\end{array} \right.
%\end{aligned}
%\end{equation}
%	has a unique limit cycle $\Gamma$ if, and only if, $a,b,c,d$ and $x_0$ satisfy any row in table \eqref{table0} and $\frac{a}{b}+\frac{c}{d}\neq 0$. Moreover, $\Gamma$ is stable (resp. unstable) provides that $\operatorname{sgn}(b)\neq \operatorname{sgn}(\frac{a}{b}+\frac{c}{d})$ (resp. $\operatorname{sgn}(b)=\operatorname{sgn}(\frac{a}{b}+\frac{c}{d})$).%This Limit cycle is stable (resp. unstable), when $\operatorname{sgn}(a)\neq \operatorname{sgn}(b)$ (resp. $\operatorname{sgn}(a)=\operatorname{sgn}(b)$).%$a<0<b$ and $d>0,$
%	\begin{equation}\label{table0}
%\begin{array}{|| c |c| c | c|c |c||}
%\hline
%a&	b &c & d &x_0 \\
%\hline\hline
%+&+	&+ &+&+ \\
%\hline
%-&+	&- &+&-\\
%\hline
%-&-	&- &-&+ \\
%\hline
%+&-	&+ &-&- \\
%\hline
%%%%%%%%%%%%%%%%%%%%5
%+&+	&- &+&\operatorname{sgn}(x_0)=\operatorname{sgn}(\frac{a}{b}+\frac{c}{d}) \\
%\hline
%-&+	&+ &+&\operatorname{sgn}(x_0)=\operatorname{sgn}(\frac{a}{b}+\frac{c}{d})\\
%\hline
%-&-	&+ &-&\operatorname{sgn}(x_0)=\operatorname{sgn}(\frac{a}{b}+\frac{c}{d}) \\
%\hline
%+&-	&- &-&\operatorname{sgn}(x_0)=\operatorname{sgn}(\frac{a}{b}+\frac{c}{d}) \\
%\hline
%%%%%%%%%%%%%%%%%%
%0&+	&+ &+&+  \\
%\hline
%0&+&- &+&- \\
%\hline
%0&-&+ &-&-\\
%\hline
%0&-	&- &-&+  \\
%\hline
%\end{array}
%\end{equation}
%\end{theorem}

\begin{proof} Without loss of generality suppose that the piecewise linear holomorphic system has one of its equilibrium points at the origin. Thus, this system can be written as follows:
	\begin{equation}\label{ex1_cycle}
\begin{aligned}
\left\{\begin{array}{l}
\dot{z}^{+}=(a+ib)(z-x_0),\text{ when } \Im{(z)}>0, \\[5pt]
\dot{z}^{-}=(c+id)z, \text{ when } \Im{(z)}< 0,
\end{array} \right.
\end{aligned}
\end{equation}
where $a,b,c,d$ and $x_0$ are real numbers. It is easy to verify that if any of the coefficients $b,d$ or $x_0$ are zero then the system has no limit cycles. So we assume that these coefficients are not zero. Consider $w_0\in\R^+.$ 
If $b>0,$ then \[z^+(t)=(w_0-x_0)e^{at}(\cos (bt)+i\sin (bt))+x_0\]
is a solution of $\dot{z}^{+}=(a+bi)(z-x_0)$ satisfying that $z^+(0)=w_0$ and $z^{+}(\frac{\pi}{b})=x_0-(w_0-x_0)e^{\frac{a\pi}{b}}$. In addition, 
\[z^-(t)=-e^{ct}((w_0-x_0)e^\frac{a\pi}{b}-x_0)(\cos(dt)+i\sin(dt))\]
is a solution of $\dot{z}^{-}=(c+id)z$ such that $z^{-}(0)=x_0-(w_0-x_0)e^{\frac{a\pi}{b}}$ and $z^{-}(\frac{\pi}{d})=e^\frac{c\pi}{d}((w_0-x_0)e^\frac{a\pi}{b}-x_0),$ with $d>0$.
On the other hand, if $b<0,$ then \[z^+(t)=-(w_0+x_0)e^{at}(\cos (bt)+i\sin (bt))+x_0\]
is a solution of $\dot{z}^{+}=(a+bi)(z-x_0)$ satisfying that $z^+(0)=-w_0$ and $z^{+}(-\frac{\pi}{b})=x_0+(w_0+x_0)e^{-\frac{a\pi}{b}}$. Moreover, 
\[z^-(t)=e^{ct}((w_0+x_0)e^{-\frac{a\pi}{b}}+x_0)(\cos(dt)+i\sin(dt))\]
is a solution of $\dot{z}^{-}=(c+id)z$ such that $z^{-}(0)=x_0+(w_0+x_0)e^{-\frac{a\pi}{b}}$ and $z^{-}(-\frac{\pi}{d})=e^{-\frac{c\pi}{d}}((-w_0-x_0)e^{-\frac{a\pi}{b}}-x_0),$ with $d<0$.

\begin{figure}[h]
	\begin{center}
		\begin{overpic}[scale=0.8]{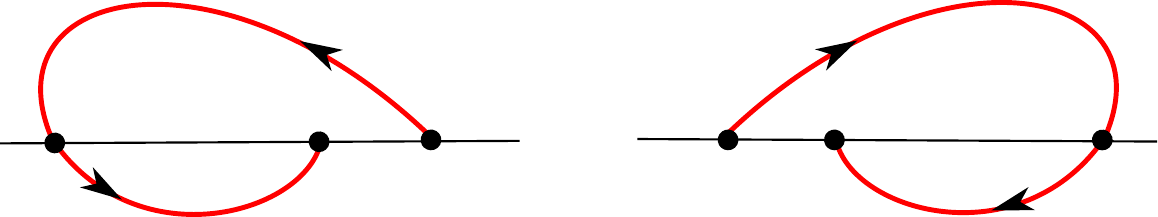}
		%\begin{overpic}[grid,tics=5,width=10cm]{frmap.pdf}		
        \put(35,3.5){$w_0$}
        \put(59,3.5){$-w_0$}
            \put(22,8.5){$\Pi(w_0)$}
        \put(70,8.5){$\Pi(-w_0)$}
		\put(103,6){$\Sigma$}
		\put(47,6){$\Sigma$}
		\put(14,-4.5){$b,d>0$}
		\put(75,-4.5){$b,d<0$}
		\end{overpic}
		\caption{The Poincaré map around $z=\pm w_0$.}
	\label{frmap}
	\end{center}
	\end{figure}
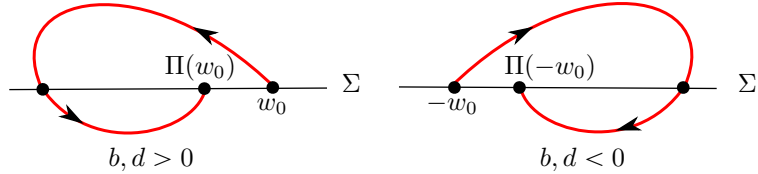

Therefore, the Poincaré map around $z=\pm w_0$ is given by 
$$\Pi(z)=e^{\pm\frac{c\pi}{d}}((z-x_0)e^{\pm\frac{a\pi}{b}}-x_0)$$ and $\Pi'(\pm w_{0})=e^{\pm(\frac{a}{b}+\frac{c}{d})\pi}$ (see Figure \ref{frmap}). Now, we must seek solutions for the equation $\Pi(\pm w_{0})=\pm w_{0}$. The number of roots of this equations correspond to the number of limit cycles. If $\frac{a}{b}+\frac{c}{d}=0,$ then $\Pi(\pm w_{0})=\pm w_{0}$ has no solution. Otherwise, we have a unique solution given by $w_0=\dfrac{e^{\frac{c\pi}{d}}(1+e^{\frac{a\pi}{b}})x_0}{-1+e^{(\frac{a}{b}+\frac{c}{d})\pi}}$ provided that $b>0$ and $w_0=\dfrac{(1+e^{\frac{a\pi}{b}})x_0}{-1+e^{(\frac{a}{b}+\frac{c}{d})\pi}}$ provided that $b<0$, thus we have a unique limit cycle $\Gamma$. Finally, using the first derivative of the Poincaré map, we can conclude that $\Gamma$ is stable (resp. unstable) provides that $\operatorname{sgn}(b)\neq \operatorname{sgn}(\frac{a}{b}+\frac{c}{d})$ (resp. $\operatorname{sgn}(b)=\operatorname{sgn}(\frac{a}{b}+\frac{c}{d})$).% Otherwise, if $d\neq 0$ (resp. $d=0$), then we have no periodic orbits (resp. we have infinite periodic orbits).
\end{proof}
An immediate consequence of the proof of the previous theorem is the following result.
\begin{corollary}\label{cor:lc}
	Let $b,d,$ and $x_0$ be non-zero real numbers and $a,c\in\R$. The  piecewise linear holomorphic system
	\begin{equation}\label{ex1_cycle}
\begin{aligned}
\left\{\begin{array}{l}
\dot{z}^{+}=(a+ib)(z-x_0),\text{ when } \Im{(z)}>0, \\[5pt]
\dot{z}^{-}=(c+id)z, \text{ when } \Im{(z)}< 0,
\end{array} \right.
\end{aligned}
\end{equation}
	has a unique limit cycle $\Gamma$ if, and only if, $a,b,c,d,$ and $x_0$ satisfy any row in tables \eqref{table0} and \eqref{table00} and $\frac{a}{b}+\frac{c}{d}\neq 0$. Moreover, $\Gamma$ is stable (resp. unstable) provides that $\operatorname{sgn}(b)\neq \operatorname{sgn}(\frac{a}{b}+\frac{c}{d})$ (resp. $\operatorname{sgn}(b)=\operatorname{sgn}(\frac{a}{b}+\frac{c}{d})$).%This Limit cycle is stable (resp. unstable), when $\operatorname{sgn}(a)\neq \operatorname{sgn}(b)$ (resp. $\operatorname{sgn}(a)=\operatorname{sgn}(b)$).%$a<0<b$ and $d>0,$
\end{corollary}	
	\begin{minipage}[t]{.58\textwidth}
\raggedright
    \begin{equation}\label{table0}
\begin{array}{|| c |c| c | c|c |c||}
\hline
a&	b &c & d &x_0 \\
\hline\hline
+&+	&+ &+&+ \\
\hline
-&+	&- &+&-\\
\hline
-&-	&- &-&+ \\
\hline
+&-	&+ &-&- \\
\hline
%%%%%%%%%%%%%%%%%%%5
+&+	&- &+&\operatorname{sgn}(x_0)=\operatorname{sgn}(\frac{a}{b}+\frac{c}{d}) \\
\hline
-&+	&+ &+&\operatorname{sgn}(x_0)=\operatorname{sgn}(\frac{a}{b}+\frac{c}{d})\\
\hline
-&-	&+ &-&\operatorname{sgn}(x_0)=\operatorname{sgn}(\frac{a}{b}+\frac{c}{d}) \\
\hline
+&-	&- &-&\operatorname{sgn}(x_0)=\operatorname{sgn}(\frac{a}{b}+\frac{c}{d}) \\
\hline
\end{array}
\end{equation}
\end{minipage}
\begin{minipage}[t]{.35\textwidth}
\raggedright
	\begin{equation}\label{table00}
\begin{array}{|| c |c| c | c|c |c||}
\hline
a&	b &c & d &x_0 \\
\hline\hline
0&+	&+ &+&+  \\
\hline
0&+&- &+&- \\
\hline
0&-&+ &-&-\\
\hline
0&-	&- &-&+  \\
\hline
%%%%%%%%%%%%%%%%%%%%%%5
+&+	&0 &+&+  \\
\hline
-&+&0 &+&- \\
\hline
-&-&0 &-&+\\
\hline
+&-	&0 &-&-  \\
\hline
\end{array}
\end{equation}
\end{minipage}\vspace{0.5cm}

It is important to emphasize that the conditions given in the previous theorem are not empty. Indeed, taking $a=-1$ or $a=0,$ $b=1,$ $c=-1,$ $d=1,$ and $x_0=-1$ we have the existence of a unique stable limit cycle (see Figure \ref{limit_cycle_1}).
\begin{figure}[h]
	\begin{center}
		\begin{overpic}[scale=0.4]{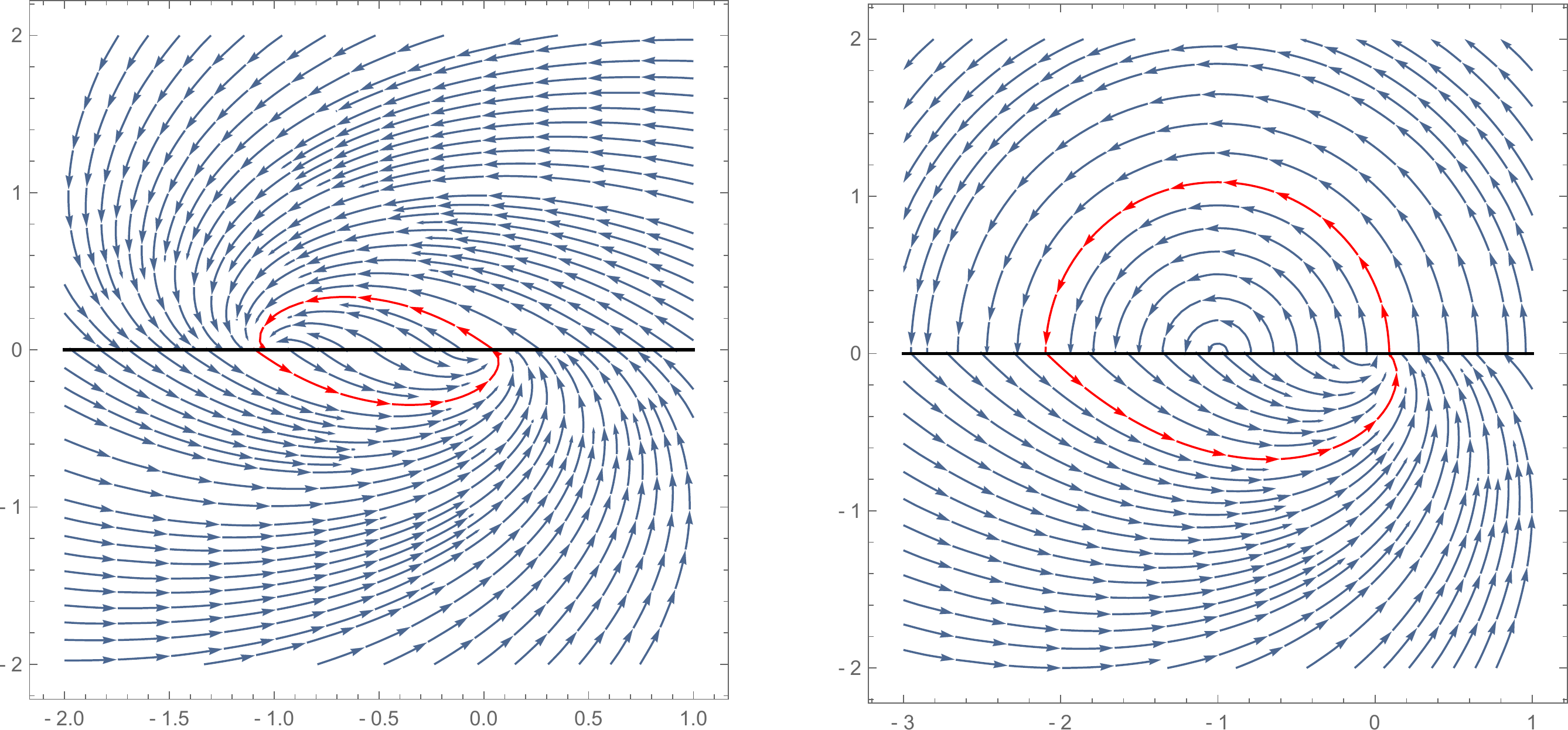}
		%\begin{overpic}[grid,tics=5,width=10cm]{fig_focus_focus1.pdf}		
        \put(-4,13){$\Sigma^-$}
        \put(-4,35){$\Sigma^+$}
		\put(101,24){$\Sigma$}
		 \put(50,13){$\Sigma^-$}
        \put(50,35){$\Sigma^+$}
		\put(47,24){$\Sigma$}
		\put(19,-2){$a=-1$}
		\put(74,-2){$a=0$}
		\end{overpic}
		\caption{Phase portrait of PWHS \eqref{ex1_cycle} with $b=1$, $c=-1,$ $d=1,$ and $x_0=-1$. The red  trajectory is the limit cycle of \eqref{ex1_cycle}.}
	\label{limit_cycle_1}
	\end{center}
	\end{figure}
%	\begin{remark}
%	There are examples of piecewise linear systems that have more than one limit cycle. For more details see, for instance, \cite{math8050755,MR3360760,MR3328261}.
%	\end{remark}
%\begin{remark}
%Recall that the PWHS \eqref{ex1_cycle} with $\frac{a}{b}+\frac{c}{d}=0$ has no limit cycles. Even more, if $\frac{a}{b}+\frac{c}{d}=0$ and $x_0\neq 0$ (resp. $\frac{a}{b}+\frac{c}{d}=0$ and $x_0=0$), then we have no periodic orbits (resp. we have infinite periodic orbits).
%\end{remark}
%Writing the system in Cartesian coordinates, we have
%\begin{equation}\label{sys1}
%	\begin{aligned}
%	\left\{\begin{array}{l}
%	(\dot{x}^{-},\dot{y}^{-})=(ax-by,bx+ay),\text{ when } y<0, \\[5pt]
%	(\dot{x}^{+},\dot{y}^{+})=(-y,x-d), \text{ when } y> 0,
%	\end{array} \right.
%	\end{aligned}
%	\end{equation}
%Solving $(\dot{x}^{-},\dot{y}^{-})$ with initial conditions $x^{-}(0)=x_{0}$, $y^{-}(0)=0$, we obtain $r^{-}(t)=x_{0}e^{-t}$, $\theta^{-}(t)=t$. So, for $t=\pi$, we have $r^{-}(\pi)=x_{0}e^{-\pi}$, $\theta^{-}(\pi)=\pi$. The next step is to solve  $(\dot{r}^{+},\dot{\theta}^{+})$ with initial conditions $r^{+}(0)=x_{0}e^{-\pi}$, $\theta^{+}(0)=\pi$. The solution evaluate in $\pi$ is given by $r^{+}(\pi) =-2+x_{0}\cosh(\pi)-x_{0}\sinh(\pi), \theta^{+}(\pi) = 2\pi$. Therefore, the Poincaré map is given by $\Pi(x_{0})=-2+x_{0}\cosh(\pi)-x_{0}\sinh(\pi)$. Now, we must seek solutions for the equation $\Pi(x_{0})-x_{0}$. The number of roots of this equations correspond to the number of limit cycles. As we have only one solution, given by $\dfrac{2}{e^{\pi}-1}$, we have only one limit cycle.	

Now, we study the analytical vector fields $z^n$ for $n\geq 2$, which are divided into 5 cases that depend on $n$. For that, we use the symmetry of this normal form and that the rays $\frac{j\pi}{n-1}$, $j=\{1,\cdots, 2(n-1)\}$ (resp. $\frac{j\pi}{2(n-1)}$, $j=\{1,3,\cdots, 4(n-1)-1\}$) are invariant by the flow of the equation $\dot{z}=z^n,$ with $n$ even (resp. $\dot{z}=iz^n$, with $n$ odd). Moreover, we  consider virtual equilibrium points of $z^n$ in the following sense: given a piecewise smooth vector field
\begin{equation}
\begin{aligned}
\left\{\begin{array}{l}
\dot{z}^{+}=f^{+}(z), \text{ when } \Im(z)> 0,\\[5pt]
\dot{z}^{-}=f^{-}(z),\text{ when }\Im(z)<0,
\end{array} \right.
\end{aligned}
\end{equation}
we say that a equilibrium point $z_0$ of $f^+$ (resp. $f^-$) is virtual when $z_0\in\Sigma^-$ (resp. $z_0\in\Sigma^+$).

%Note that $z^n $ at $e^{\frac{i\pi}{n-1}} $ results $e^{\frac{in\pi}{n-1}}$. This implies the ray  $\frac{\pi}{n-1}$ is invariant by the flow of the equation. 
%It is not hard to see that there will be $n-1$ rays 
%where the  solutions tends to origin for $ t\rightarrow+\infty$  and $n-1$ rays where the  solutions tends to origin for $ t\rightarrow-\infty.$
%We can use the polar blow-up to see that the phase portrait  in a neighborhood 
%of $0$ is a union of $2(n-1)$ elliptic sectors, and so the index of $0$ is $n$.  In fact,
%\[\dot{z}=z^n,\quad z=re^{i\theta} \implies \dot{r}=r^n\cos (n-1)\theta,\quad \dot{\theta}=r^{n-1}\sin (n-1)\theta,\]
%and performing a rescalling of time we get
%\[\dot{r}=r\cos (n-1)\theta,\quad \dot{\theta}=\sin (n-1)\theta\]
%and then
%\[  \int\frac{dr}{r} =\int \cot (n-1)\theta d\theta.  \]
%Thus
%\[ \ln r=\frac{1}{n-1}\ln|\sin (n-1)\theta| +C\]
%and
%\[   r=|\sin(n-1)\theta|^{\frac{1}{n-1}}e^C.    \]

\begin{theorem}\label{cl2}
	Given $n\in\N_{n\geq2}$, there exist $a,b,d,$ and $y_0$  non-zero real numbers and $z_0=x_0+iy_0\in\C$ satisfying table \eqref{table1}, such that the PWHS
	\begin{equation}\label{2_ex_cycle}
\begin{aligned}
\left\{\begin{array}{l}
\dot{z}^{+}=i^m(z+z_0)^n,\text{ when } \Im{(z)}>0, \\[5pt]
\dot{z}^{-}=(a+ib)(z-d), \text{ when } \Im{(z)}< 0,
\end{array} \right.
\end{aligned}
\end{equation}
	has a unique stable limit cycle, where $m=0$ if $n$ is even and $m=1$ otherwise.
\end{theorem}
\begin{equation}\label{table1}
\begin{array}{|| c |c| c | c | c | c | c |c||}
\hline
n&	k&a &b & d & x_0&y_0& \text{Main condition} \\
\hline\hline
2&	&- &+ & \R & \R&+& d>-x_0 \\
\hline
4k-1&\geq 1&- & -& - &  0  &+ & \cot\left(\dfrac{n\pi}{2(n-1)}\right)y_0<-\dfrac{d(1+e^\frac{a\pi}{b})}{1-e^{\frac{a\pi}{b}}}<0 \\
\hline
4k&\geq 1	& - & - & -&0 & + & \cot\left(\dfrac{n\pi}{2(n-1)}\right)y_0<-\dfrac{d(1+e^\frac{a\pi}{b})}{1-e^{\frac{a\pi}{b}}}<0\\
\hline
4k-2&>1	& - & + & + &0& + &0<\dfrac{d(1+e^\frac{a\pi}{b})}{1-e^{\frac{a\pi}{b}}}<\cot\left(\dfrac{(n-2)\pi}{2(n-1)}\right)y_0\\
\hline
4k+1&\geq 1	& - & + & + &0& + &0<\dfrac{d(1+e^\frac{a\pi}{b})}{1-e^{\frac{a\pi}{b}}}<\cot\left(\dfrac{(n-2)\pi}{2(n-1)}\right)y_0\\
\hline
\end{array}
\end{equation}
The proof of this theorem is an immediate consequence of the following 5 propositions.
\begin{proposition}
	Let $a$ and $b$ be non-zero real numbers, $d\in\R,$ and $z_0=x_0+iy_0,$ with $y_0>0$. If $a<0<b$ and $d>-x_0,$ then the PWHS
	\begin{equation}\label{ex2_cycle}
\begin{aligned}
\left\{\begin{array}{l}
\dot{z}^{+}=(z+z_0)^2,\text{ when } \Im{(z)}>0, \\[5pt]
\dot{z}^{-}=(a+ib)(z-d), \text{ when } \Im{(z)}< 0,
\end{array} \right.
\end{aligned}
\end{equation}
	has a unique stable limit cycle.
\end{proposition}
\begin{proof}
Writing system \eqref{ex2_cycle} in Cartesian coordinates, we have
\begin{equation}\label{csys_2}
	\begin{aligned}
	\left\{\begin{array}{l}
	(\dot{x}^{+},\dot{y}^{-})=((x+x_0)^2-(y+y_0)^2,2(x+x_0)(y+y_0)),\text{ when } y>0, \\[5pt]
	(\dot{x}^{-},\dot{y}^{+})=(a(x-d)-by,b(x-d)+ay), \text{ when } y< 0.
	\end{array} \right.
	\end{aligned}
	\end{equation}
Now, consider the equation for the orbits of system \eqref{csys_2} when $y>0$
\begin{equation}\label{orb_2}
\frac{dy}{dx}=\frac{2(x+x_0)(y+y_0)}{(x+x_0)^2-(y+y_0)^2},
\end{equation}
or equivalent 
\begin{equation}\label{orb_22}
-2(x+x_0)(y+y_0)dx+((x+x_0)^2-(y+y_0)^2)dy=0.
\end{equation}

We emphasize that equation \eqref{orb_22} will be exact when we multiply it by the integrating factor $\mu(y)=\frac{1}{(y+y_0)^2}.$ Thus, the solution of equation \eqref{orb_2}, with initial condition $x(0)=w_0>-x_0$ and $y(0)=0$, in implicit form is 
$$\frac{(x+x_0)^2}{y+y_0}+y=\frac{(w_0+x_0)^2}{y_0}.$$
Notice that $y=0$ if, and only if, $x=w_0$ or $x=-2x_0-w_0.$ Hence, there exists $t_0>0$ such that $x(t_0)=-2x_0-w_0$ and $y(t_0)=0$. Moreover, 
\[z^-(t)=-(d+w_0+2x_0)e^{at}(\cos (bt)+i\sin (bt))+d\]
is a solution of $\dot{z}^{-}=(a+bi)(z-d)$ satisfying that $z^-(0)=-2x_0-w_0$ and $z^{-}(\frac{\pi}{b})=d+(d+2x_0+w_0)e^{\frac{a\pi}{b}}$.

Therefore, the Poincaré map at $z=w_0$ is given by $\Pi(w_{0})=d+(d+2x_0+w_0)e^{\frac{a\pi}{b}}$ and $\Pi'(w_{0})=e^{\frac{a\pi}{b}}<1$. Now, we must seek solutions for the equation $\Pi(w_{0})=w_{0}$. The number of roots of this equations correspond to the number of limit cycles. Since $a\neq 0,$ then we have a unique solution, given by $\dfrac{d+e^\frac{a\pi}{b}(d+2x_0)}{1-e^{\frac{a\pi}{b}}}$, thus we have only one limit cycle, which is stable.
\end{proof}
Emphasize that the conditions given in the previous proposition are not empty. Indeed, taking $a=-1,$ $b=1,$ $d=1,$ and $z_0=i$ we have the existence of a unique limit cycle (see Figure \ref{limit_cycle_2}).
\begin{figure}[h]
	\begin{center}
		\begin{overpic}[scale=0.26]{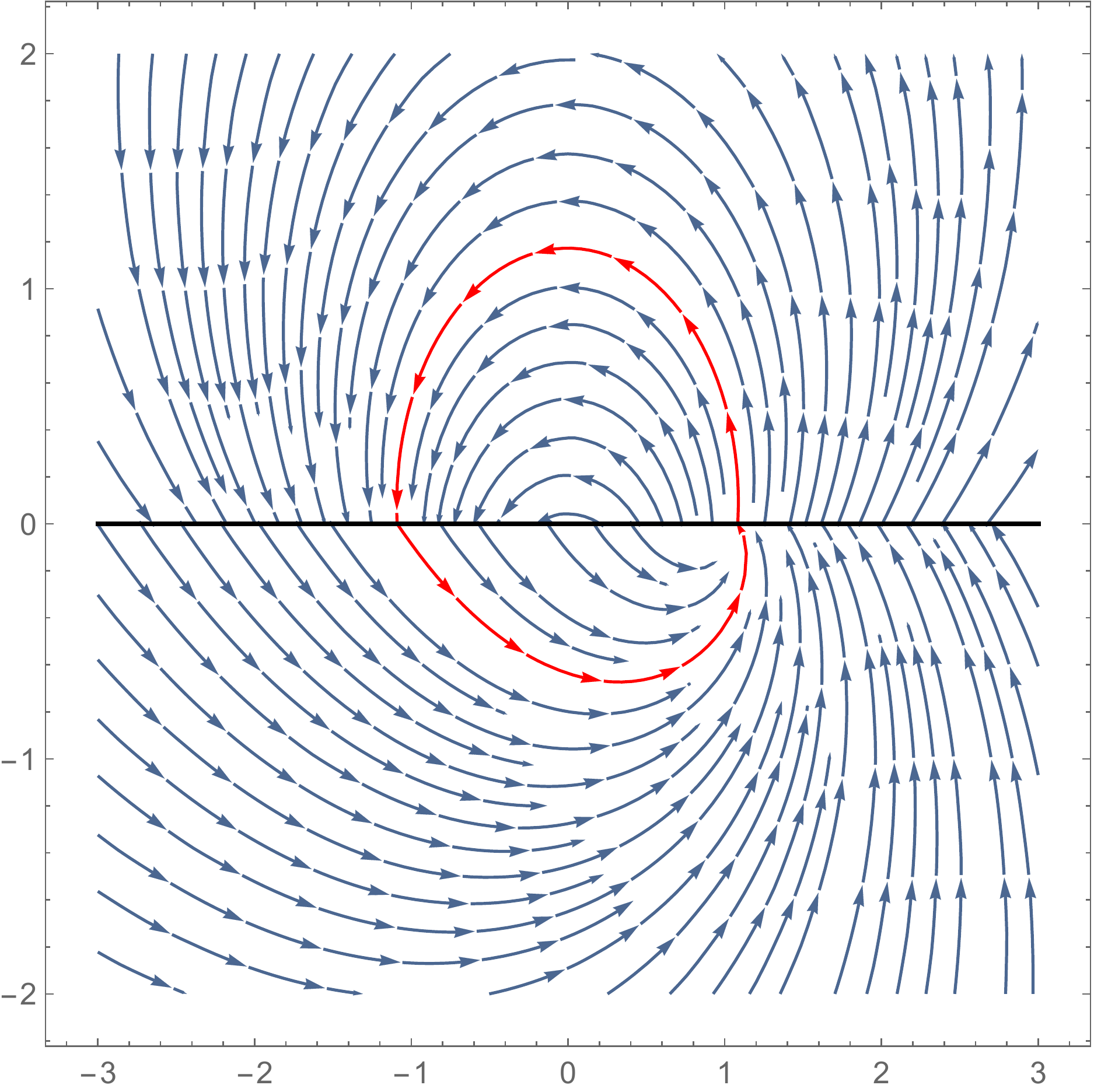}
		%\begin{overpic}[grid,tics=5,width=5cm]{limit_cycle_1.pdf}		
        \put(-7,23){$\Sigma^-$}
        \put(-7,75){$\Sigma^+$}
		\put(102,52){$\Sigma$}
		\end{overpic}
		\caption{Phase portrait of PWHS \eqref{ex1_cycle} with $a=-1,$ $b=1,$ $d=1,$ and $z_0=i$. The red  trajectory is the limit cycle of \eqref{ex2_cycle}.}
	\label{limit_cycle_2}
	\end{center}
	\end{figure}
\begin{remark}
Recall that the PWHS \eqref{ex2_cycle} with $a=0$ and $b\neq 0$ has no limit cycles. Even more, if $a=0,$ $b\neq 0,$ and $d\neq -x_0$ (resp. $a=0,$ $b\neq 0,$ and $d=-x_0$), then we have no periodic orbits (resp. we have infinite periodic orbits).
\end{remark}
\begin{proposition}\label{prop_polar}
	Let $a,b,d,$ and $y_0$ be non-zero real numbers. If $a,b,d<0,$ $y_0>0,$ $n=4k$ for some integer $k\geq 1,$ and $\cot\left(\frac{n\pi}{2(n-1)}\right)y_0<-\frac{d(1+e^\frac{a\pi}{b})}{1-e^{\frac{a\pi}{b}}}<0,$ then the PWHS
	\begin{equation}\label{ex4_cycle}
\begin{aligned}
\left\{\begin{array}{l}
\dot{z}^{+}=(z+iy_0)^n,\text{ when } \Im{(z)}>0, \\[5pt]
\dot{z}^{-}=(a+ib)(z-d), \text{ when } \Im{(z)}< 0,
\end{array} \right.
\end{aligned}
\end{equation}
	has a unique stable limit cycle.
\end{proposition}
\begin{proof}
Consider $\dot{z}^{+}=(z+iy_0)^n$. First, we shall prove that the solutions of $z^+$ are symmetric about the $y-$axis. Indeed, writing $\dot{z}^+$ in its polar form we have
%$$z+y_0i=re^{i\theta}=r\cos(\theta)+i\sin(\theta),$$
%thus
\begin{equation}\label{pczn4}
\left\{\begin{array}{rcl}
\dot{r}&=&r^n\cos(n-1)\theta,\\
\dot{\theta}&=&r^{n-1}\sin(n-1)\theta,
\end{array}\right.
\end{equation}
where $z+y_0i=re^{i\theta}=r(\cos(\theta)+i\sin(\theta)).$ It is easy to see that the orbits of this system satisfy the following equation:
\begin{equation} \label{rzn4}  
r=|\sin(n-1)\theta|^{\frac{1}{n-1}}e^C. 
\end{equation}
Since equation \eqref{rzn4} evaluated in $\pi-\theta$ and $\theta$ are the same, then the orbits of \eqref{pczn4} are symmetric with respect to the straight line $\theta=\frac{\pi}{2}.$ Therefore, we can conclude the symmetry of the solutions of $\dot{z}^+$ with respect to $y-$axis.
%$$\begin{array}{rcl}
%r&=&|\sin(n-1)(\pi-\theta)|^\frac{1}{n-1}e^c\\
%&=&|\sin[(n-1)\pi-(n-1)\theta)]|^\frac{1}{n-1}e^c\\
%&=&|\sin(n-1)\theta|^\frac{1}{n-1}e^c.\\
%\end{array}$$
%Writing the system \eqref{ex2_cycle} in Cartesian coordinates, we have
%\begin{equation}\label{csys_2}
%	\begin{aligned}
%	\left\{\begin{array}{l}
%	(\dot{x}^{+},\dot{y}^{-})=((x-x_0)^2-(y-y_0)^2,2(x-x_0)(y-y_0)),\text{ when } y>0, \\[5pt]
%	(\dot{x}^{-},\dot{y}^{+})=(a(x-d)-by,b(x-d)+ay), \text{ when } y< 0.
%	\end{array} \right.
%	\end{aligned}
%	\end{equation}

Now, consider the solution $z^+(t)$ of \eqref{ex4_cycle} with initial condition $z^+(0)=-w_0<0$. By the symmetry of the solutions of \eqref{pczn4}, we have that there exists $t_0>0$ such that $z^+(t_0)=w_0.$ Moreover, 
\[z^-(t)=-(d-w_0)e^{at}(\cos (bt)+i\sin (bt))+d\]
is a solution of $\dot{z}^{-}=(a+bi)(z-d)$ satisfying that $z^-(0)=w_0$ and $z^{-}(-\frac{\pi}{b})=d+(d-w_0)e^{-\frac{a\pi}{b}}$.

Therefore, the Poincaré map around $z=-w_0$ is given by $\Pi(z)=d+(d+z)e^{-\frac{a\pi}{b}}$ and $\Pi'(-w_{0})=e^{-\frac{a\pi}{b}}<1$. Now, we must seek solutions for the equation $\Pi(-w_{0})=-w_{0}$.  Since $a\neq 0,$ then we have a unique solution, given by $w_0=\dfrac{d(1+e^\frac{a\pi}{b})}{1-e^{\frac{a\pi}{b}}}$, thus we have only one limit cycle (see remark \ref{unique_cycle}), which is stable.
\end{proof}

\begin{figure}[h]
	\begin{center}
		\begin{overpic}[scale=0.3]{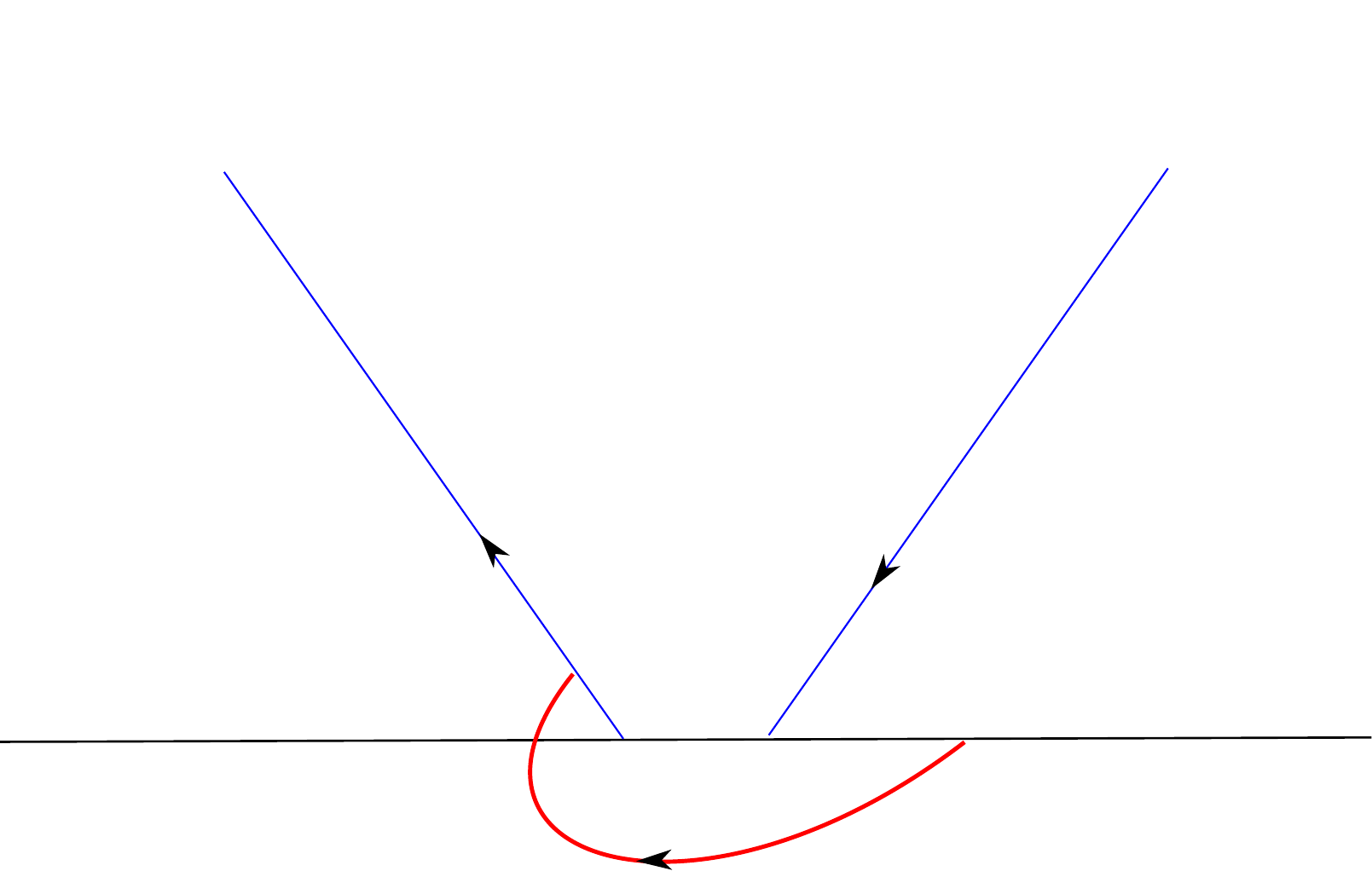}
		%\begin{overpic}[grid,tics=10,width=4cm]{limit_cycle_unique2.pdf}		
         % \put(-11,-2){$\Sigma^-$}
        %\put(-11,15){$\Sigma^+$}
		\put(101,10){$\Sigma$}
		\put(8,55){$\frac{(n-2)\pi}{2(n-1)}$}
		\put(75,55){$\frac{n\pi}{2(n-1)}$}
		%\put(101,11){$\Sigma$}
		%\put(16,62){$\frac{(n-2)\pi}{2(n-1)}$}
		%\put(68,62){$\frac{n\pi}{2(n-1)}$}
		\end{overpic}
		\caption{Uniqueness of the limit cycle.}
	\label{limit_cycle_unique2}
	\end{center}
	\end{figure}
	\begin{remark}\label{unique_cycle}
Recall that the limit cycle found is determined by rays $\frac{(n-2)\pi}{2(n-1)}$ and $\frac{n\pi}{2(n-1)}$, however due to the invariance of the rays of $z^+$ and the orientation of the trajectories, then this limit cycle is unique (see figure \ref{limit_cycle_unique2}).
\end{remark}
Notice that the conditions given in the previous proposition are not empty. Indeed, taking $a=-1,$ $b=-1,$ $d=-\frac{1}{2},$ and $y_0=1$ we have the existence of a limit cycle (see Figure \ref{limit_cycle_4}).
\begin{figure}[h]
	\begin{center}
		\begin{overpic}[scale=0.40]{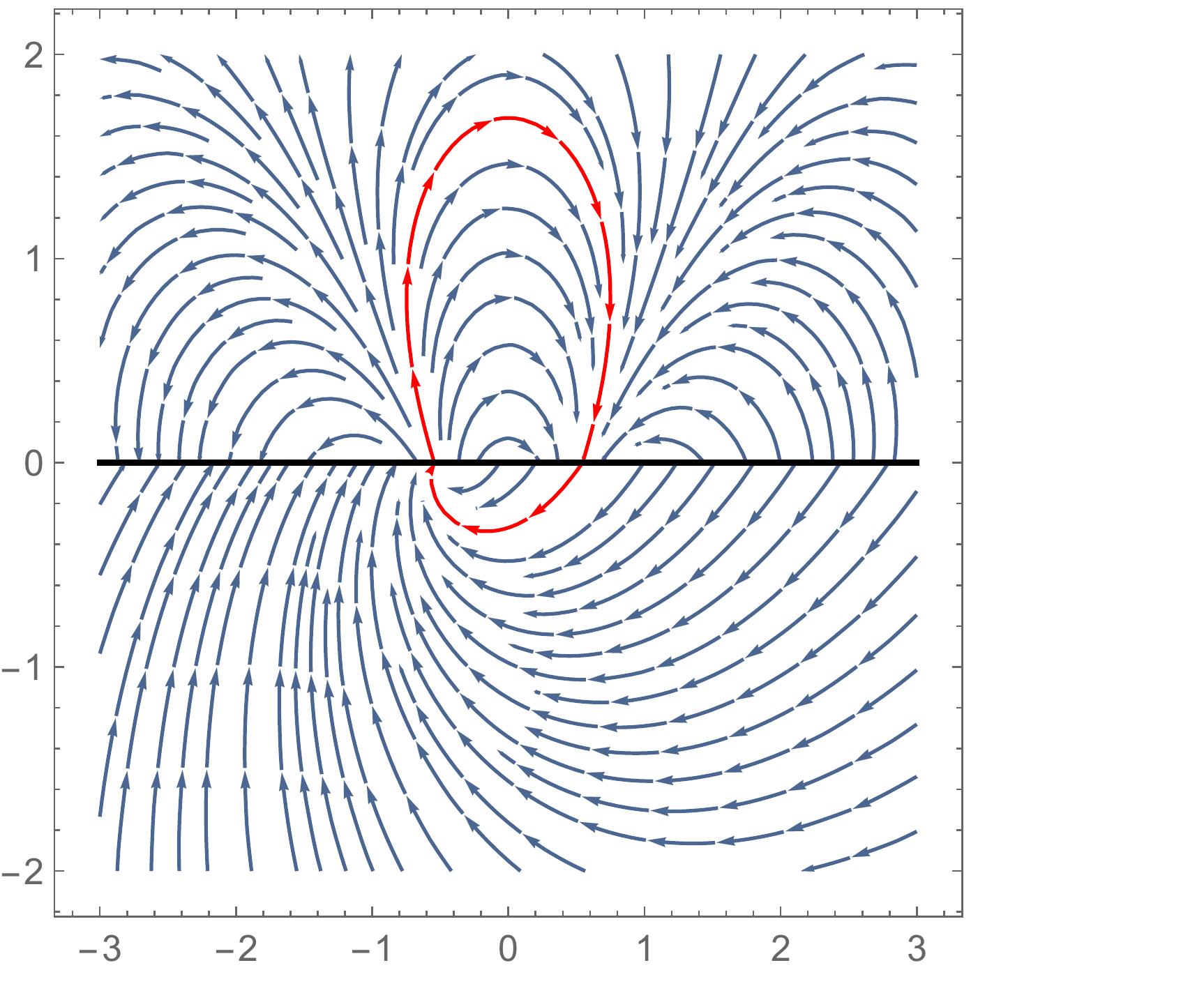}
		%\begin{overpic}[grid,tics=5,width=5cm]{limit_cycle_1.pdf}		
          \put(-7,23){$\Sigma^-$}
        \put(-7,58){$\Sigma^+$}
		\put(82,43){$\Sigma$}
		\end{overpic}
		\caption{Phase portrait of PWHS \eqref{ex4_cycle} with $n=4,$ $a=-1,$ $b=-1,$ $d=-\frac{1}{2},$ and $y_0=1$. The red  trajectory is the limit cycle of \eqref{ex4_cycle}.}
	\label{limit_cycle_4}
	\end{center}
	\end{figure}
\begin{proposition}
	Let $a,b,d,$ and $y_0$ be non-zero real numbers. If $a<0<b,$ $y_0,d>0,$ $n=4k-2$ for some integer $k> 1,$ and $0<\frac{d(1+e^\frac{a\pi}{b})}{1-e^{\frac{a\pi}{b}}}<\cot\left(\frac{(n-2)\pi}{2(n-1)}\right)y_0,$ then the PWHS
	\begin{equation}\label{ex3_cycle}
\begin{aligned}
\left\{\begin{array}{l}
\dot{z}^{+}=(z+iy_0)^n,\text{ when } \Im{(z)}>0, \\[5pt]
\dot{z}^{-}=(a+ib)(z-d), \text{ when } \Im{(z)}< 0,
\end{array} \right.
\end{aligned}
\end{equation}
	has a unique stable limit cycle.
\end{proposition}
\begin{proof}
Consider $\dot{z}^{+}=(z+iy_0)^n$. Writing $\dot{z}^+$ in its polar form we have
\begin{equation}\label{pczn3}
\left\{\begin{array}{rcl}
\dot{r}&=&r^n\cos(n-1)\theta,\\
\dot{\theta}&=&r^{n-1}\sin(n-1)\theta,
\end{array}\right.
\end{equation}
where $z+y_0i=re^{i\theta}=r(\cos(\theta)+i\sin(\theta)).$ By the proof of Proposition \ref{prop_polar}, we know that the solutions of $z^+$ are symmetric about the $y-$axis.
%First, we shall prove that the solutions of $z^+$ are symmetric about the $y-$axis. Indeed, writing $\dot{z}^+$ in its polar form we have
%%$$z+y_0i=re^{i\theta}=r\cos(\theta)+i\sin(\theta),$$
%%thus
%\begin{equation}\label{pczn3}
%\left\{\begin{array}{rcl}
%\dot{r}&=&r^n\cos(n-1)\theta,\\
%\dot{\theta}&=&r^{n-1}\sin(n-1)\theta,
%\end{array}\right.
%\end{equation}
%where $z+y_0i=re^{i\theta}=r(\cos(\theta)+i\sin(\theta)).$ It is easy to see that the orbits of this system satisfy the following equation:
%\begin{equation} \label{rzn3}  
%r=|\sin(n-1)\theta|^{\frac{1}{n-1}}e^C. 
%\end{equation}
%Since equation \eqref{rzn3} evaluated in $\pi-\theta$ and $\theta$ are the same, then the orbits of \eqref{pczn3} are symmetric with respect to the straight line $\theta=\frac{\pi}{2}.$ Therefore, we can conclude the symmetry of the solutions of $\dot{z}^+$ with respect to $y-$axis.
%$$\begin{array}{rcl}
%r&=&|\sin(n-1)(\pi-\theta)|^\frac{1}{n-1}e^c\\
%&=&|\sin[(n-1)\pi-(n-1)\theta)]|^\frac{1}{n-1}e^c\\
%&=&|\sin(n-1)\theta|^\frac{1}{n-1}e^c.\\
%\end{array}$$
%Writing the system \eqref{ex2_cycle} in Cartesian coordinates, we have
%\begin{equation}\label{csys_2}
%	\begin{aligned}
%	\left\{\begin{array}{l}
%	(\dot{x}^{+},\dot{y}^{-})=((x-x_0)^2-(y-y_0)^2,2(x-x_0)(y-y_0)),\text{ when } y>0, \\[5pt]
%	(\dot{x}^{-},\dot{y}^{+})=(a(x-d)-by,b(x-d)+ay), \text{ when } y< 0.
%	\end{array} \right.
%	\end{aligned}
%	\end{equation}

Now, consider the solution $z^+(t)$ of \eqref{ex3_cycle} with initial condition $z^+(0)=w_0>0$. By the symmetry of the solutions of \eqref{pczn3}, we have that there exists $t_0>0$ such that $z^+(t_0)=-w_0.$ Moreover, 
\[z^-(t)=-(d+w_0)e^{at}(\cos (bt)+i\sin (bt))+d\]
is a solution of $\dot{z}^{-}=(a+bi)(z-d)$ satisfying that $z^-(0)=-w_0$ and $z^{-}(\frac{\pi}{b})=d+(d+w_0)e^{\frac{a\pi}{b}}$.

Consequently, the Poincaré map at $z=w_0$ is given by $\Pi(w_{0})=d+(d+w_0)e^{\frac{a\pi}{b}}$ and $\Pi'(w_{0})=e^{\frac{a\pi}{b}}<1$. Now, we must seek solutions for the equation $\Pi(w_{0})=w_{0}$.  Since $a\neq 0,$ then we have a unique solution, given by $\dfrac{d(1+e^\frac{a\pi}{b})}{1-e^{\frac{a\pi}{b}}}$, thus we have only one limit cycle (see Remark \ref{unique_cycle}), which is stable.
\end{proof}
It is important to note that the conditions given in the previous proposition are not empty. Indeed, taking $a=-1,$ $b=1,$ $d=-\frac{1}{5},$ and $y_0=1$ we have the existence of a limit cycle (see Figure \ref{limit_cycle_3}).
\begin{figure}[h]
	\begin{center}
		\begin{overpic}[scale=0.40]{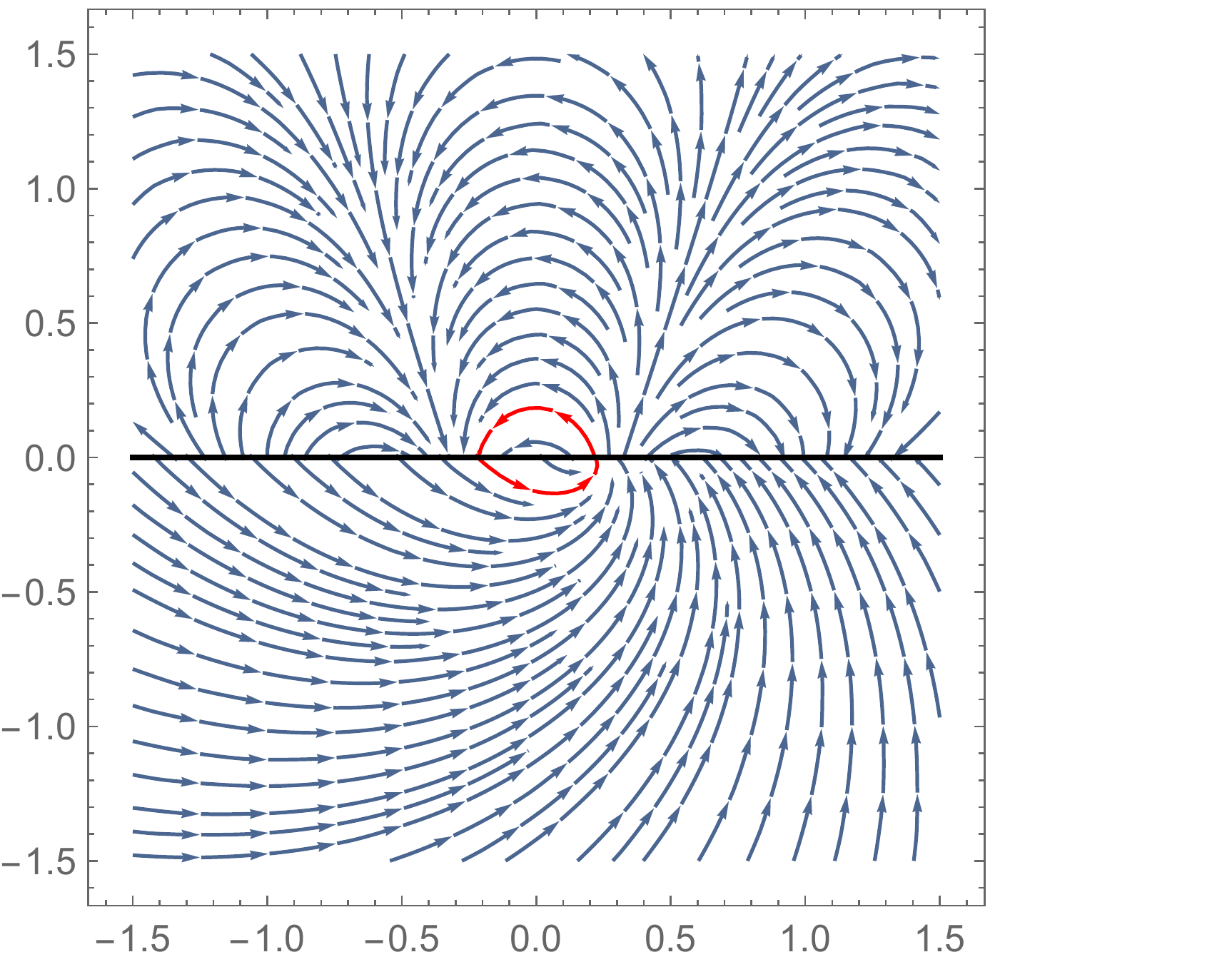}
		%\begin{overpic}[grid,tics=5,width=5cm]{limit_cycle_1.pdf}		
        \put(-7,23){$\Sigma^-$}
        \put(-7,58){$\Sigma^+$}
		\put(82,42){$\Sigma$}
		\end{overpic}
		\caption{Phase portrait of PWHS \eqref{ex3_cycle} with $n=6,$ $a=-1,$ $b=1,$ $d=\frac{1}{5},$ and $y_0=1$. The red  trajectory is the limit cycle of \eqref{ex3_cycle}.}
	\label{limit_cycle_3}
	\end{center}
	\end{figure}
	\begin{proposition}\label{prop_polar2}
	Let $a,b,d,$ and $y_0$ be non-zero real numbers. If $a,b,d<0,$ $y_0>0,$ $n=4k-1$ for some integer $k\geq 1,$ and $\cot\left(\frac{n\pi}{2(n-1)}\right)y_0<-\frac{d(1+e^\frac{a\pi}{b})}{1-e^{\frac{a\pi}{b}}}<0,$ then the PWHS
	\begin{equation}\label{ex5_cycle}
\begin{aligned}
\left\{\begin{array}{l}
\dot{z}^{+}=i(z+iy_0)^n,\text{ when } \Im{(z)}>0, \\[5pt]
\dot{z}^{-}=(a+ib)(z-d), \text{ when } \Im{(z)}< 0,
\end{array} \right.
\end{aligned}
\end{equation}
	has a unique stable limit cycle.
\end{proposition}
\begin{proof}
Consider $\dot{z}^{+}=i(z+iy_0)^n$. First, we shall prove that the solutions of $z^+$ are symmetric about the $y-$axis. Indeed, writing $\dot{z}^+$ in its polar form we have
%$$z+y_0i=re^{i\theta}=r\cos(\theta)+i\sin(\theta),$$
%thus
\begin{equation}\label{pczn5}
\left\{\begin{array}{rcl}
\dot{r}&=&-r^n\sin(n-1)\theta,\\
\dot{\theta}&=&r^{n-1}\cos(n-1)\theta,
\end{array}\right.
\end{equation}
where $z+y_0i=re^{i\theta}=r(\cos(\theta)+i\sin(\theta)).$ It is easy to see that the orbits of this system satisfy the following equation:
\begin{equation} \label{rzn5}  
r=|\cos(n-1)\theta|^{\frac{1}{n-1}}e^C. 
\end{equation}
Since equation \eqref{rzn5} evaluated in $\pi-\theta$ and $\theta$ are the same, then the orbits of \eqref{pczn5} are symmetric with respect to the straight line $\theta=\frac{\pi}{2}.$ Therefore, we can conclude the symmetry of the solutions of $\dot{z}^+$ with respect to $y-$axis.
%$$\begin{array}{rcl}
%r&=&|\sin(n-1)(\pi-\theta)|^\frac{1}{n-1}e^c\\
%&=&|\sin[(n-1)\pi-(n-1)\theta)]|^\frac{1}{n-1}e^c\\
%&=&|\sin(n-1)\theta|^\frac{1}{n-1}e^c.\\
%\end{array}$$
%Writing the system \eqref{ex2_cycle} in Cartesian coordinates, we have
%\begin{equation}\label{csys_2}
%	\begin{aligned}
%	\left\{\begin{array}{l}
%	(\dot{x}^{+},\dot{y}^{-})=((x-x_0)^2-(y-y_0)^2,2(x-x_0)(y-y_0)),\text{ when } y>0, \\[5pt]
%	(\dot{x}^{-},\dot{y}^{+})=(a(x-d)-by,b(x-d)+ay), \text{ when } y< 0.
%	\end{array} \right.
%	\end{aligned}
%	\end{equation}

Now, consider the solution $z^+(t)$ of \eqref{ex5_cycle} with initial condition $z^+(0)=-w_0<0$. By the symmetry of the solutions of \eqref{pczn5}, we have that there exists $t_0>0$ such that $z^+(t_0)=w_0.$ Moreover,
\[z^-(t)=-(d-w_0)e^{at}(\cos (bt)+i\sin (bt))+d\]
is a solution of $\dot{z}^{-}=(a+bi)(z-d)$ satisfying that $z^-(0)=w_0$ and $z^{-}(-\frac{\pi}{b})=d+(d-w_0)e^{-\frac{a\pi}{b}}$.

Therefore, the Poincaré map around $z=-w_0$ is given by $\Pi(z)=d+(d+z)e^{-\frac{a\pi}{b}}$ and $\Pi'(-w_{0})=e^{-\frac{a\pi}{b}}<1$. Now, we must seek solutions for the equation $\Pi(-w_{0})=-w_{0}$.  Since $a\neq 0,$ then we have a unique solution, given by $w_0=\dfrac{d(1+e^\frac{a\pi}{b})}{1-e^{\frac{a\pi}{b}}}$, thus we have only one limit cycle (see Remark \ref{unique_cycle}), which is stable.
\end{proof}
Notice that the conditions given in the previous proposition are not empty. Indeed, taking $a=-1,$ $b=-1,$ $d=-\frac{1}{2},$ and $y_0=1$ we have the existence of a limit cycle (see Figure \ref{limit_cycle_5}).
\begin{figure}[h]
	\begin{center}
		\begin{overpic}[scale=0.40]{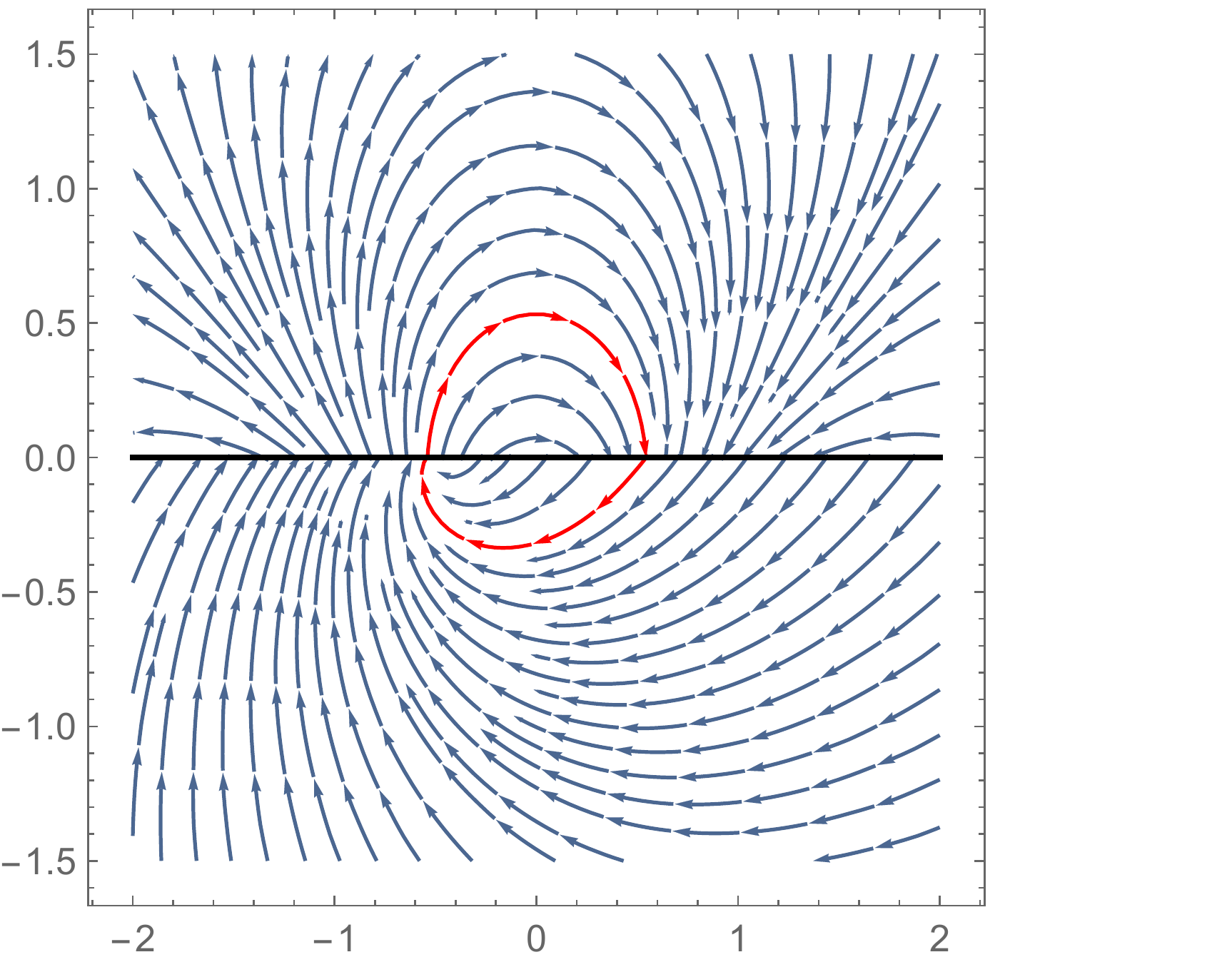}
		%\begin{overpic}[grid,tics=5,width=5cm]{limit_cycle_1.pdf}		
          \put(-7,23){$\Sigma^-$}
        \put(-7,58){$\Sigma^+$}
		\put(82,43){$\Sigma$}
		\end{overpic}
		\caption{Phase portrait of PWHS \eqref{ex5_cycle} with $n=3,$ $a=-1,$ $b=-1,$ $d=-\frac{1}{2}$ and $y_0=1$. The red  trajectory is the limit cycle of \eqref{ex5_cycle}.}
	\label{limit_cycle_5}
	\end{center}
	\end{figure}
\begin{proposition}
	Let $a,b,d,$ and $y_0$ be non-zero real numbers. If $a<0<b,$ $y_0,d>0,$ $n=4k+1$ for some integer $k\geq 1,$ and $0<\frac{d(1+e^\frac{a\pi}{b})}{1-e^{\frac{a\pi}{b}}}<\cot\left(\frac{(n-2)\pi}{2(n-1)}\right)y_0,$ then the PWHS
	\begin{equation}\label{ex6_cycle}
\begin{aligned}
\left\{\begin{array}{l}
\dot{z}^{+}=i(z+iy_0)^n,\text{ when } \Im{(z)}>0, \\[5pt]
\dot{z}^{-}=(a+ib)(z-d), \text{ when } \Im{(z)}< 0,
\end{array} \right.
\end{aligned}
\end{equation}
	has a unique stable limit cycle.
\end{proposition}
\begin{proof}
Consider $\dot{z}^{+}=i(z+iy_0)^n$. Writing $\dot{z}^+$ in its polar form we have
\begin{equation}\label{pczn6}
\left\{\begin{array}{rcl}
\dot{r}&=&-r^n\sin(n-1)\theta,\\
\dot{\theta}&=&r^{n-1}\cos(n-1)\theta,
\end{array}\right.
\end{equation}
where $z+y_0i=re^{i\theta}=r(\cos(\theta)+i\sin(\theta)).$ By the proof of Proposition \ref{prop_polar2}, we know that the solutions of $z^+$ are symmetric about the $y-$axis.
%Consider $\dot{z}^{+}=i(z+iy_0)^n$. First, we shall prove that the solutions of $z^+$ are symmetric about the $y-$axis. Indeed, writing $\dot{z}^+$ in its polar form we have
%%$$z+y_0i=re^{i\theta}=r\cos(\theta)+i\sin(\theta),$$
%%thus
%\begin{equation}\label{pczn6}
%\left\{\begin{array}{rcl}
%\dot{r}&=&-r^n\sin(n-1)\theta,\\
%\dot{\theta}&=&r^{n-1}\cos(n-1)\theta,
%\end{array}\right.
%\end{equation}
%where $z+y_0i=re^{i\theta}=r(\cos(\theta)+i\sin(\theta)).$ It is easy to see that the orbits of this system satisfy the following equation:
%\begin{equation} \label{rzn6}  
%r=|\cos(n-1)\theta|^{\frac{1}{n-1}}e^C. 
%\end{equation}
%Since equation \eqref{rzn6} evaluated in $\pi-\theta$ and $\theta$ are the same, then the orbits of \eqref{pczn6} are symmetric with respect to the straight line $\theta=\frac{\pi}{2}.$ Therefore, we can conclude the symmetry of the solutions of $\dot{z}^+$ with respect to $y-$axis.
%$$\begin{array}{rcl}
%r&=&|\sin(n-1)(\pi-\theta)|^\frac{1}{n-1}e^c\\
%&=&|\sin[(n-1)\pi-(n-1)\theta)]|^\frac{1}{n-1}e^c\\
%&=&|\sin(n-1)\theta|^\frac{1}{n-1}e^c.\\
%\end{array}$$
%Writing the system \eqref{ex2_cycle} in Cartesian coordinates, we have
%\begin{equation}\label{csys_2}
%	\begin{aligned}
%	\left\{\begin{array}{l}
%	(\dot{x}^{+},\dot{y}^{-})=((x-x_0)^2-(y-y_0)^2,2(x-x_0)(y-y_0)),\text{ when } y>0, \\[5pt]
%	(\dot{x}^{-},\dot{y}^{+})=(a(x-d)-by,b(x-d)+ay), \text{ when } y< 0.
%	\end{array} \right.
%	\end{aligned}
%	\end{equation}

Now, consider the solution $z^+(t)$ of \eqref{ex6_cycle} with initial condition $z^+(0)=w_0>0$. By the symmetry of the solutions of \eqref{pczn6}, we have that there exists $t_0>0$ such that $z^+(t_0)=-w_0.$ Moreover, 
\[z^-(t)=-(d+w_0)e^{at}(\cos (bt)+i\sin (bt))+d\]
is a solution of $\dot{z}^{-}=(a+bi)(z-d)$ satisfying that $z^-(0)=-w_0$ and $z^{-}(\frac{\pi}{b})=d+(d+w_0)e^{\frac{a\pi}{b}}$.

Consequently, the Poincaré map at $z=w_0$ is given by $\Pi(w_{0})=d+(d+w_0)e^{\frac{a\pi}{b}}$ and $\Pi'(w_{0})=e^{\frac{a\pi}{b}}<1$. Now, we must seek solutions for the equation $\Pi(w_{0})=w_{0}$.  Since $a\neq 0,$ then we have a unique solution, given by $\dfrac{d(1+e^\frac{a\pi}{b})}{1-e^{\frac{a\pi}{b}}}$, thus we have only one limit cycle (see Remark \ref{unique_cycle}), which is stable.
\end{proof}
Notice that the conditions given in the previous proposition are not empty. Indeed, taking $a=-1,$ $b=1,$ $d=\frac{3}{10},$ and $y_0=1$ we have the existence of a limit cycle (see Figure \ref{limit_cycle_6}).
\begin{figure}[h]
	\begin{center}
		\begin{overpic}[scale=0.40]{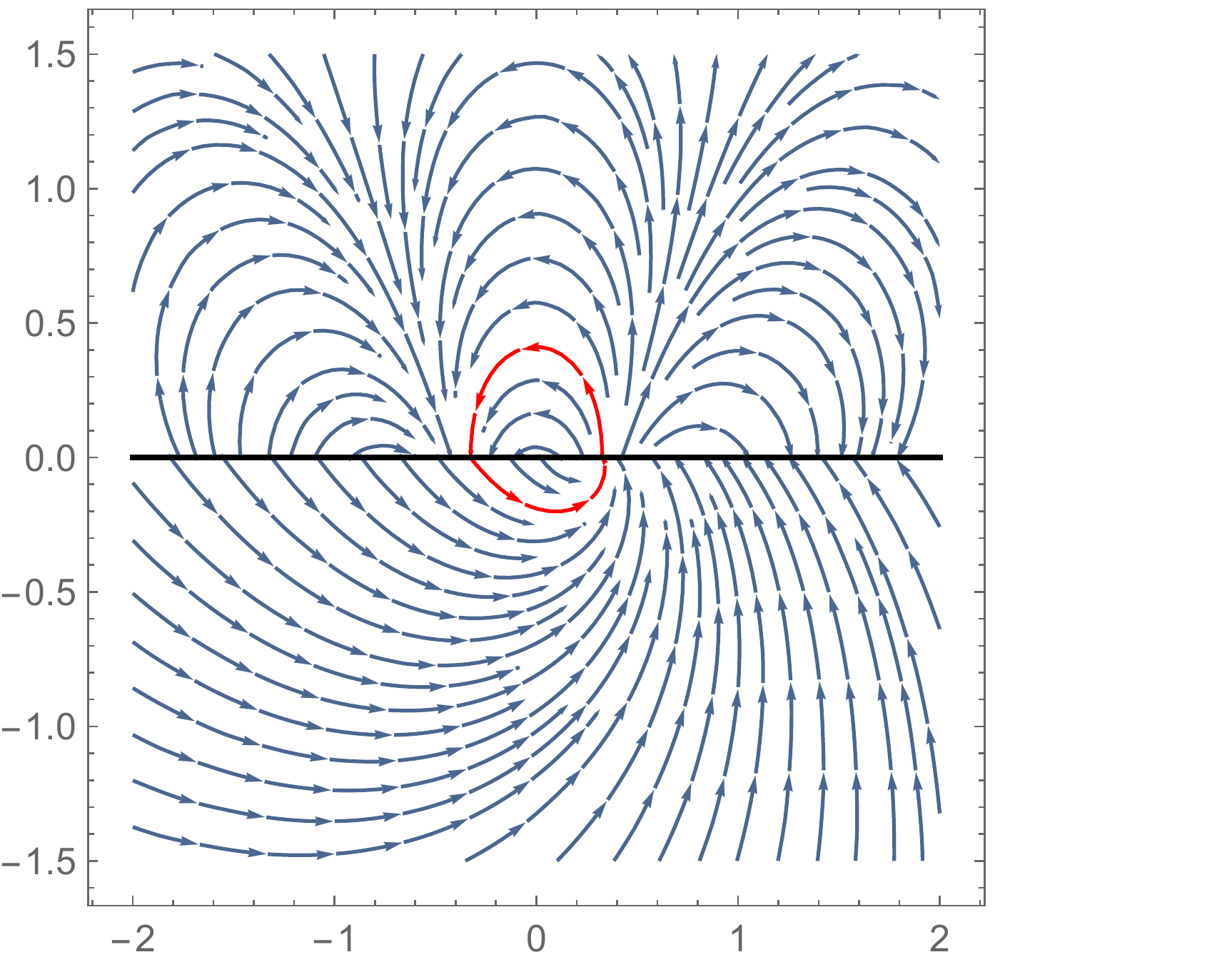}
		%\begin{overpic}[grid,tics=5,width=5cm]{limit_cycle_1.pdf}		
        \put(-7,23){$\Sigma^-$}
        \put(-7,58){$\Sigma^+$}
		\put(82,42){$\Sigma$}
		\end{overpic}
		\caption{Phase portrait of PWHS \eqref{ex6_cycle} with $n=5,$ $a=-1,$ $b=1,$ $d=\frac{3}{10},$ and $y_0=1$. The red  trajectory is the limit cycle of \eqref{ex6_cycle}.}
	\label{limit_cycle_6}
	\end{center}
	\end{figure}

Now, we do the study of vector fields that admit poles, $\frac{1}{z^n}$ for $n\geq 1$, which are divided into 4 cases that depend on $n$. For that,  we use the symmetry of this normal form and that the rays $\frac{j\pi}{n+1}$, $j=\{1,\cdots, 2(n+1)\}$ (resp. $\frac{j\pi}{2(n+1)}$, $j=\{1,3,\cdots, 4(n+1)-1\}$) are invariant by the flow of the equation $\dot{z}=\frac{1}{z^n},$ with $n$ even (resp. $\dot{z}=\frac{i}{z^n}$, with $n$ odd).

For this normal form we consider real singularities of the pole type in the following sense: given a piecewise smooth vector field
\begin{equation}
\begin{aligned}
\left\{\begin{array}{l}
\dot{z}^{+}=f^{+}(z), \text{ when } \Im(z)> 0,\\[5pt]
\dot{z}^{-}=f^{-}(z),\text{ when }\Im(z)<0,
\end{array} \right.
\end{aligned}
\end{equation}
we say that a singularity of the pole type $z_0$ of $f^+$ (resp. $f^-$) is real when $z_0\in\Sigma^+$ (resp. $z_0\in\Sigma^-$). We recall that it is possible to construct limit cycles using virtual singularities.
\begin{theorem}\label{cl3}
	Given $n\in\N_{n\geq 1}$, there exist $a,b,d,$ and $y_0$ be non-zero real numbers satisfying table \eqref{table2}, such that the PWHS
	\begin{equation}\label{1_ex_cycle}
\begin{aligned}
\left\{\begin{array}{l}
\dot{z}^{+}=(a+ib)(z-d), \text{ when } \Im{(z)}> 0,\\[5pt]
\dot{z}^{-}=\frac{i^m}{(z+iy_0)^n},\text{ when } \Im{(z)}<0,
\end{array} \right.
\end{aligned}
\end{equation}
	has a unique stable limit cycle, where $m=0$ if $n$ is even and $m=1$ otherwise.
\end{theorem}
\begin{equation}\label{table2}
\begin{array}{|| c|c | c | c | c | c | c ||}
\hline
n&k	&a &b & d & y_0& \text{Main condition} \\
\hline\hline
4k-2&\geq 1	& - & - & + & + &0<\dfrac{d(1+e^\frac{a\pi}{b})}{-1+e^{\frac{a\pi}{b}}}<\cot\left(\dfrac{n\pi}{2(n+1)}\right)y_0\\
\hline
4k-1&\geq 1	& - & - & + & + & 0<\dfrac{d(1+e^\frac{a\pi}{b})}{-1+e^{\frac{a\pi}{b}}}<\cot\left(\dfrac{n\pi}{2(n+1)}\right)y_0 \\
\hline
4k&\geq 1	& - & + & - & + & \cot\left(\dfrac{(n+2)\pi}{2(n+1)}\right)y_0<-\dfrac{d(1+e^\frac{a\pi}{b})}{-1+e^{\frac{a\pi}{b}}}<0\\
\hline
4k+1&\geq 0	& - & + & - & + & \cot\left(\dfrac{(n+2)\pi}{2(n+1)}\right)y_0<-\dfrac{d(1+e^\frac{a\pi}{b})}{-1+e^{\frac{a\pi}{b}}}<0\\
\hline
\end{array}
\end{equation}
The proof of this theorem is an immediate consequence of the following 4 propositions.
	\begin{proposition}\label{prop_polar3}
	Let $a,b,d,$ and $y_0$ be non-zero real numbers. If $a,b<0,$ $d,y_0>0,$ $n=4k-2$ for some integer $k\geq 1,$ and $0<\frac{d(1+e^\frac{a\pi}{b})}{-1+e^{\frac{a\pi}{b}}}<\cot\left(\frac{n\pi}{2(n+1)}\right)y_0,$ then the PWHS
	\begin{equation}\label{ex7_cycle}
\begin{aligned}
\left\{\begin{array}{l}
\dot{z}^{+}=(a+ib)z, \text{ when } \Im{(z)}> 0, \\[5pt]
\dot{z}^{-}=\frac{1}{(z+iy_0)^n},\text{ when } \Im{(z)}<0,
\end{array} \right.
\end{aligned}
\end{equation}
	has a unique stable limit cycle.
\end{proposition}
\begin{proof}
Consider $\dot{z}^{-}=\frac{1}{(z+iy_0)^n}$. First, we shall prove that the solutions of $z^-$ are symmetric about the $y-$axis. Indeed, writing $\dot{z}^-$ in its polar form we have
%$$z+y_0i=re^{i\theta}=r\cos(\theta)+i\sin(\theta),$$
%thus
\begin{equation}\label{pczn7}
\left\{\begin{array}{rcl}
\dot{r}&=&r^{-n}\cos(n+1)\theta,\\
\dot{\theta}&=&-r^{-n-1}\sin(n+1)\theta,
\end{array}\right.
\end{equation}
where $z+y_0i=re^{i\theta}=r(\cos(\theta)+i\sin(\theta)).$ It is easy to see that the orbits of this system satisfy the following equation:
\begin{equation} \label{rzn7}  
r=\frac{e^C}{|\sin(n+1)\theta|^{\frac{1}{n+1}}}. 
\end{equation}
Since equation \eqref{rzn7} evaluated in $\pi-\theta$ and $\theta$ are the same, then the orbits of \eqref{pczn7} are symmetric with respect to the straight line $\theta=\frac{\pi}{2}.$ Therefore, we can conclude the symmetry of the solutions of $\dot{z}^-$ with respect to $y-$axis.
%$$\begin{array}{rcl}
%r&=&|\sin(n-1)(\pi-\theta)|^\frac{1}{n-1}e^c\\
%&=&|\sin[(n-1)\pi-(n-1)\theta)]|^\frac{1}{n-1}e^c\\
%&=&|\sin(n-1)\theta|^\frac{1}{n-1}e^c.\\
%\end{array}$$
%Writing the system \eqref{ex2_cycle} in Cartesian coordinates, we have
%\begin{equation}\label{csys_2}
%	\begin{aligned}
%	\left\{\begin{array}{l}
%	(\dot{x}^{+},\dot{y}^{-})=((x-x_0)^2-(y-y_0)^2,2(x-x_0)(y-y_0)),\text{ when } y>0, \\[5pt]
%	(\dot{x}^{-},\dot{y}^{+})=(a(x-d)-by,b(x-d)+ay), \text{ when } y< 0.
%	\end{array} \right.
%	\end{aligned}
%	\end{equation}

Now, consider the solution $z^-(t)$ of \eqref{ex7_cycle} with initial condition $z^-(0)=w_0>0$. By the symmetry of the solutions of \eqref{pczn7}, we have that there exists $t_0>0$ such that $z^-(t_0)=-w_0.$ Moreover, 
\[z^+(t)=-(d+w_0)e^{at}(\cos (bt)+i\sin (bt))+d\]
is a solution of $\dot{z}^{+}=(a+bi)(z-d)$ satisfying that $z^+(0)=-w_0$ and $z^{+}(-\frac{\pi}{b})=d+(d+w_0)e^{-\frac{a\pi}{b}}$.

Consequently, the Poincaré map at $z=w_0$ is given by $\Pi(w_{0})=d+(d+w_0)e^{-\frac{a\pi}{b}}$ and $\Pi'(w_{0})=e^{\frac{-a\pi}{b}}<1$. Now, we must seek solutions for the equation $\Pi(w_{0})=w_{0}$.  Since $a\neq 0,$ then we have a unique solution, given by $\dfrac{d(1+e^\frac{a\pi}{b})}{-1+e^{\frac{a\pi}{b}}}$, thus we have only one limit cycle (see Remark \ref{unique_cycle_2}), which is stable.
\end{proof}
\begin{figure}[h]
	\begin{center}
		\begin{overpic}[scale=0.3]{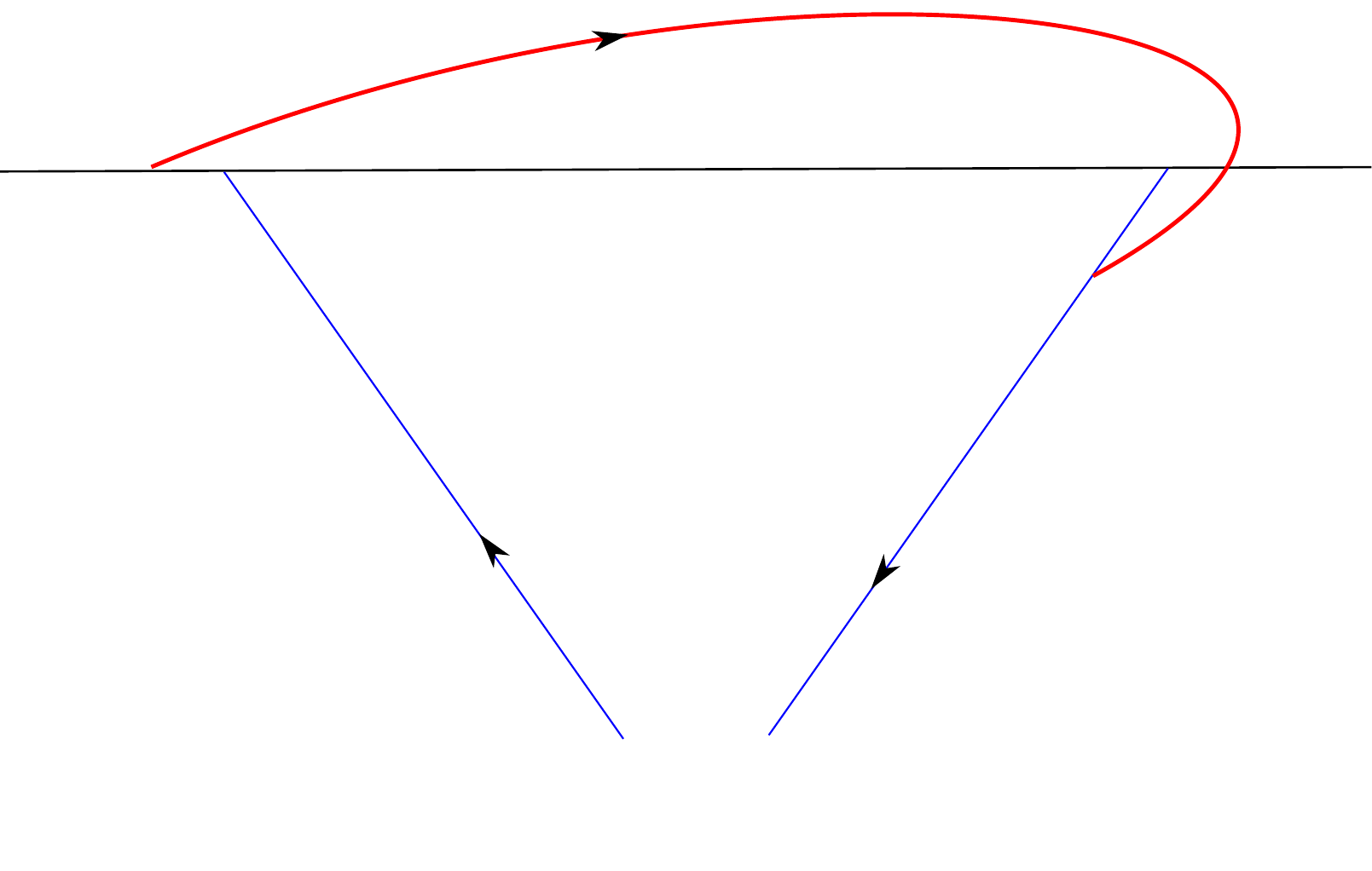}
		%\begin{overpic}[grid,tics=10,width=4cm]{limit_cycle_unique3.pdf}		
         % \put(-11,-2){$\Sigma^-$}
        %\put(-11,15){$\Sigma^+$}
		\put(101,51){$\Sigma$}
		\put(28,1){$\frac{(n+2)\pi}{2(n+1)}$}
		\put(53,2){$\frac{n\pi}{2(n+1)}$}
		%\put(101,11){$\Sigma$}
		%\put(16,62){$\frac{(n-2)\pi}{2(n-1)}$}
		%\put(68,62){$\frac{n\pi}{2(n-1)}$}
		\end{overpic}
		\caption{Uniqueness of the limit cycle.}
	\label{limit_cycle_unique}
	\end{center}
	\end{figure}
	\begin{remark}\label{unique_cycle_2}
Recall that the limit cycle found is determined by rays $\frac{n\pi}{2(n+1)}$ and $\frac{(n+2)\pi}{2(n+1)}$, however due to the invariance of the rays of $z^-$ and the orientation of the trajectories, then this limit cycle is unique (see figure \ref{limit_cycle_unique}).
\end{remark}
It is important to emphasize that the conditions given in the previous proposition are not empty. Indeed, taking $a=-1,$ $b=-1,$ and $d=\frac{1}{2}$ we have the existence of a limit cycle (see Figure \ref{limit_cycle_7}).
\begin{figure}[h]
	\begin{center}
		\begin{overpic}[scale=0.29]{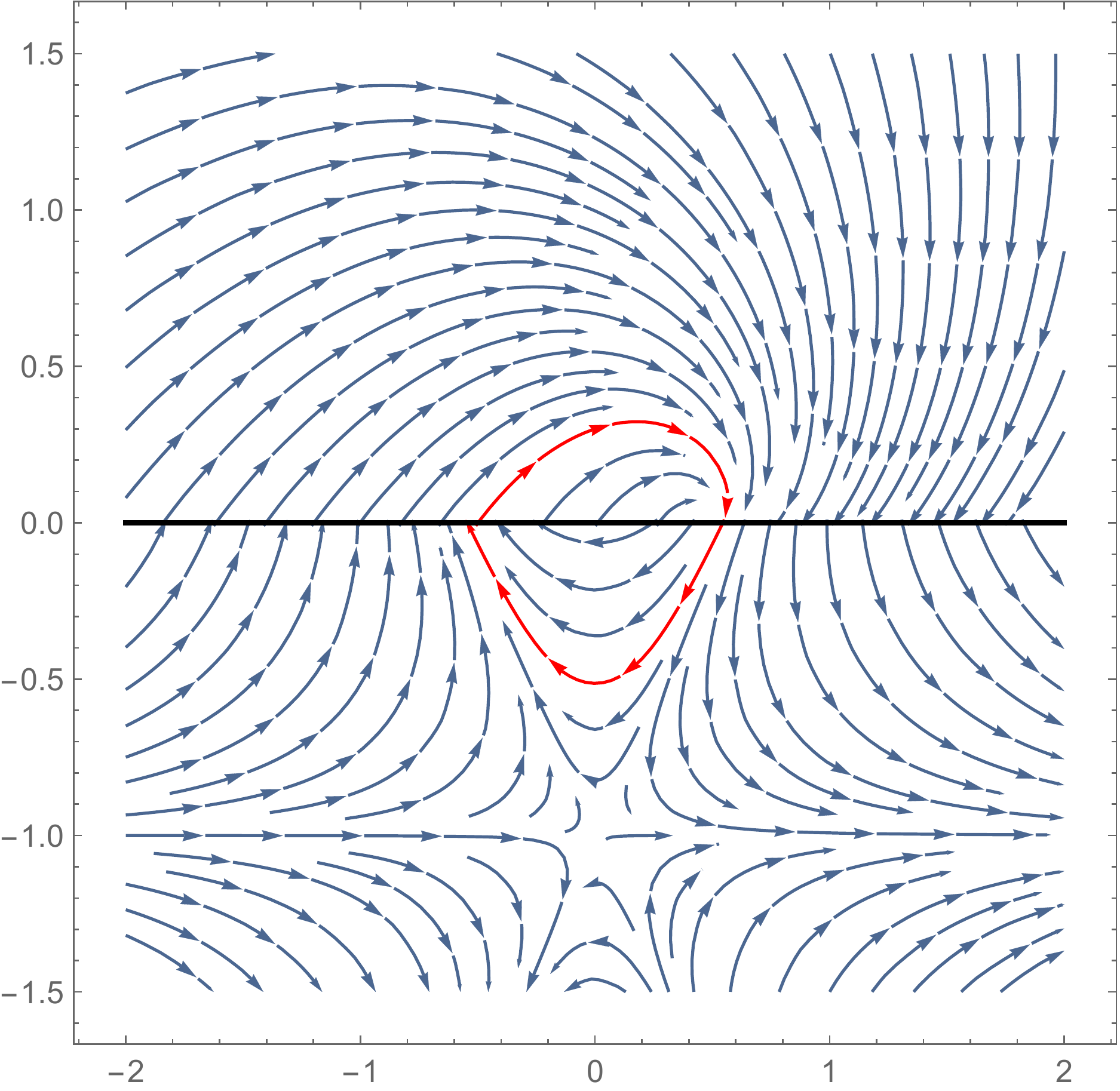}
		%\begin{overpic}[grid,tics=5,width=5cm]{limit_cycle_1.pdf}		
          \put(-7,27){$\Sigma^-$}
        \put(-7,68){$\Sigma^+$}
		\put(102,51){$\Sigma$}
		\end{overpic}
		\caption{Phase portrait of PWHS \eqref{ex7_cycle} with $n=2,$ $a=-1,$ $b=-1,$ $d=\frac{1}{2},$ and $y_0=1$. The red  trajectory is the limit cycle of \eqref{ex7_cycle}.}
	\label{limit_cycle_7}
	\end{center}
	\end{figure}
		\begin{proposition}
	Let $a,b,d,$ and $y_0$ be non-zero real numbers. If $a,d<0,$ $b,y_0>0,$ $n=4k$ for some integer $k\geq 1,$ and $\cot\left(\frac{(n+2)\pi}{2(n+1)}\right)y_0<-\frac{d(1+e^\frac{a\pi}{b})}{-1+e^{\frac{a\pi}{b}}}<0,$ then the PWHS
	\begin{equation}\label{ex8_cycle}
\begin{aligned}
\left\{\begin{array}{l}
\dot{z}^{+}=(a+ib)(z-d), \text{ when } \Im{(z)}> 0, \\[5pt]
\dot{z}^{-}=\frac{1}{(z+iy_0)^n},\text{ when } \Im{(z)}<0,
\end{array} \right.
\end{aligned}
\end{equation}
	has a unique stable limit cycle.
\end{proposition}
\begin{proof}
Consider $\dot{z}^{-}=\frac{1}{(z+iy_0)^n}$. Writing $\dot{z}^+$ in its polar form we have
\begin{equation}\label{pczn8}
\left\{\begin{array}{rcl}
\dot{r}&=&r^{-n}\cos(n+1)\theta,\\
\dot{\theta}&=&-r^{-n-1}\sin(n+1)\theta,
\end{array}\right.
\end{equation}
where $z+y_0i=re^{i\theta}=r(\cos(\theta)+i\sin(\theta)).$ By the proof of Proposition \ref{prop_polar3}, we know that the solutions of $z^+$ are symmetric about the $y-$axis.
%Consider $\dot{z}^{-}=\frac{1}{(z+iy_0)^n}$. First, we shall prove that the solutions of $z^-$ are symmetric about the $y-$axis. Indeed, writing $\dot{z}^-$ in its polar form we have
%%$$z+y_0i=re^{i\theta}=r\cos(\theta)+i\sin(\theta),$$
%%thus
%\begin{equation}\label{pczn8}
%\left\{\begin{array}{rcl}
%\dot{r}&=&r^{-n}\cos(n+1)\theta,\\
%\dot{\theta}&=&-r^{-n-1}\sin(n+1)\theta,
%\end{array}\right.
%\end{equation}
%where $z+y_0i=re^{i\theta}=r(\cos(\theta)+i\sin(\theta)).$ It is easy to see that the orbits of this system satisfy the following equation:
%\begin{equation} \label{rzn8}  
%r=\frac{e^C}{|\sin(n+1)\theta|^{\frac{1}{n+1}}}. 
%\end{equation}
%Since equation \eqref{rzn8} evaluated in $\pi-\theta$ and $\theta$ are the same, then the orbits of \eqref{pczn8} are symmetric with respect to the straight line $\theta=\frac{\pi}{2}.$ Therefore, we can conclude the symmetry of the solutions of $\dot{z}^-$ with respect to $y-$axis.
%$$\begin{array}{rcl}
%r&=&|\sin(n-1)(\pi-\theta)|^\frac{1}{n-1}e^c\\
%&=&|\sin[(n-1)\pi-(n-1)\theta)]|^\frac{1}{n-1}e^c\\
%&=&|\sin(n-1)\theta|^\frac{1}{n-1}e^c.\\
%\end{array}$$
%Writing the system \eqref{ex2_cycle} in Cartesian coordinates, we have
%\begin{equation}\label{csys_2}
%	\begin{aligned}
%	\left\{\begin{array}{l}
%	(\dot{x}^{+},\dot{y}^{-})=((x-x_0)^2-(y-y_0)^2,2(x-x_0)(y-y_0)),\text{ when } y>0, \\[5pt]
%	(\dot{x}^{-},\dot{y}^{+})=(a(x-d)-by,b(x-d)+ay), \text{ when } y< 0.
%	\end{array} \right.
%	\end{aligned}
%	\end{equation}

Now, consider the solution $z^-(t)$ of \eqref{ex8_cycle} with initial condition $z^-(0)=-w_0>0$. By the symmetry of the solutions of \eqref{pczn8}, we have that there exists $t_0>0$ such that $z^-(t_0)=w_0.$ Moreover, 
\[z^+(t)=-(d-w_0)e^{at}(\cos (bt)+i\sin (bt))+d\]
is a solution of $\dot{z}^{+}=(a+bi)(z-d)$ satisfying that $z^+(0)=w_0$ and $z^{+}(\frac{\pi}{b})=d+(d-w_0)e^{\frac{a\pi}{b}}$.

Therefore, the Poincaré map around $z=-w_0$ is given by $\Pi(z)=d+(d+z)e^{\frac{a\pi}{b}}$ and $\Pi'(-w_{0})=e^{\frac{a\pi}{b}}<1$. Now, we must seek solutions for the equation $\Pi(-w_{0})=-w_{0}$.  Since $a\neq 0,$ then we have a unique solution, given by $w_0=\dfrac{d(1+e^\frac{a\pi}{b})}{-1+e^{\frac{a\pi}{b}}}$, thus we have only one limit cycle (see Remark \ref{unique_cycle_2}), which is stable.
\end{proof}
Notice that the conditions given in the previous proposition are not empty. Indeed, taking $a=-1,$ $b=1,$ $d=-\frac{1}{5},$ and $y_0=1$ we have the existence of a limit cycle (see Figure \ref{limit_cycle_8}).
\begin{figure}[h]
	\begin{center}
		\begin{overpic}[scale=0.29]{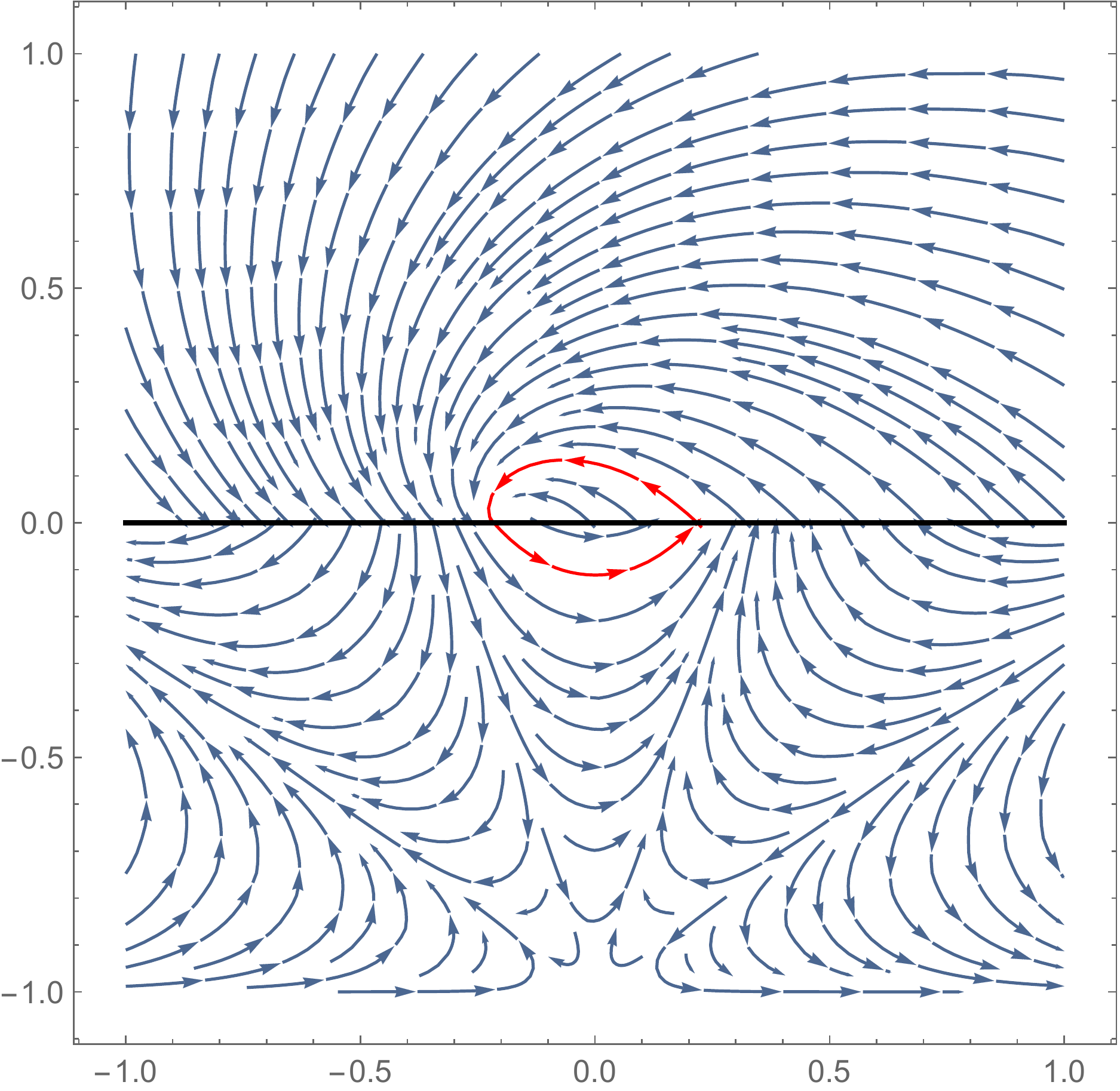}
		%\begin{overpic}[grid,tics=5,width=5cm]{limit_cycle_1.pdf}		
          \put(-7,27){$\Sigma^-$}
        \put(-7,68){$\Sigma^+$}
		\put(102,51){$\Sigma$}
		\end{overpic}
		\caption{Phase portrait of PWHS \eqref{ex8_cycle} with $n=4,$ $a=-1,$ $b=1,$ $d=-\frac{1}{5},$ and $y_0=1$. The red  trajectory is the limit cycle of \eqref{ex8_cycle}.}
	\label{limit_cycle_8}
	\end{center}
	\end{figure}
		\begin{proposition}\label{prop_polar4}
	Let $a,b,d,$ and $y_0$ be non-zero real numbers. If $a,d<0,$ $b,y_0>0,$ $n=4k+1$ for some integer $k\geq 0,$ and $\cot\left(\frac{(n+2)\pi}{2(n+1)}\right)y_0<-\frac{d(1+e^\frac{a\pi}{b})}{-1+e^{\frac{a\pi}{b}}}<0,$ then the PWHS
	\begin{equation}\label{ex9_cycle}
\begin{aligned}
\left\{\begin{array}{l}
\dot{z}^{+}=(a+ib)(z-d), \text{ when } \Im{(z)}> 0, \\[5pt]
\dot{z}^{-}=\frac{i}{(z+iy_0)^n},\text{ when } \Im{(z)}<0,
\end{array} \right.
\end{aligned}
\end{equation}
	has a unique stable limit cycle.
\end{proposition}
\begin{proof}
Consider $\dot{z}^{-}=\frac{i}{(z+iy_0)^n}$. First, we shall prove that the solutions of $z^-$ are symmetric about the $y-$axis. Indeed, writing $\dot{z}^-$ in its polar form we have
%$$z+y_0i=re^{i\theta}=r\cos(\theta)+i\sin(\theta),$$
%thus
\begin{equation}\label{pczn9}
\left\{\begin{array}{rcl}
\dot{r}&=&r^{-n}\sin(n+1)\theta,\\
\dot{\theta}&=&r^{-n-1}\cos(n+1)\theta,
\end{array}\right.
\end{equation}
where $z+y_0i=re^{i\theta}=r(\cos(\theta)+i\sin(\theta)).$ It is easy to see that the orbits of this system satisfy the following equation:
\begin{equation} \label{rzn9}  
r=\frac{e^C}{|\cos(n+1)\theta|^{\frac{1}{n+1}}}. 
\end{equation}
Since equation \eqref{rzn9} evaluated in $\pi-\theta$ and $\theta$ are the same, then the orbits of \eqref{pczn9} are symmetric with respect to the straight line $\theta=\frac{\pi}{2}.$ Therefore, we can conclude the symmetry of the solutions of $\dot{z}^-$ with respect to $y-$axis.
%$$\begin{array}{rcl}
%r&=&|\sin(n-1)(\pi-\theta)|^\frac{1}{n-1}e^c\\
%&=&|\sin[(n-1)\pi-(n-1)\theta)]|^\frac{1}{n-1}e^c\\
%&=&|\sin(n-1)\theta|^\frac{1}{n-1}e^c.\\
%\end{array}$$
%Writing the system \eqref{ex2_cycle} in Cartesian coordinates, we have
%\begin{equation}\label{csys_2}
%	\begin{aligned}
%	\left\{\begin{array}{l}
%	(\dot{x}^{+},\dot{y}^{-})=((x-x_0)^2-(y-y_0)^2,2(x-x_0)(y-y_0)),\text{ when } y>0, \\[5pt]
%	(\dot{x}^{-},\dot{y}^{+})=(a(x-d)-by,b(x-d)+ay), \text{ when } y< 0.
%	\end{array} \right.
%	\end{aligned}
%	\end{equation}

Now, consider the solution $z^-(t)$ of \eqref{ex9_cycle} with initial condition $z^-(0)=-w_0>0$. By the symmetry of the solutions of \eqref{pczn9}, we have that there exists $t_0>0$ such that $z^-(t_0)=w_0.$ Moreover, 
\[z^+(t)=-(d-w_0)e^{at}(\cos (bt)+i\sin (bt))+d\]
is a solution of $\dot{z}^{+}=(a+bi)(z-d)$ satisfying that $z^+(0)=w_0$ and $z^{+}(\frac{\pi}{b})=d+(d-w_0)e^{\frac{a\pi}{b}}$.

Therefore, the Poincaré map around $z=-w_0$ is given by $\Pi(z)=d+(d+z)e^{\frac{a\pi}{b}}$ and $\Pi'(-w_{0})=e^{\frac{a\pi}{b}}<1$. Now, we must seek solutions for the equation $\Pi(-w_{0})=-w_{0}$.  Since $a\neq 0,$ then we have a unique solution, given by $w_0=\dfrac{d(1+e^\frac{a\pi}{b})}{-1+e^{\frac{a\pi}{b}}}$, thus we have only one limit cycle (see Remark \ref{unique_cycle_2}), which is stable.
\end{proof}
Recall that the conditions given in the previous proposition are not empty. Indeed, taking $a=-1,$ $b=1,$ $d=-\frac{1}{2},$ and $y_0=1$ we have the existence of a limit cycle (see Figure \ref{limit_cycle_9}).
\begin{figure}[h]
	\begin{center}
		\begin{overpic}[scale=0.29]{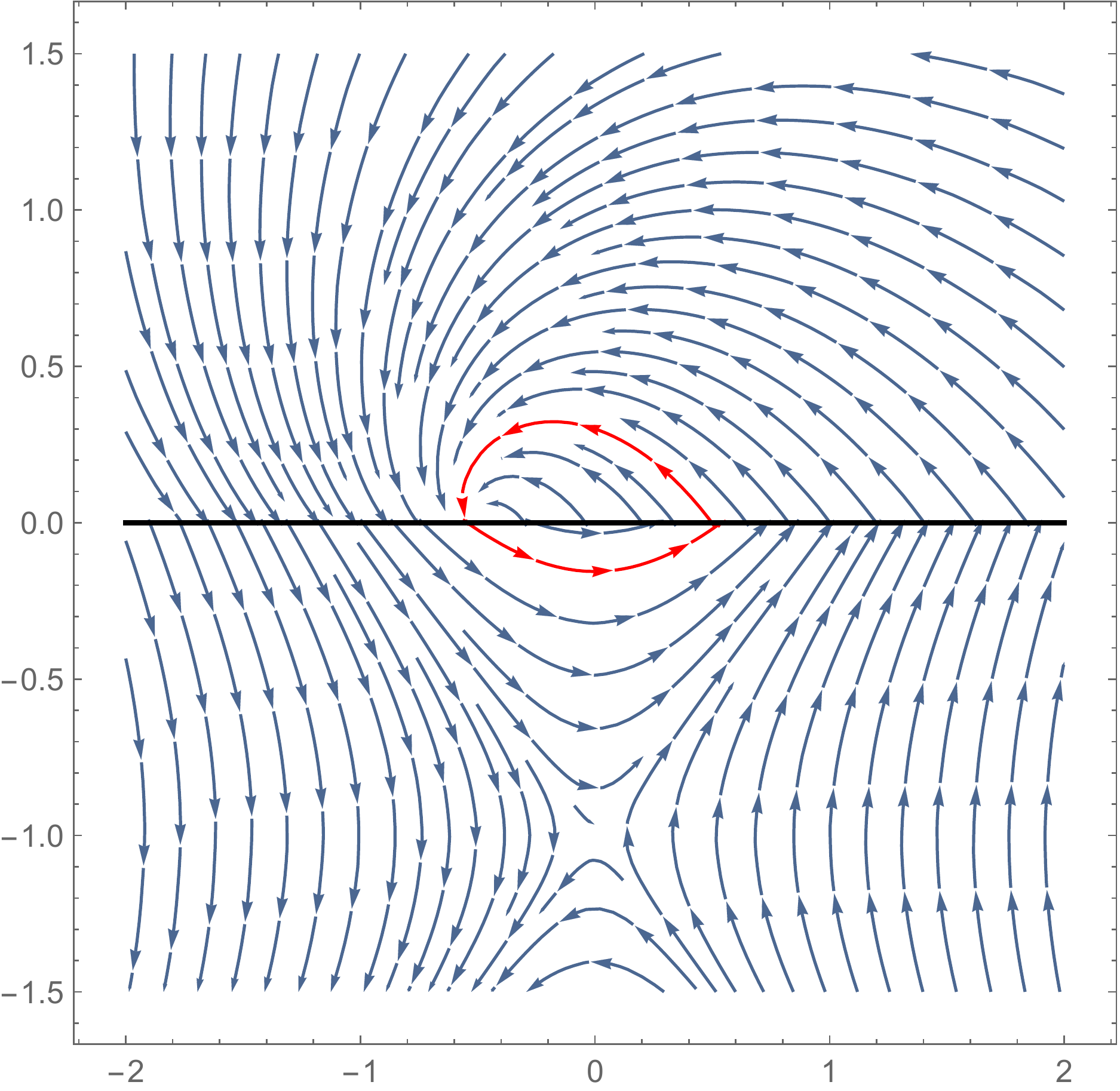}
		%\begin{overpic}[grid,tics=5,width=5cm]{limit_cycle_1.pdf}		
          \put(-7,27){$\Sigma^-$}
        \put(-7,68){$\Sigma^+$}
		\put(102,51){$\Sigma$}
		\end{overpic}
		\caption{Phase portrait of PWHS \eqref{ex9_cycle} with $n=1,$ $a=-1,$ $b=1,$ $d=-\frac{1}{2},$ and $y_0=1$. The red  trajectory is the limit cycle of \eqref{ex9_cycle}.}
	\label{limit_cycle_9}
	\end{center}
	\end{figure}
	\begin{proposition}
	Let $a,b,d,$ and $y_0$ be non-zero real numbers. If $a,b<0,$ $d,y_0>0,$ $n=4k-1$ for some integer $k\geq 1,$ and $0<\frac{d(1+e^\frac{a\pi}{b})}{-1+e^{\frac{a\pi}{b}}}<\cot\left(\frac{n\pi}{2(n+1)}\right)y_0,$ then the PWHS
	\begin{equation}\label{ex10_cycle}
\begin{aligned}
\left\{\begin{array}{l}
\dot{z}^{+}=(a+ib)(z-d), \text{ when } \Im{(z)}> 0, \\[5pt]
\dot{z}^{-}=\frac{i}{(z+iy_0)^n},\text{ when } \Im{(z)}<0,
\end{array} \right.
\end{aligned}
\end{equation}
	has a unique stable limit cycle.
\end{proposition}
\begin{proof}
Consider $\dot{z}^{-}=\frac{i}{(z+iy_0)^n}$. Writing $\dot{z}^+$ in its polar form we have
\begin{equation}\label{pczn10}
\left\{\begin{array}{rcl}
\dot{r}&=&r^{-n}\sin(n+1)\theta,\\
\dot{\theta}&=&r^{-n-1}\cos(n+1)\theta,
\end{array}\right.
\end{equation}
where $z+y_0i=re^{i\theta}=r(\cos(\theta)+i\sin(\theta)).$ By the proof of Proposition \ref{prop_polar4}, we know that the solutions of $z^+$ are symmetric about the $y-$axis.
%Consider $\dot{z}^{-}=\frac{1}{(z+iy_0)^n}$. First, we shall prove that the solutions of $z^-$ are symmetric about the $y-$axis. Indeed, writing $\dot{z}^-$ in its polar form we have
%%$$z+y_0i=re^{i\theta}=r\cos(\theta)+i\sin(\theta),$$
%%thus
%\begin{equation}\label{pczn10}
%\left\{\begin{array}{rcl}
%\dot{r}&=&r^{-n}\sin(n+1)\theta,\\
%\dot{\theta}&=&r^{-n-1}\cos(n+1)\theta,
%\end{array}\right.
%\end{equation}
%where $z+y_0i=re^{i\theta}=r(\cos(\theta)+i\sin(\theta)).$ It is easy to see that the orbits of this system satisfy the following equation:
%\begin{equation} \label{rzn10}  
%r=\frac{e^C}{|\cos(n+1)\theta|^{\frac{1}{n+1}}}. 
%\end{equation}
%Since equation \eqref{rzn10} evaluated in $\pi-\theta$ and $\theta$ are the same, then the orbits of \eqref{pczn10} are symmetric with respect to the straight line $\theta=\frac{\pi}{2}.$ Therefore, we can conclude the symmetry of the solutions of $\dot{z}^-$ with respect to $y-$axis.
%$$\begin{array}{rcl}
%r&=&|\sin(n-1)(\pi-\theta)|^\frac{1}{n-1}e^c\\
%&=&|\sin[(n-1)\pi-(n-1)\theta)]|^\frac{1}{n-1}e^c\\
%&=&|\sin(n-1)\theta|^\frac{1}{n-1}e^c.\\
%\end{array}$$
%Writing the system \eqref{ex2_cycle} in Cartesian coordinates, we have
%\begin{equation}\label{csys_2}
%	\begin{aligned}
%	\left\{\begin{array}{l}
%	(\dot{x}^{+},\dot{y}^{-})=((x-x_0)^2-(y-y_0)^2,2(x-x_0)(y-y_0)),\text{ when } y>0, \\[5pt]
%	(\dot{x}^{-},\dot{y}^{+})=(a(x-d)-by,b(x-d)+ay), \text{ when } y< 0.
%	\end{array} \right.
%	\end{aligned}
%	\end{equation}

Now, consider the solution $z^-(t)$ of \eqref{ex10_cycle} with initial condition $z^-(0)=w_0>0$. By the symmetry of the solutions of \eqref{pczn10}, we have that there exists $t_0>0$ such that $z^-(t_0)=-w_0.$ Moreover, 
\[z^+(t)=-(d+w_0)e^{at}(\cos (bt)+i\sin (bt))+d\]
is a solution of $\dot{z}^{+}=(a+bi)(z-d)$ satisfying that $z^+(0)=-w_0$ and $z^{+}(-\frac{\pi}{b})=d+(d+w_0)e^{-\frac{a\pi}{b}}$.

Consequently, the Poincaré map at $z=w_0$ is given by $\Pi(w_{0})=d+(d+w_0)e^{-\frac{a\pi}{b}}$ and $\Pi'(w_{0})=e^{\frac{-a\pi}{b}}<1$. Now, we must seek solutions for the equation $\Pi(w_{0})=w_{0}$.  Since $a\neq 0,$ then we have a unique solution, given by $\dfrac{d(1+e^\frac{a\pi}{b})}{-1+e^{\frac{a\pi}{b}}}$, thus we have only one limit cycle (see Remark \ref{unique_cycle_2}), which is stable.
\end{proof}
Emphasize that the conditions given in the previous proposition are not empty. Indeed, taking $a=-1,$ $b=-1,$ $d=\frac{3}{10},$ and $y_0=1$ we have the existence of a limit cycle (see Figure \ref{limit_cycle_10}).
\begin{figure}[h]
	\begin{center}
		\begin{overpic}[scale=0.29]{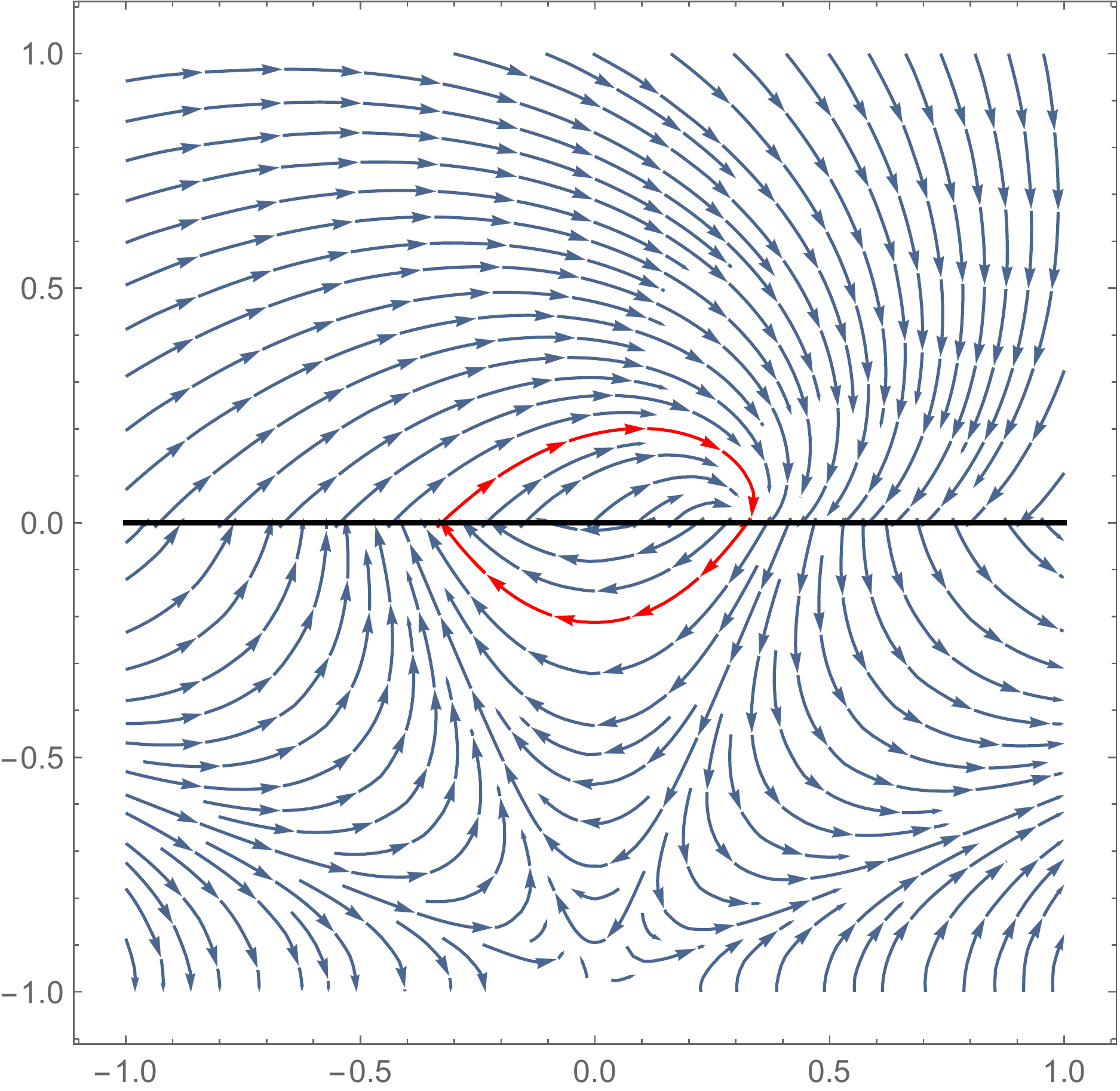}
		%\begin{overpic}[grid,tics=5,width=5cm]{limit_cycle_1.pdf}		
          \put(-7,27){$\Sigma^-$}
        \put(-7,68){$\Sigma^+$}
		\put(102,51){$\Sigma$}
		\end{overpic}
		\caption{Phase portrait of PWHS \eqref{ex10_cycle} with $n=3,$ $a=-1,$ $b=-1,$ $d=\frac{3}{10},$ and $y_0=1$. The red  trajectory is the limit cycle of \eqref{ex10_cycle}.}
	\label{limit_cycle_10}
	\end{center}
	\end{figure}
	
	To end this section, we give an example of a limit cycle of PWHS using the normal form $\frac{\gamma z^n}{1+z^{n-1}},$ for $n=2$ and $\gamma=1.$
\begin{example}
	The PWHS
	\begin{equation}\label{ex11_cycle}
\begin{aligned}
\left\{\begin{array}{l}
\dot{z}^{+}=\dfrac{(z+\frac{i}{5})^2}{1+(z+\frac{i}{5})},\text{ when } \Im{(z)}>0, \\[5pt]
\dot{z}^{-}=i(z+0.0381415), \text{ when } \Im{(z)}< 0,
\end{array} \right.
\end{aligned}
\end{equation}
	has at least an unstable limit cycle. Indeed, writing $\dot{z}^+$ in its polar form we have
\begin{equation}\label{pczn11}
\left\{\begin{array}{rcl}
\dot{r}&=&\frac{r^2(r+\cos(\theta))}{1+r^2+2r\cos(\theta)},\\
\dot{\theta}&=&\frac{r\sin(\theta)}{1+r^2+2r\cos(\theta)},
\end{array}\right.
\end{equation}
where $z+\frac{i}{5}=re^{i\theta}=r(\cos(\theta)+i\sin(\theta)).$ It is easy to see that the orbits of system \eqref{pczn11} satisfy the equation $r=\frac{-\sin(\theta)}{\theta-c_1},$
%\begin{equation} \label{rzn11}  
%r=\frac{-\sin(\theta)}{\theta-c_1}. 
%\end{equation}
where $c_1=\arctan\left(\frac{\frac{1}{5}}{w_0}\right)+\frac{\frac{1}{5}}{w_0^2+(\frac{1}{5})^2},$ with $z^+(0)=w_0.$
Thus, the solutions of system \eqref{pczn11} can be parametrized by
$$z^+(\theta)=\frac{-\sin(\theta)\cos(\theta)}{\theta-c_1}-i\left(\frac{\sin^2(\theta)}{\theta-c_1}+\frac{1}{5}\right).$$
 Notice that if $w_1=\frac{1}{20}$ and $w_2=\frac{13}{100}$, then $z^+(w_1)\approx -0.100229$ and $z^+(w_2)\approx -0.238348.$

 Now,
 \[z_1^-(t)=-0.062087(\cos (t)+i\sin (t))-0.0381415\]\quad and  
 \[z_2^-(t)=-0.200207(\cos (t)+i\sin (t))-0.0381415\]
are the solutions of $\dot{z}^{-}=i(z+0.0381415)$ satisfying that $z_1^-(0)=-0.100229$ and $z_2^{-}(0)=-0.238348$, respectively. Thus, $z_1^-(\pi)\approx 0.0239455$ and $z_2^-(\pi)\approx 0.162065.$  Therefore, $\Delta=\pi(w_1)-w_1>0$ and $\Delta=\pi(w_2)-w_2<0.$ Then, by continuity there exists $w_0\in(w_1,w_2),$ such that $\pi(\widetilde{w_0})=\widetilde{w_0}.$ %It is easy to see that $\widetilde{w_0}=\frac{1}{10}.$
%Now, consider the solution $z^+(t)$ of \eqref{rzn11} with initial condition $z^+(1.10715)=\frac{1}{10}$. Hence, $c_1=\arctan(2)+4$ and $z^+(2.29325)=-0.17628.$ Moreover, 
%\[z^-(t)=-0.138142(\cos (t)+i\sin (t))-0.0381415\]
%is a solution of $\dot{z}^{-}=i(z+0.0381415)$ satisfying that $z^-(0)=-0.17628$ and $z^{-}(\pi)=1/10$.

Consequently, system \eqref{ex11_cycle} has a periodic orbit. Even more, numerical approximations indicate that it is an unstable limit cycle(see Figure \ref{limit_cycle_11}).
\end{example}

\begin{figure}[h]
	\begin{center}
		\begin{overpic}[scale=0.35]{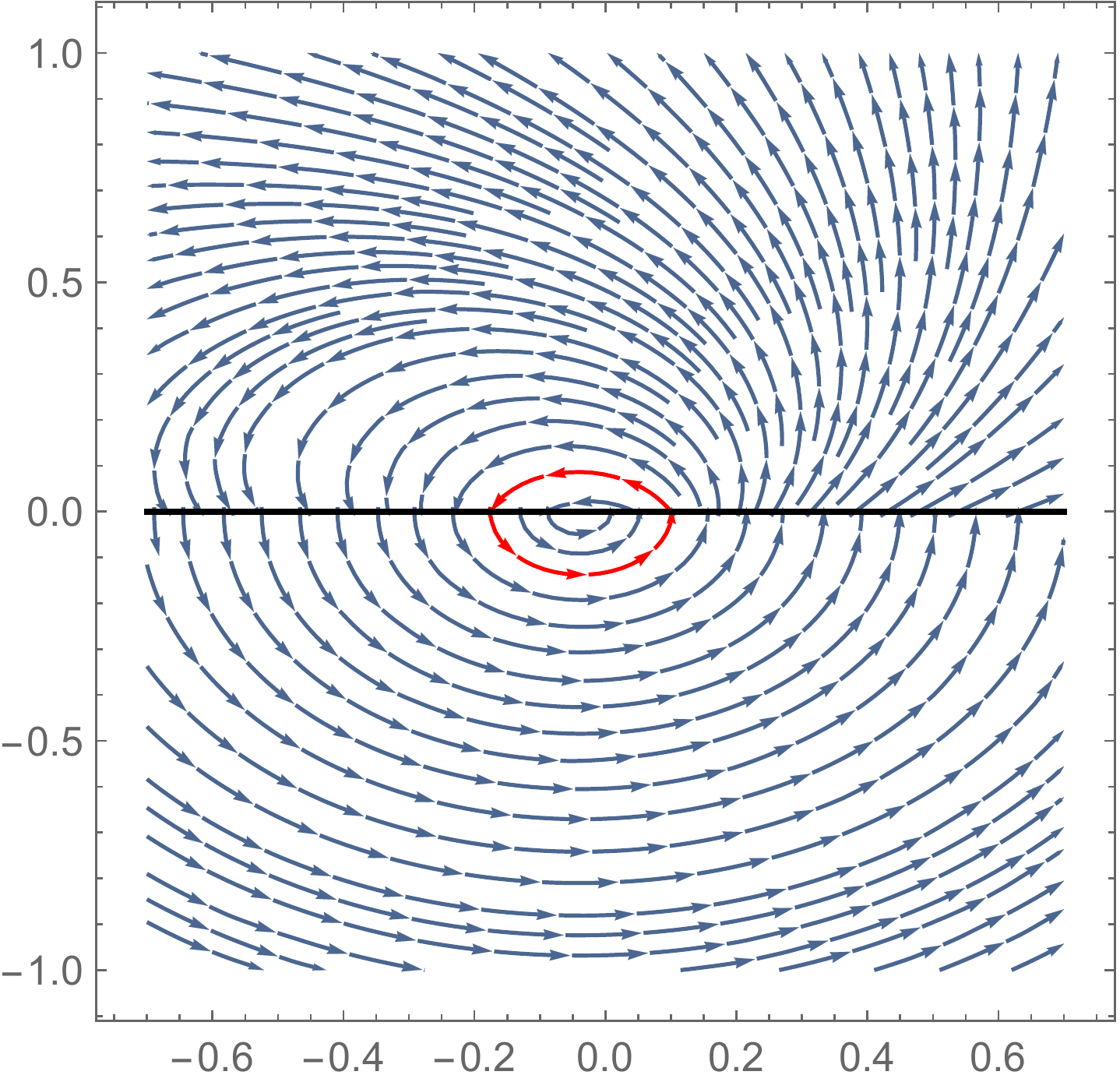}
		%\begin{overpic}[grid,tics=5,width=5cm]{limit_cycle_1.pdf}		
          \put(-7,27){$\Sigma^-$}
        \put(-7,68){$\Sigma^+$}
		\put(102,51){$\Sigma$}
		\end{overpic}
		\caption{Phase portrait of PWHS \eqref{ex11_cycle}. The red trajectory is the limit cycle of \eqref{ex11_cycle}.}
	\label{limit_cycle_11}
	\end{center}
	\end{figure}
	\section{Homoclinic Orbits of PWHS}\label{sec:homoclinic_orbits}
	This section is devoted to give some families of PWHS that have homoclinic orbits. For that, notice that it is possible to form homoclinic orbits in PWHS considering $\dot{z}^-=(z-z_0)^n$ or $\dot{z}^-=\frac{1}{(z-z_0)^n}$ and $\dot{z}^+=biz$. For this we use the invariant rays of $\dot{z}^-$ and, depending on the case, we consider $z_0$ as a real singularity or real equilibrium point of $\dot{z}^-$. 
	\begin{proposition}\label{propho1}
	Given $n\in\N_{n\geq 1}$, $b,$ and $y_0$ be non-zero real numbers with $y_0>0$ and $b$ satisfies the table \eqref{hol_table1}. Then the PWHS
	\begin{equation}\label{ex_hom_orb_1}
\begin{aligned}
\left\{\begin{array}{l}
\dot{z}^{+}=ibz, \text{ when } \Im{(z)}> 0,\\[5pt]
\dot{z}^{-}=\frac{i^m}{(z+iy_0)^n},\text{ when } \Im{(z)}<0, 
\end{array} \right.
\end{aligned}
\end{equation}
	has at least one homoclinic orbit, where $m=0$ if $n$ is even and $m=1$ otherwise.
	\begin{equation}\label{hol_table1}
\begin{array}{|| c |c| c | c | c | c | c |c||}
\hline
n&	b \\
\hline\hline
\hline
4k-1&+ \\
\hline
4k&-\\
\hline
4k-2&+\\
\hline
4k+1&-\\
\hline
\end{array}
\end{equation}
\end{proposition}
\begin{proof}
Without loss of generality assume $b>0.$ Consider the invariant rays $\frac{n\pi}{2(n+1)}$ and $\frac{(n+2)\pi}{2(n+1)}$ associated with $z^-.$ Notice that these rays intersect $\Sigma$ at points $x=\cot\left(\frac{n\pi}{2(n+1)}\right)y_0$ and $x=\cot\left(\frac{(n+2)\pi}{2(n+1)}\right)y_0,$ respectively. 

Now, 
\[z^-(t)=\cot\left(\frac{n\pi}{2(n+1)}\right)y_0\left(\cos (bt)+i\sin (bt)\right)\]
is a solution of $\dot{z}^{-}=biz$ satisfying that $z^-(0)=\cot\left(\frac{n\pi}{2(n+1)}\right)y_0$ and $z^{-}(\frac{\pi}{b})=\cot\left(\frac{(n+2)\pi}{2(n+1)}\right)y_0$. Thus, we get a homoclinic orbit of \eqref{ex_hom_orb_1} (see Figure \ref{hol_orb_1}). 
\end{proof}
\begin{figure}[h]
	\begin{center}
		\begin{overpic}[scale=0.35]{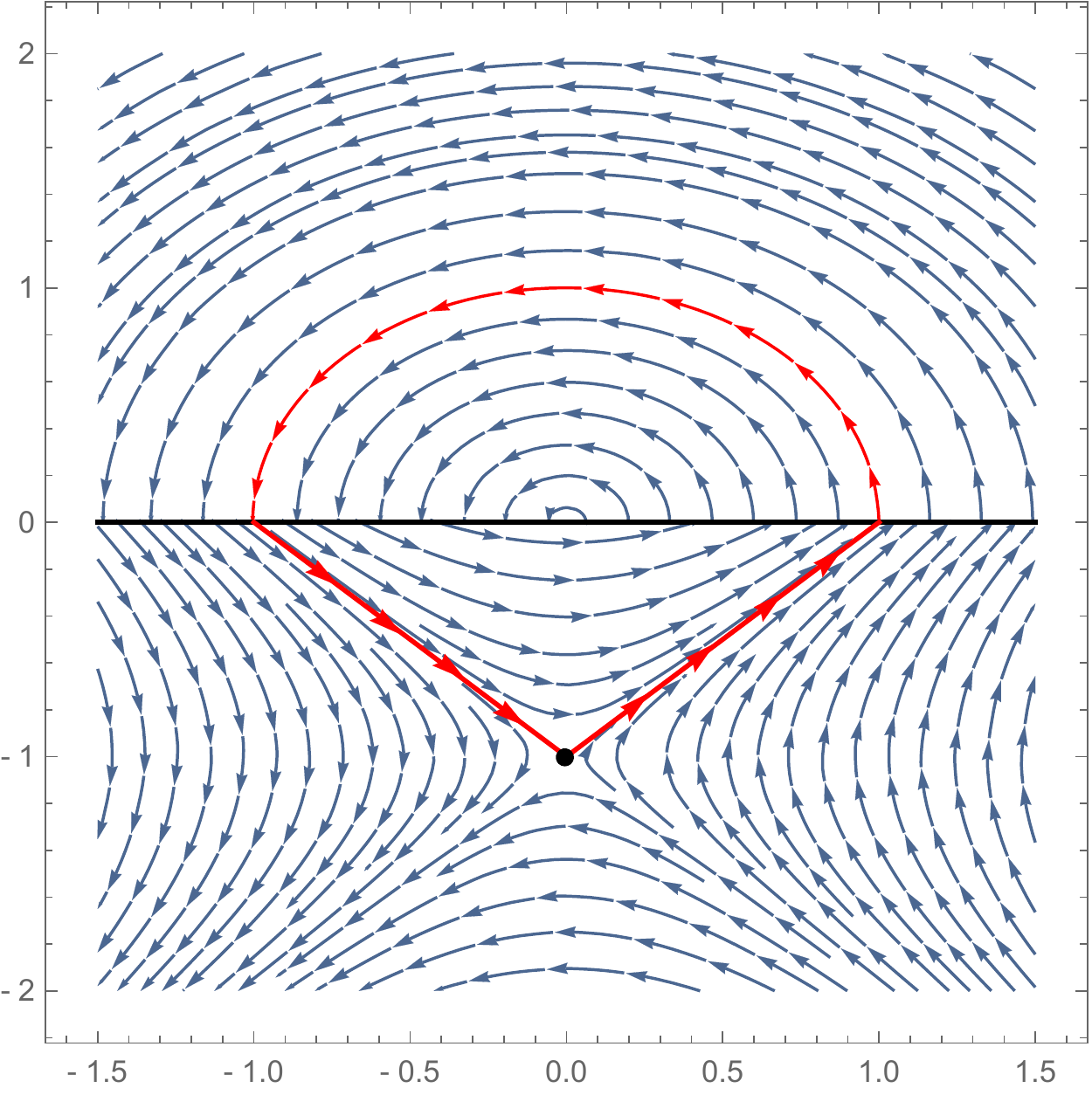}
		%\begin{overpic}[grid,tics=5,width=5cm]{limit_cycle_1.pdf}		
          \put(-9,27){$\Sigma^-$}
        \put(-9,68){$\Sigma^+$}
		\put(102,51){$\Sigma$}
		\end{overpic}
		\caption{Phase portrait of PWHS \eqref{ex_hom_orb_1}, with $n=1,$ $b=1,$ and $y_0=1$. The red trajectory is a homoclinic orbit of \eqref{ex_hom_orb_1}.}
	\label{hol_orb_1}
	\end{center}
	\end{figure}	
\begin{proposition}\label{propho2}
	Given $n\in\N_{n>2}$, $b,$ and $y_0$ be non-zero real numbers with $y_0>0$ and $b$ satisfies the table \eqref{hol_table2}. Then the PWHS
	\begin{equation}\label{ex_hom_orb_2}
\begin{aligned}
\left\{\begin{array}{l}
\dot{z}^{+}=ibz, \text{ when } \Im{(z)}> 0,\\[5pt]
\dot{z}^{-}=i^m(z+iy_0)^n,\text{ when } \Im{(z)}<0, 
\end{array} \right.
\end{aligned}
\end{equation}
	has at least one homoclinic orbit, where $m=0$ if $n$ is even and $m=1$ otherwise.
	\begin{equation}\label{hol_table2}
\begin{array}{|| c |c| c | c | c | c | c |c||}
\hline
n&	b \\
\hline\hline
\hline
4k-1&- \\
\hline
4k&-\\
\hline
4k-2&+\\
\hline
4k+1&+\\
\hline
\end{array}
\end{equation}
\end{proposition}
\begin{proof}
Without loss of generality assume $b>0.$ Consider the invariant rays $\frac{n\pi}{2(n-1)}$ and $\frac{(n-2)\pi}{2(n-1)}$ associated with $z^-.$ Notice that these rays intersect $\Sigma$ at points $x=\cot\left(\frac{n\pi}{2(n-1)}\right)y_0$ and $x=\cot\left(\frac{(n-2)\pi}{2(n-1)}\right)y_0,$ respectively. 

Now, 
\[z^-(t)=\cot\left(\frac{(n-2)\pi}{2(n-1)}\right)y_0\left(\cos (bt)+i\sin (bt)\right)\]
is a solution of $\dot{z}^{-}=biz$ satisfying that $z^-(0)=\cot\left(\frac{(n-2)\pi}{2(n-1)}\right)y_0$ and $z^{-}(\frac{\pi}{b})=\cot\left(\frac{n\pi}{2(n-1)}\right)y_0$. Thus, we get a homoclinic orbit of \eqref{ex_hom_orb_2} (see Figure \ref{hol_orb_2}). 
\end{proof}
\begin{figure}[h]
	\begin{center}
		\begin{overpic}[scale=0.35]{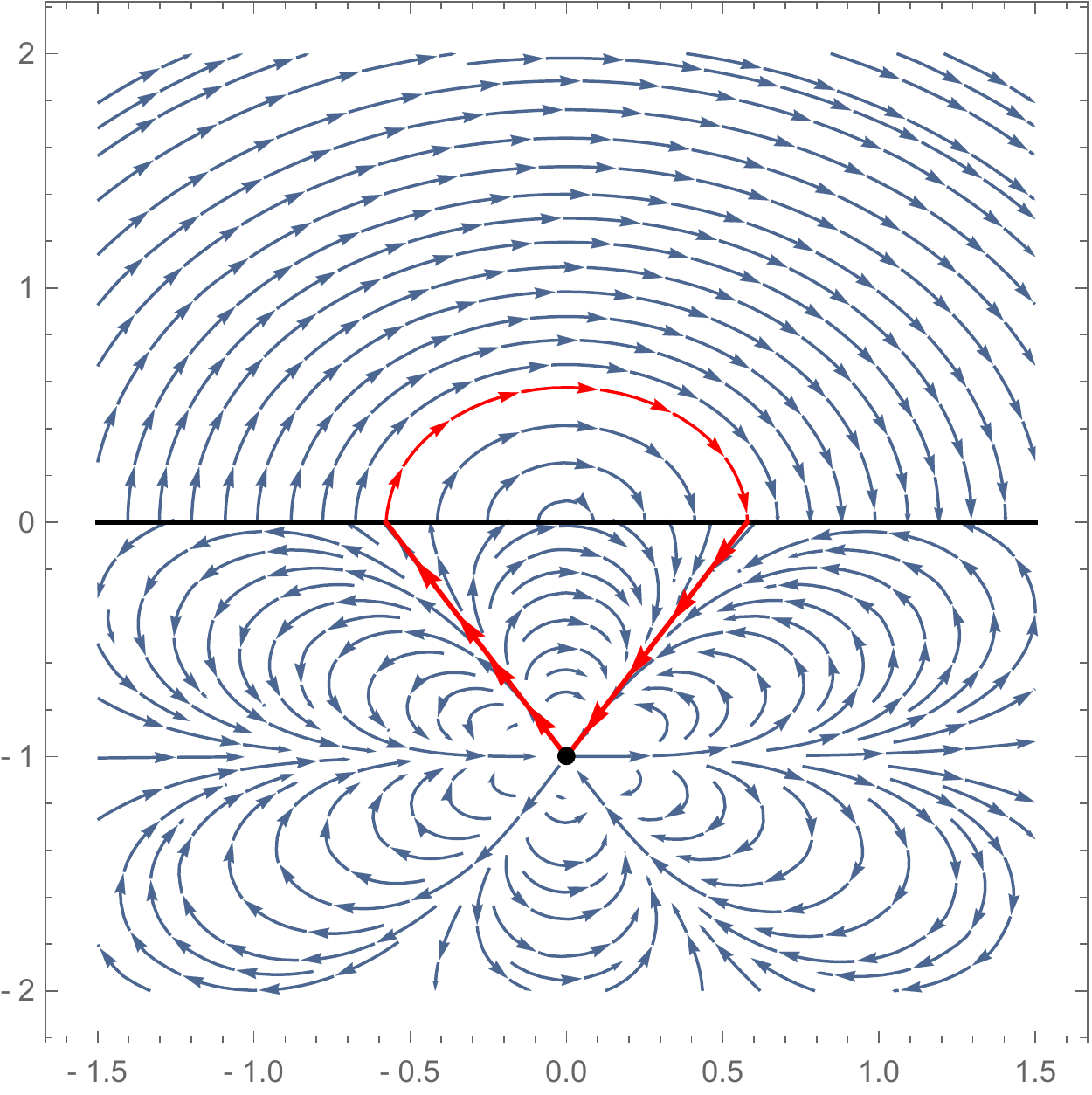}
		%\begin{overpic}[grid,tics=5,width=5cm]{limit_cycle_1.pdf}		
          \put(-9,27){$\Sigma^-$}
        \put(-9,68){$\Sigma^+$}
		\put(102,51){$\Sigma$}
		\end{overpic}
		\caption{Phase portrait of PWHS \eqref{ex_hom_orb_2}, with $n=4,$ $b=-1,$ and $y_0=1$. The red trajectory is a homoclinic orbit of \eqref{ex_hom_orb_2}.}
	\label{hol_orb_2}
	\end{center}
	\end{figure}	
	\section{Acknowledgments}
	%Paulo Ricardo da Silva is partially supported by CAPES  and FAPESP.
This article was possible thanks to the scholarship granted from the Brazilian
Federal Agency for Support and Evaluation of Graduate Education (CAPES), in
the scope of the Program CAPES-Print, process number 88887.310463/2018-00,
International Cooperation Project number 88881.310741/2018-01. Paulo Ricardo da Silva is
also partially supported by São Paulo Research Foundation (FAPESP) grant 2019/10269-3.

Luiz Fernando Gouveia is supported by São Paulo Research Foundation (FAPESP) grant 2020/04717-0. Gabriel Rondón is supported by São Paulo Research Foundation (FAPESP) grant 2020/06708-9.

\bibliographystyle{abbrv}
\bibliography{references1}

\end{document}